\documentclass[11pt,a4paper,reqno]{amsart}

\usepackage{amssymb,amsmath}
\usepackage{amsthm}
\usepackage{latexsym}
\usepackage[colorlinks=true,linktocpage=true,pagebackref=false, citecolor=black,linkcolor=black]{hyperref}
\usepackage{graphics}
\usepackage{euscript}
\usepackage{xcolor}
\usepackage[all]{xy}
\usepackage{setspace}

\newtheorem{thm}{Theorem}[section]
\newtheorem{fact}[thm]{Fact}

\newtheorem{lem}[thm]{Lemma}
\newtheorem{prop}[thm]{Proposition}

\newtheorem{cor}[thm]{Corollary}

\theoremstyle{definition}

\newtheorem{defn}[thm]{Definition}
\newtheorem{defns}[thm]{Definitions}

\theoremstyle{remark}
\newtheorem{remark}[thm]{Remark}
\newtheorem{remarks}[thm]{Remarks}
\newtheorem{example}[thm]{Example}
\newtheorem{examples}[thm]{Examples}

\numberwithin{equation}{section}

\makeatletter
\def\@seccntformat#1{%
 \protect\textup{\protect\@secnumfont
 \ifnum\pdfstrcmp{subsection}{#1}=0 \bfseries\fi% subsection # in \bfseries
 \csname the#1\endcsname
 \protect\@secnumpunct
 }%
} 
\makeatother

 \newcommand{\N}{{\mathbb N}}
\newcommand{\Z}{{\mathbb Z}} \newcommand{\R}{{\mathbb R}}

\newcommand{\gtp}{{\mathfrak p}} \newcommand{\gtq}{{\mathfrak q}}
\newcommand{\gtm}{{\mathfrak m}} \newcommand{\gtn}{{\mathfrak n}}
\newcommand{\gta}{{\mathfrak a}} \newcommand{\gtb}{{\mathfrak b}}
 \newcommand{\gtQ}{{\mathfrak Q}}

\newcommand{\Dd}{{\EuScript D}}

\newcommand{\Uu}{{\EuScript U}}

\newcommand{\Ff}{{\EuScript F}}

\newcommand{\partialder}{d}
\newcommand{\im}{\operatorname{Im}}
\newcommand{\qf}{\operatorname{qf}}

\newcommand{\hgt}{\operatorname{ht}}
\newcommand{\supp}{\operatorname{supp}}
\newcommand{\depth}{\operatorname{{\tt d}}}
\newcommand{\Spec}{\operatorname{Spec}}
\newcommand{\Specmax}{\beta}

\newcommand{\Sper}{\operatorname{Sper}}
\newcommand{\tr}{\operatorname{tr}}

\newcommand{\cl}{\operatorname{Cl}}

\newcommand{\diam}{{\text{\tiny$\displaystyle\diamond$}}}

\newcommand{\x}{{\tt x}} \newcommand{\y}{{\tt y}} 
\newcommand{\z}{{\tt z}} \renewcommand{\t}{{\tt t}}

\newcommand{\veps}{\varepsilon}

\newcommand{\ol }{\overline}

\begin{document}

\title[Rings of differentiable semialgebraic functions]{Rings of differentiable semialgebraic functions}
\author{E. Baro}
\author{Jos\'e F. Fernando}
\author{J.M. Gamboa}
\address{Departamento de \'Algebra, Geometr\'\i a y Topolog\'\i a, Facultad de Ciencias Matem\'aticas, Universidad Complutense de Madrid, 28040 MADRID (SPAIN)}
\email{eliasbaro@pdi.ucm.es, josefer@mat.ucm.es, jmgamboa@mat.ucm.es}

\date{August 16th, 2019}

\thanks{Authors supported by Spanish GRAYAS MTM2014-55565-P, Spanish STRANO MTM2017-82105-P and Grupos UCM 910444}
\begin{abstract} 
In this work we analyze the main properties of the Zariski and maximal spectra of the ring ${\mathcal S}^r(M)$ of differentiable semialgebraic functions of class ${\mathcal C}^r$ on a semialgebraic set $M\subset\R^m$. Denote ${\mathcal S}^0(M)$ the ring of semialgebraic functions on $M$ that admit a continuous extension to an open semialgebraic neighborhood of $M$ in $\cl(M)$. This ring is the real closure of ${\mathcal S}^r(M)$. If $M$ is locally compact, the ring ${\mathcal S}^r(M)$ enjoys a \L ojasiewicz's Nullstellensatz, which becomes a crucial tool. Despite ${\mathcal S}^r(M)$ is not real closed for $r\geq1$, the Zariski and maximal spectra of this ring are homeomorphic to the corresponding ones of the real closed ring ${\mathcal S}^0(M)$. In addition, the quotients of ${\mathcal S}^r(M)$ by its prime ideals have real closed fields of fractions, so the ring ${\mathcal S}^r(M)$ is close to be real closed. The missing property is that the sum of two radical ideals needs not to be a radical ideal. The homeomorphism between the spectra of ${\mathcal S}^r(M)$ and ${\mathcal S}^0(M)$ guarantee that all the properties of these rings that arise from spectra are the same for both rings. For instance, the ring ${\mathcal S}^r(M)$ is a Gelfand ring and its Krull dimension is equal to $\dim(M)$. We also show similar properties for the ring ${\mathcal S}^{r*}(M)$ of differentiable bounded semialgebraic functions. In addition, we confront the ring ${\mathcal S}^{\infty}(M)$ of differentiable semialgebraic functions of class ${\mathcal C}^{\infty}$ with the ring ${\mathcal N}(M)$ of Nash functions on $M$.
\end{abstract}

\maketitle
\setcounter{tocdepth}{2}
{\small
\begin{spacing}{0.01}
\tableofcontents
\end{spacing}}

\section{Introduction}\label{s1}

Recall that a \em semialgebraic set \em $M\subset\R^m$ is a set that can be described as a finite boolean combination of polynomial equalities and inequalities. A map $f:M\to N$ between semialgebraic sets $M\subset\R^m$ and $N\subset\R^n$ is \em semialgebraic \em if its graph is a semialgebraic subset of $\R^m\times\R^n$. In order to lighten notation we call \em semialgebraic \em the functions $f:M\to\R$ that are continuous and have semialgebraic graph.

Let $r\geq1$ be a positive integer. An initial problem when dealing with differentiable semialgebraic functions of class ${\mathcal C}^r$ on a semialgebraic set $M\subset\R^m$ is to find an intrinsic definition of such type of functions. If $M\subset\R^m$ is in addition an open subset, we say that $f:M\to\R$ is an \em ${\mathcal S}^r$-function \em if it is a differentiable function of class ${\mathcal C}^r$ with semialgebraic graph. In \cite{kp2,at,th} the authors made a careful analysis of an intrinsic definition of ${\mathcal S}^r$-function in terms of jets of order $r$ of (continuous) semialgebraic functions \cite[Def.1.1]{at}. Their purpose is to achieve a semialgebraic version of Whitney's extension theorem when $M$ is closed in $\R^m$. More precisely, \em if $f:M\to\R$ is an ${\mathcal S}^r$-function on a closed semialgebraic subset of $\R^m$, there exists an ${\mathcal S}^r$-function $F:\R^m\to\R$ such that $F|_M=f$\em, see \cite[Thm.1.2]{at}. This approach follows the same type of ideas developed to prove Whitney's extension theorem \cite[\S1]{m} adapted to the semialgebraic case.

In case $r=1$ the authors go beyond and prove that \em a semialgebraic function $f:M\to\R$ on a closed semialgebraic subset of $\R^m$ is ${\mathcal S}^1$ if and only if for each point $x\in M$ there exists a (non-necessarily semialgebraic) ${\mathcal C}^r$ extension $F_x:W^x\to\R$ to an open semialgebraic neighborhood $W^x\subset\R^m$ of the restriction $f|_{M\cap W^x}$\em, see \cite[Cor.7.14]{at}. Their proof follows the strategy developed in \cite{feff,feff2} adapted to the semialgebraic case. The authors hope that a suitable modification to the semialgebraic case of the very sophisticated techniques developed in \cite{feff} will allow to prove an analogous result for each positive integer $r\geq1$. In fact, if $M\subset\R^n$ is compact and the local ${\mathcal C}^r$ extension $F_x$ of $f$ on an open semialgebraic neighborhood of each point $x\in M$ are in addition semialgebraic, the existence of an ${\mathcal S}^r$-extension of $f$ to $\R^m$ is guaranteed via an appropriate (finite) ${\mathcal S}^r$-partition of unity. 

In this work we adopt the definition of ${\mathcal S}^r$-functions proposed in \cite{at} in terms of jets of order $r$ of (continuous) semialgebraic functions and we denote ${\mathcal S}^r(M)$ the set of ${\mathcal S}^r$-functions on an arbitrary semialgebraic set $M\subset\R^m$. This set is an $\R$-algebra with respect to the usual sum and product of functions (see Definition \ref{def:jets}). We will denote by ${\mathcal S}(M)$ the set of semialgebraic functions on $M$, whereas ${\mathcal S}^0(M)$ is the set of semialgebraic functions on $M$ that can be extended continuously to an open semialgebraic neighborhood of $M$ in $\cl(M)$, or equivalently, to an open semialgebraic neighborhood of $M$ in $\R^m$ (see Lemma \ref{ext}). For $r\geq 0$ we denote ${\mathcal S}^{r*}(M)$ the subring of ${\mathcal S}^r(M)$ of bounded ${\mathcal S}^r$-functions on $M$, whereas ${\mathcal S}^*(M)$ is the subring of ${\mathcal S}(M)$ of bounded ${\mathcal S}$-functions on $M$. In order to ease notation we write ${\mathcal S}^{r\diam}(M)$ to refer indistinctly to both rings ${\mathcal S}^r(M)$ and ${\mathcal S}^{r*}(M)$. We proceed similarly with ${\mathcal S}^\diam(M)$.

In Section \ref{s2} we undertake a development of ${\mathcal S}^r$-functions for $r\geq 0$. One of our first results (Lemma \ref{lips}) is that even though the definition of ${\mathcal S}^r$-function on an \emph{arbitrary} semialgebraic set $M$ is of intrinsic nature, it enclose a Lipschitz condition. Therefore ${\mathcal S}^r$-functions own ${\mathcal S}^{r-1}$ extensions to open semialgebraic neighborhoods of $M$ in $\R^m$. In particular, we have the following chain of inclusions:
$$
{\mathcal S}^{r\diam}(M)\hookrightarrow{\mathcal S}^{0\diam}(M)\hookrightarrow{\mathcal S}^{\diam}(M).
$$
Recall that if $M$ is locally compact, ${\mathcal S}^{0\diam}(M)={\mathcal S}^{\diam}(M)$. In this paper we are mainly concerned in \emph{arbitrary} semialgebraic sets, so it is natural to ask in which situations the equality ${\mathcal S}^{0\diam}(M)={\mathcal S}^{\diam}(M)$ holds. To that aim, we introduce the following definitions. Let $M\subset\R^m$ be a semialgebraic set. We say that a semialgebraic set $M\subset\R^m$ is \em problematic at $x\in M$ \em if there exists a sequence of points $\{x_k\}_k\subset\cl(M)$ converging to $x$ such that each germ $M_{x_k}$ is disconnected. If $M$ is locally compact at $x\in M$, then there exists an open ball $B$ centered in $x$ such that $M\cap B=\cl(M) \cap B$, so $M$ is not problematic at $x$. In particular, locally compact semialgebraic sets are \em non-problematic\em. 

The set of points of $M$ of dimension $k$ is a semialgebraic subset of $M$. Denote $M_1,\ldots,M_s$ the closures in $M$ of those ones that are non-empty and order them in such a way that $\dim(M_i)>\dim(M_{i+1})$ for $i=1,\ldots,s-1$. Each semialgebraic set $M_i$ is pure dimensional, $M=\bigcup_{i=1}^sM_i$ and they are univocally determined by $M$. It holds that $M$ is locally compact at $x\in M$ if and only if each $M_i$ is locally compact at $x$. Thus, $M$ is locally compact if and only if each $M_i$ is locally compact.

\begin{thm}\label{main2}
Let $M\subset\R^m$ be a semialgebraic set. The following assertions are equivalent:
\begin{itemize}
\item[(i)] ${\mathcal S}^{0\diam}(M)={\mathcal S}^{\diam}(M)$. 
\item[(ii)] The map $\varphi:\Spec^\diam(M)\to\Spec^{0\diam}(M),\ \gtp\mapsto\gtp\cap{\mathcal S}^{0\diam}(M)$ is injective.
\item[(iii)] $M$ is either locally compact or the set ${\mathfrak F}$ of the indexes $i=1,\ldots,s$ such that $M_i$ is non-locally compact is a singleton $\{i_0\}$, $\dim(M_{i_0})=2$ and $M_{i_0}$ is non-problematic.
\end{itemize} 
In particular, if $\varphi$ is injective then it is surjective.
\end{thm}

Another key result from Section \ref{s2} is: \em If $M\subset\R^m$ is a semialgebraic set, then
\begin{equation}\label{dlim}
{\mathcal S}^r(M)\cong\displaystyle\lim_{\longrightarrow}({\mathcal S}^r(E),{\tt j})
\end{equation}
where $E\subset\R^m$ is a closed semialgebraic set and ${\tt j}:=({\tt j}_1,\ldots,{\tt j}_m):M\to E$ is an ${\mathcal S}^r$-embedding such that $E\subset\cl({\tt j}(M))$\em. A similar result
\begin{equation}\label{dlim2}
{\mathcal S}^{r*}(M)\cong\displaystyle\lim_{\longrightarrow}({\mathcal S}^r(E),{\tt j}).
\end{equation}
holds in the bounded case, but in this case we impose that the semialgebraic sets $E$ are compact. 

We show also that some properties of ${\mathcal S}^r$-functions on locally compact semialgebraic sets transfer through the direct limit to arbitrary semialgebraic sets using formula \eqref{dlim}. This transfer method has limitations. For instance, in \cite{fg5} it is shown that \L ojasiewicz's inequality \cite[Cor.2.6.7]{bcr} holds in the $\R$-algebra ${\mathcal S}(M)$ if and only if $M$ is locally compact. 

In Section \ref{s3} we analyze the Zariski spectra of the ring ${\mathcal S}^{r\diam}(M)$. Given a commutative ring $A$, we denote $\Spec(A)$ its Zariski spectrum endowed with the Zariski topology, which has as a subbasis of open sets the family of \emph{basic open sets} $\Dd(a):=\{\gtp\in\Spec(A):\ a\not\in\gtp\}$ for $a\in A$. The \emph{constructible} subsets of $\Spec(A)$ are the Boolean combinations of the latter basic open sets. To lighten notations, for each $r\geq0$ we write $\Spec^{r\diam}(M):=\Spec({\mathcal S}^{r\diam}(M))$ and $\Spec^\diam(M):=\Spec({\mathcal S}^\diam(M))$. The main result of Section \ref{s3} is that the Zariski spectra of ${\mathcal S}^{r\diam}(M)$ and ${\mathcal S}^{0\diam}(M)$ are homeomorphic. More precisely,

\begin{thm}\label{main1}
Let $M\subset\R^m$ be a semialgebraic set. Then the map
$$
\varphi:\Spec^{0\diam}(M)\to\Spec^{r\diam}(M),\ \gtp\mapsto\gtp\cap{\mathcal S}^{r\diam}(M)
$$
is a homeomorphism and its inverse map is
$$
\psi:\Spec^{r\diam}(M)\to\Spec^{0\diam}(M),\ \gtq\mapsto\sqrt{\gtq{\mathcal S}^0(M)}.
$$
\end{thm}

The previous result implies that for each $r\geq0$, both ${\mathcal S}^r(M)$ and ${\mathcal S}^{r*}(M)$ are Gelfand rings and their Krull dimensions coincide with the dimension of $M$. Denote $\beta_s^{r\diam}M$ the maximal spectra of ${\mathcal S}^{r\diam}(M)$. As homeomorphisms between topological spaces preserve closed points, the restriction
\begin{align*}
\varphi|_{\beta^{0\diam}M}:\beta^{0\diam}M\to\beta^{r\diam}M,\ &\gtm\mapsto\gtm\cap{\mathcal S}^{r\diam}(M)
\end{align*}
is also a homeomorphism. 

The proof of Theorem \ref{main1} relies on the fact that each ring ${\mathcal S}^r(M)$ enjoys a \L ojasiewiz's Nullstellesatz when $M$ is locally compact. Denote $Z(f)$ the zero-set of a function $f\in{\mathcal S}^r(M)$. We have the following.

\begin{thm}[\L ojasiewicz's Nullstellensatz]\label{null2}
Let $M\subset\R^m$ be a locally compact semialgebraic set and let $f,g\in{\mathcal S}^r(M)$ with $Z(f)\subset Z(g)$. Then there exist an integer $\ell>0$ and $h\in{\mathcal S}^r(M)$ such that $g^{\ell}=hf$ and $Z(g)=Z(h)$.
\end{thm}

In Section \ref{s4} we analyze the field of fractions of the quotient ${\mathcal S}^{r\diam}(M)/\gtq$ of the ring ${\mathcal S}^{r\diam}(M)$ by a prime ideal $\gtq$ and we prove the following.

\begin{thm}\label{main3}
Let $\gtp$ be a prime ideal of ${\mathcal S}^{0\diam}(M)$ and consider the inclusion 
$$
{\tt j}:{\mathcal S}^{r\diam}(M)/(\gtp\cap{\mathcal S}^{r\diam}(M))\to{\mathcal S}^{0\diam}(M)/\gtp.
$$
Then ${\tt j}$ induces an isomorphism between the fields of fractions $\kappa(\gtp\cap{\mathcal S}^{r\diam}(M))$ and $\kappa(\gtp)$ of the integral domains ${\mathcal S}^{r\diam}(M)/(\gtp\cap{\mathcal S}^{r\diam}(M))$ and ${\mathcal S}^{0\diam}(M)/\gtp$. In particular, $\kappa(\gtp\cap{\mathcal S}^{r\diam}(M))$ is a real closed field.
\end{thm}

In Section \ref{s5} we contextualize the ring of ${\mathcal S}^{r\diam}$-functions within the theory of \emph{real closed rings} (Definition \ref{def:rcr}), introduced by Schwartz in the 80's of the last century \cite{s1,s3}. As it is well-known, the rings ${\mathcal S}^{\diam}(M)$ are particular cases of such real closed rings. The theory of real closed rings has been deeply developed until now as a fruitful attempt to establish new foundations for semi-algebraic geometry with relevant interconnections to model theory, see \cite{cd1,cd2}, \cite{s1,s2,s3,s4,s5}, \cite{ps,sm,scht} and \cite{t1,t2,t3}. This theory generalizes the classical techniques concerning the semi-algebraic spaces of Delfs-Knebusch (see \cite{dk2}), provides a powerful machinery to approach problems concerning certain rings of real-valued functions and contributes to achieve a better understanding of the algebraic and topological properties of such rings. We highlight some celebrated examples:
\begin{itemize}
\item Rings of real-valued continuous functions on Tychonoff spaces.
\item Rings of semi-algebraic functions on semi-algebraic sets over an arbitrary
real closed field.
\item Rings of definable continuous functions on definable sets in o-minimal
expansions of fields.
\end{itemize}

Every unital commutative ring $A$ has a so-called \em real closure ${\rm rcl}(A)$\em, which is unique up to a unique ring homomorphism over $A$ (see \cite[\S I]{s4}). This means that ${\rm rcl}(A)$ is a real closed ring and there exists a (not necessarily injective) ring homomorphism ${\tt i}:A\to{\rm rcl}(A)$ such that for every ring homomorphism $f:A\to B$ to some other real closed ring $B$ there exists a unique ring homomorphism $\ol{f}:{\rm rcl(A)}\to B$ satisfying $f=\ol{f}\circ{\tt i}$ (universal property). A strong property is that the real spectra of ${\rm rcl}(A)$ and $A$ are homeomorphic. A crucial example for us is that the real closure of the polynomial ring $\R[\x]:=\R[\x_1,\ldots,\x_n]$ is the ring of continuous semialgebraic functions ${\mathcal S}(\R^n)$. In fact, if $Z\subset\R^n$ is an algebraic set, the real closed ring ${\mathcal S}(Z)$ is the real closure of the ring of polynomial functions on $Z$. 

In our setting we show: \em the real closure of ${\mathcal S}^{r\diam}(M)$ is ${\mathcal S}^{0\diam}(M)$\em, see Proposition \ref{f:rcl1}. In case $M$ is locally compact, the previous statement follows almost straightforwardly from the results of Schwartz and Tressl. Thus, in such situation we know from the very beginning that the real spectra of ${\mathcal S}^{r\diam}(M)$ and ${\mathcal S}^{0\diam}(M)$ are homeomorphic (and the residue fields of the latter are the real closures of the former). In Theorems \ref{main1} and \ref{main3} we provide more (explicit) information concerning Zariski spectra and the residue fields of ${\mathcal S}^{r\diam}(M)$, which suggests that ${\mathcal S}^{r\diam}(M)$ is close to be a real closed ring. We would like to point out that the only proof we know of the fact that ${\mathcal S}^{0\diam}(M)$ is the real closure of ${\mathcal S}^{r\diam}(M)$ for an arbitrary semialgebraic set $M$ is the one provided here, which uses Theorems \ref{main1} and \ref{main3} in a crucial way.
 
Finally, we consider the ring ${\mathcal S}^{\infty\diam}(M):=\bigcap_{r\geq0}{\mathcal S}^{r\diam}(M)$, which we can also describe as the inverse limit of the family of rings $\{{\mathcal S}^r(M)\}_{r\geq0}$ with respect to the inclusion relation. It is natural to compare the ring ${\mathcal S}^{\infty}(M)$ with its subring of Nash functions ${\mathcal N}(M)$, that is, the collection of functions $f:M\to\R$ that admit a Nash extension to an open semialgebraic neighborhood $U\subset\R^m$ of $M$. Recall that a real function $g$ on an open semialgebraic set $U\subset\R^m$ is \em Nash \em if it is semialgebraic and smooth on $U$. This property is equivalent to be an analytic function and algebraic over the polynomials, that is, there exists a non-zero polynomial $P\in\R[\x_1,\ldots,\x_m,\y]$ such that $P(x,f(x))=0$ for each $x\in U$. We show that the rings ${\mathcal S}^{\infty}(M)$ and ${\mathcal N}(M)$ are different in general (Example \ref{cex:nash}) and they coincide if $M$ is a coherent Nash set (Corollary \ref{cohNash}). Recall that a semialgebraic set $M\subset\R^m$ is a \em Nash set \em if it is locally compact and there exists a Nash function $f$ on the open semialgebraic set $U:=\R^m\setminus(\cl(M)\setminus M)$ such that $M=Z(f)$, see \cite{fg2}. The coherent condition refers to Serre's notion \cite[\S 2.B]{bfr} and it is the natural tool when globalizing local information in the analytic and Nash settings. In practice coherence in the Nash setting can be summarized as follows. Let $I(M):=\{f\in{\mathcal N}(U):\ M\subset Z(f)\}$ and let $I(M_x):=\{f\in{\mathcal N}(U_x):\ M_x\subset Z(f_x)\}$ for each $x\in U$, where subindices are used to stress the use of germs. The Nash set $M$ is \em coherent \em if $I(M_x)=I(M){\mathcal N}(U_x)$ for each $x\in U$. In the forth coming paper \cite{fgh} it is provided a full characterization of the semialgebraic sets $M\subset\R^m$ such that ${\mathcal N}(M)={\mathcal S}^{\infty}(M)$ in terms of the equality of the ideals $I(M_x)$ and $I(M){\mathcal N}(U_x)$ at each point $x\in M$. In case $M$ is a Nash set the equality ${\mathcal N}(M)={\mathcal S}^{\infty}(M)$ holds if and only if $M$ is coherent.

The fact that the real closure of ${\mathcal S}^{\infty\diam}(M)$ equals the ring ${\mathcal S}^{0\diam}(M)$ is a strong condition as the following result reveals.

\begin{thm}\label{nash}
Let $M\subset\R^m$ be a semialgebraic set such that the inclusion ${\tt j}:{\mathcal S}^{\infty\diam}(M)\hookrightarrow{\mathcal S}^{0\diam}(M)$ provides the real closure of ${\mathcal S}^{\infty\diam}(M)$. Then for each $x\in M$ the germ $\cl(M)_x$ is a Nash germ.
\end{thm}

The previous results stress that it is not true in general that the real closure of an inverse limit is the inverse limit of the real closures. It seems difficult to deal with semialgebraic sets such that `for each $x\in M$ the germ $\cl(M)_x$ is a Nash germ', as it is a property of infinitesimal nature. However, when we approach such property as a global one, we obtain the following positive result.

\begin{prop}\label{nash2}
Let $M\subset\R^m$ be a semialgebraic set and assume that there exists an open semialgebraic set $U\subset\R^m$ such that $\cl(M)\cap U$ is a Nash set. Then the inclusion ${\tt j}:{\mathcal S}^{\infty\diam}(M)\hookrightarrow{\mathcal S}^{0\diam}(M)$ provides the real closure of ${\mathcal S}^{\infty\diam}(M)$.
\end{prop}

It is natural to wonder if the `local property at each point' implies `the global property'. But this is false because `to be locally a Nash set' does not imply `to be globally a Nash set'. For an evidence consider for instance the semialgebraic set $M:=\{x^2-(z^2-1)y^2=0,z>0\}\cup\{x=0,y=0\}$. It holds that the Nash closure in each open semialgebraic neighborhood $U$ of $M$ is the connected component of $X:=\{x^2-(z^2-1)y^2=0\}\cap U$ that contains $M$. Clearly, $M\subsetneq X$.

\vspace{1mm}
\noindent{\bf Structure of the article.}
The article is organized as follows. In Section \ref{s2} we introduce and analyze with care the concept of ${\mathcal S}^r$-function on a semialgebraic set $M\subset\R^m$ and we prove equalities \eqref{dlim} and \eqref{dlim2}. We also prove Theorem \ref{main2}. The main purpose of Section \ref{s3} is to prove Theorem \ref{main1}. The proofs for the ${\mathcal S}^r$ case and ${\mathcal S}^{r*}$ case are of different nature. The first one is based on a ${\mathcal S}^r$ version of \L ojasiewicz Nullstellensatz (Proposition \ref{null20}) valid only in its full generality for the locally compact case. In Proposition \ref{crucial} we modify the statement in order to approach the general case. The proof of Theorem \ref{main1} for the ${\mathcal S}^{r*}$ case requires other type of techniques (like Nash approximation) combined with Proposition \ref{null20}. In Section \ref{s4} we prove Theorem \ref{main3} whereas in Section \ref{s5} we prove Theorem \ref{nash} and Proposition \ref{nash2}. In Corollary \ref{a} we show that the sum of two radical ideals of a ring of ${\mathcal S}^{r\diam}$ functions needs not to be a radical ideal. In view of Definition \ref{def:rcr} this reveals to be `the missing property'.

\vspace{1mm}
\noindent{\bf Acknowledgements.} 
The authors would like to thank Professor Marcus Tressl for sharing with them his deep knowledge in the theory of real closed rings.

\section{Rings of ${\mathcal S}^r$-functions}\label{s2}

If $U\subset\R^m$ is an open semialgebraic set and $r\geq0$ is an integer, we say that a semialgebraic function on $U$ is an ${\mathcal S}^r$ \em function \em if it is a function of class ${\mathcal C}^r$ in the classical sense. Our purpose next is to extend the previous definition for functions on an arbitrary semialgebraic set. In the literature the ${\mathcal S}^r$-functions are considered over closed semialgebraic subsets of $\R^m$ and are introduced as the restriction of an ${\mathcal S}^r$-function defined on an open semialgebraic neighborhood of the original semialgebraic set. The existence of a (continuous) semialgebraic extension is never an obstacle: \em any locally compact semialgebraic set $M\subset\R^m$ is the semialgebraic retract of a small open semialgebraic neighborhood $V\subset\R^m$ of $M$\em, see \cite[Thm.1]{dk1}. For non-locally compact semialgebraic sets even the existence of a semialgebraic extension map is a delicate point \cite{fe2}. As we see below, if $f:M\to\R$ is a semialgebraic function, there exists a semialgebraic embedding ${\tt j}:M\hookrightarrow\R^p$ and a semialgebraic function $F:W\to\R$ on an open semialgebraic neighborhood $W\subset\R^p$ of ${\tt j}(M)$ such that $F\circ{\tt j}=f$.

\subsection{Differentiability for semialgebraic functions}
Our purpose is to work in the differentiable case as in the continuous one, where we always deal with the classical intrinsic notion of continuity that avoids the intricate extension aspect. This is a classical problem that goes back to Whitney's extension theorem. Inspired by the latter and based on jets of semialgebraic functions, we recall and develop a notion of differentiable semialgebraic function of intrinsic nature proposed in \cite{at} motivated by the proof of Whitney's extension theorem collected in \cite[\S1]{m}. We denote $\N:=\{0,1,2,\ldots\}$ the set of natural numbers including $0$.

\begin{defns}\label{def:jets}
Fix a semialgebraic set $M\subset\R^m$ and denote $\x:=(\x_1,\ldots,\x_m)$. A \em semialgebraic jet on $M$ of order $r\geq0$ \em is a collection of semialgebraic functions $F:=(f_{\alpha})_{|\alpha|\leq r}$ on $M$. Here we denote $\alpha:=(\alpha_1,\ldots,\alpha_m)\in\N^m$ and $|\alpha|:=\alpha_1+\cdots+\alpha_m$. For each $a\in M$ write
$$
T_a^rF:=\sum_{|\alpha|\leq r}\frac{f_\alpha(a)}{\alpha!}(\x-a)^\alpha\quad\text{and}\quad R_a^rF:=f_0-T_a^rF.
$$
Denote the set of all semialgebraic jets on $M$ of order $r$ with ${\mathcal J}^r(M)$, which has a natural structure of $\R$-vector space. For each $\beta\in\N^m$ such that $|\beta|\leq r$ we denote 
$$
D^\beta:{\mathcal J}^r(M)\to{\mathcal J}^{r-|\beta|}(M),\ F:=((f_{\alpha})_{|\alpha|\leq r})\mapsto F_\beta:=(f_{\gamma+\beta})_{|\gamma|\leq r-|\beta|},
$$
that is a linear map.
\end{defns}

\begin{defn}\label{defsr}
Let $f\in{\mathcal S}(M)$ be a semialgebraic function. We say that $f$ is a \em ${\mathcal S}^r$-function \em if there exists a semialgebraic jet $F:=(f_\alpha)_{|\alpha|\leq r}$ on $M$ of order $r$ such that 

\paragraph{}\label{diff}\em $f_0=f$ and for each $\beta$ with $|\beta|\leq r$ and each point $a\in M$ it holds $|R_x^{r-|\beta|}F_\beta(y)|=o(\|x-y\|^{r-|\beta|})$ for $x,y\in M$ when $x,y\to a$,\em

\noindent that is, the function $\tau:(0,+\infty)\to\R$ defined by
$
\tau(t):=\sup_{\substack{x,y\in M,\ x\neq y\\ \|x-a\|\leq t,\ \|y-a\|\leq t}}\frac{|R_x^{r-|\beta|}F_\beta(y)|}{\|x-y\|^{r-|\beta|}}
$
is an increasing function that can be continuously extended to $t=0$ by $\tau(0)=0$.

We denote by ${\mathcal S}^r(M)$ the collection of all the ${\mathcal S}^r$-functions on $M$.
\end{defn}
\begin{remarks}\label{Rest}
(i) Condition \ref{diff} is equivalent by \cite[I.Thm.2.2]{m} to the following one:

\paragraph{}\label{diff2}\em $f_0=f$ and for each point $a\in M$ there exists an increasing, continuous and concave function $\alpha_a:[0,+\infty)\to[0,+\infty)$ such that $\alpha_a(0)=0$ and for each $z\in\R^m$ it holds
$$
|T_x^rF(z)-T_y^rF(z)|\leq\alpha_a(\|x-y\|)\cdot(\|z-x\|^r+\|z-y\|^r)
$$ 
for $x,y\in M$ when $x,y\to a$. \em

(ii) If $M\subset\R^m$ is an open semialgebraic subset, Definition \ref{defsr} coincides with the expected one. That is, \em $f$ is an ${\mathcal S}^r$-function on $M$ if and only if $f$ is a ${\mathcal C}^r$ function and a semialgebraic function\em. Indeed, if $f\in{\mathcal S}(M)$ is a ${\mathcal C}^r$ function the semialgebraic jet $(f_\alpha:=\frac{\partial^{|\alpha|}f}{\partial\x^\alpha})_{|\alpha|\leq r}$ of order $r$ on $M$ satisfies by Taylor's Theorem condition \ref{diff}. Conversely, condition \ref{diff} straightforwardly implies that $f_\alpha=\frac{\partial^{|\alpha|}f}{\partial\x^\alpha}$ if $|\alpha|\leq r$, so $f$ is a ${\mathcal C}^r$ function. In addition, $f=f_0$ is semialgebraic.

(iii) It follows directly from the definition that if $M_1\subset M_2\subset\R^m$ are semialgebraic sets and $f$ is an ${\mathcal S}^r$-function on $M_2$, then $f|_{M_1}$ is an ${\mathcal S}^r$-function on $M_1$.
\end{remarks}

Whitney already proved in his pioneer work that jets on closed subsets satisfying condition \ref{diff} are restrictions of global differentiable functions (Whitney's extension theorem). The first result in a tame category (the subanalytic one) can be found in \cite{kp1}. Recently, this has been generalized to o-minimal structures in \cite{kp2,th}. Specifically, in our setting we have:

\begin{fact}[\cite{th}]\label{thfact}
Let $M\subset\R^m$ be a closed semialgebraic set and let $f$ be an ${\mathcal S}^r$-function on $M$. Let $(f_\alpha)_{|\alpha|\leq r}$ be a semialgebraic jet of order $r$ on $M$ satisfying condition \em\ref{diff}\em. Then there exists an ${\mathcal S}^r$-function $F$ on $\R^m$ such that $F|_M=f$ and $\frac{\partial^{|\alpha|}F}{\partial\x^\alpha}|_M=f_\alpha$ for $|\alpha|\leq r$.
\end{fact}
\begin{remark}
Observe that an ${\mathcal S}^r$-function $f$ on a closed semialgebraic set $M\subset\R^m$ admits many semialgebraic jets of order $r$. In fact, the different ${\mathcal S}^r$-extensions $F$ of $f$ to $\R^m$ define different semialgebraic jets $(\frac{\partial^{|\alpha|}F}{\partial\x^\alpha}|_M)_{|\alpha|\leq r}$ of order $r$ on $M$. Consider for instance $M:=\{\y=0\}\subset\R^2$ and the ${\mathcal S}^2$-functions $F_1:=\x+\y^2$ and $F_2:=\x-\y-2\y^2$. We have $f:=F_1|_M=F_2|_M$, so $f$ is an ${\mathcal S}^2$-function on $M$. However, the semialgebraic jets 
\begin{align*}
&\Big(\frac{\partial^{|\alpha|}F_1}{\partial\x^\alpha}\Big|_M\Big)_{|\alpha|\leq 2}=(\x;1,0;0,0,2),\\
&\Big(\frac{\partial^{|\alpha|}F_2}{\partial\x^\alpha}\Big|_M\Big)_{|\alpha|\leq 2}=(\x;1,-1;0,0,-4)
\end{align*} 
of order $2$ are different.
\end{remark}

Our definition of an ${\mathcal S}^r$-function on a closed subset coincides with the classical one concerning extensions that is frequently used in the literature. However, as we point out in the Introduction we will not focus our study only on closed or locally compact semialgebraic subsets of $\R^m$. We prove next the following extension result. As usual $B(a,\veps)$ denotes the ball of $\R^m$ of center $a\in\R^n$ and radius $\veps>0$.

\begin{lem}\label{lips}
Let $f:M\to\R$ be an ${\mathcal S}^r$-function where $r\geq 1$. Let $(f_\alpha)_{|\alpha|\leq r}$ be a semialgebraic jet associated to $f$. Then 
\begin{itemize}
\item[(i)] $f_\alpha$ is locally Lipschitz on $M$ for each $\alpha$ with $|\alpha|<r$.
\item[(ii)] There exists an open semialgebraic neighborhood $N\subset\cl(M)$ of $M$ such that $f_\alpha$ admits a continuous extension $F_\alpha$ to $N$, which is locally Lipschitz on $N$ for each $\alpha$ with $|\alpha|<r$. In particular, $f_\alpha\in{\mathcal S}^0(M)$ for each $\alpha$ with $|\alpha|<r$.
\item[(iii)] After shrinking $N$ if necessary, we may assume $F_0$ is an ${\mathcal S}^{r-1}$ extension of $f$.
\end{itemize}
\end{lem}

\begin{proof}
 (i) Pick a point $a\in M$. As $f$ is a ${\mathcal S}^1$ function, there exists $\veps>0$ such that
\begin{align*}
&|f(x)-f(y)-\sum_{i=1}^nf_{{\tt e}_i}(y)(x_i-y_i)|\leq\|x-y\|,\\
&|f_{{\tt e}_i}(y)-f_{{\tt e}_i}(a)|<1
\end{align*}
for each $x,y\in M\cap B(a,\veps)$. Thus,
$$
|f(x)-f(y)|\leq\Big(n+1+\sum_{i=1}^n|f_{{\tt e}_i}(a)|\Big)\|x-y\|
$$
for each $x,y\in M\cap B(a,\veps)$, so $f$ is locally Lipschitz at $a\in M$. Analogously, each $f_\alpha$ is locally Lipschitz at $a\in M$ for each $\alpha$ with $|\alpha|<r$.

(ii) Let $N\subset\cl(M)$ be the set of points $x\in\cl(M)$ such that $f$ is locally Lipschitz in the intersection with $M$ of a neighborhood of $x$, that is, there exist $\veps>0$ and $K>0$ (depending on $x$) satisfying
$$
|f(y_1)-f(y_2)|<K\|y_1-y_2\|
$$
for each $y_1,y_2\in M\cap B(x,\veps)$. By its very definition and (i) it holds that $N$ is an open semialgebraic subset of $\cl(M)$ that contains $M$. For each point $x\in N$ there exists an open semialgebraic neighborhood $V^x\subset N\subset\cl(M)$ of $x$ such that $f|_{V^x\cap M}$ is Lipschitz, so in particular $f|_{V^x\cap M}$ is uniformly continuous. Thus, $f|_{V^x\cap M}$ admits a unique continuous extension to $\cl(V^x\cap M)$. For each $x\in N$ we have that $V^x\subset\cl(V^x\cap M)$, so there exists a (unique) continuous extension $F$ of $f$ to $N$. The function $F$ is semialgebraic because its graph $\Gamma(F)=\cl(\Gamma(f))\cap(N\times\R)$. 

Next, let us show that $F$ is locally Lipschitz. Fix point $a\in N$ and let $K>0$ and $\veps>0$ be such that $N\cap B(a,\veps)=\cl(M)\cap B(a,\veps)$ and
$$
|f(x)-f(y)|\leq K\|x-y\|
$$
for each $x,y\in M\cap B(a,\veps)$. Pick two points $u,v\in\cl(M)\cap B(a,\veps)$ with $u\neq v$. As $F$ is a continuous extension of $f$, there exists $x\in M$ such that $\max\{|f(x)-F(u)|,\|x-u\|\}<\|u-v\|$. Similarly, there exists $y\in M$ such that $\max\{|f(y)-F(v)|,\|y-v\|\}<\|u-v\|$. Therefore,
\begin{equation*}
\begin{split}
|F(u)-F(v)|&\leq|F(u)-f(x)|+|f(x)-f(y)|+|f(y)-F(v)|\\
&\leq 2\|u-v\|+K\|x-y\|\leq 2\|u-v\|+K(\|x-u\|+\|u-v\|+\|v-y\|)\\
&=(3K+2)\|u-v\|
\end{split}
\end{equation*}
and $F:N\to\R$ is locally Lipschitz at $a$. The same happens for each $\alpha$ with $|\alpha|<r$ and statement (ii) holds.

(iii) Pick a point $a\in M$. As $f$ is an ${\mathcal S}^r$-function, there exists $\veps>0$ such that
\begin{align*}
&\Big|f(x)-\sum_{|\alpha|\leq r}\frac{f_{\alpha}(y)}{\alpha!}(x-y)^\alpha\Big|\leq\|x-y\|^r\\
&|f_{\beta}(y)-f_{\beta}(a)|<1
\end{align*}
for each $x,y\in M\cap B(a,\veps)$ and each $\beta\in\N^n$ with $|\beta|=r$. Thus,
\begin{equation}\label{formula}
\begin{split}
\Big|f(x)-\sum_{|\alpha|\leq r-1}\frac{f_{\alpha}(y)}{\alpha!}(x-y)^\alpha\Big|&\leq\Big|f(x)-\sum_{|\alpha|\leq r}\frac{f_{\alpha}(y)}{\alpha!}(x-y)^\alpha\Big|\\
&+\sum_{|\beta|=r}\frac{|f_{\beta}(y)-f_{\beta}(a)|}{\beta!}|x-y|^\beta+\sum_{|\beta|=r}\frac{|f_{\beta}(a)|}{\beta!}|x-y|^\beta\\
&\leq\Big(1+\sum_{|\beta|=r}\frac{1+|f_{\beta}(a)|}{\beta!}\Big)\|x-y\|^r
\end{split}
\end{equation}
for each $x,y\in M\cap B(a,\veps)$. 

Let $N_0$ be an open semialgebraic subset of $\cl(M)$ containing $M$ such that each $f_\alpha$ with $|\alpha|<r$ admits a locally Lipschitz (continuous) semialgebraic extension $F_\alpha$ to $N_0$. Its existence is guaranteed by (ii). Let $N$ be the set of points $x\in N_0$ such that there exists $\veps>0$ and $K>0$ (depending on $x$) satisfying
$$
\Big|f(y_1)-\sum_{|\alpha|\leq r-1}\frac{f_{\alpha}(y_2)}{\alpha!}(y_1-y_2)^\alpha\Big|<K\|y_1-y_2\|^r
$$
for each $y_1,y_2\in M\cap B(x,\veps)$. By \eqref{formula} and its very definition $N$ is an open semialgebraic subset of $\cl(M)$ that contains $M$. Let us show that $F:=F_0$ is ${\mathcal S}^{r-1}$ on $N$. Fix a point $a\in N$ and let $\veps>0$ and $K>0$ be such that $N\cap B(a,\veps)=\cl(M)\cap B(a,\veps)$,
\begin{align*}
&\Big|f(x)-\sum_{|\alpha|\leq r-1}\frac{f_{\alpha}(y)}{\alpha!}(x-y)^\alpha\Big|<K\|x-y\|^r,\\
&|f_{\gamma}(y)-f_{\gamma}(a)|<1
\end{align*}
for each $x,y\in M\cap B(a,\veps)$ and each $\gamma\in\N^n$ with $|\gamma|<r$. Note that 
$$
|f_\gamma(y)|<1+|f_\gamma(a)|
$$
for each $\gamma\in\N^n$ with $|\gamma|<r$.

We must show that there exists a constant $K'>0$ such that for each two fixed points $u,v\in\cl(M)\cap B(a,\veps)$ with $u\neq v$, we have
\begin{equation}\label{r-1diff}
\frac{\Big|F(u)-\sum_{|\alpha|\leq r-1}\frac{F_{\alpha}(v)}{\alpha!}(u-v)^\alpha\Big|}{\|u-v\|^{r-1}}\leq K'\|u-v\|.
\end{equation}
To that aim, let $x,y\in M\cap B(a,\veps)$ be such that
\begin{align*}
&|F_\alpha(u)-F_\alpha(x)|<\|u-v\|^r,\\
&|F_\alpha(v)-F_\alpha(y)|<\|u-v\|^{r-|\alpha|},\\
&|(x-y)^\alpha-(u-v)^\alpha|<\|u-v\|^r,\\
&\|u-x\|<\|u-v\|,\\
&\|v-y\|<\|u-v\|
\end{align*}
for each $\alpha\in\N^n$ with $|\alpha|<r$ (to prove the existence of $x,y$ it is enough to consider the continuous map $N^2\mapsto\R^5$ whose function coordinates are the left hand side expressions of each of the above inequalities and to note that the image of $(u,v)$ is $0\in\R^5$). In particular, 
$$
\|x-y\|\leq\|x-u\|+\|u-v\|+\|v-y\|\leq 3\|u-v\|,
$$
so $\|x-y\|^r\leq 3^r\|u-v\|^r$. We deduce
\begin{equation*}
\begin{split}
\Big|F(u)&-\sum_{|\alpha|\leq r-1}\frac{F_{\alpha}(v)}{\alpha!}(u-v)^\alpha\Big|\leq|F(u)-f(x)|+\Big|f(x)-\sum_{|\alpha|\leq r-1}\frac{f_{\alpha}(y)}{\alpha!}(x-y)^\alpha\Big|\\
&+\sum_{0<|\alpha|\leq r-1}\Big|\frac{f_{\alpha}(y)}{\alpha!}\Big||(x-y)^\alpha-(u-v)^\alpha|+\sum_{|\alpha|\leq r-1}\Big|\frac{f_{\alpha}(y)-F_{\alpha}(v)}{\alpha!}\Big|\|u-v\|^{|\alpha|}\\
&\leq\|u-v\|^r+K\|x-y\|^r+\sum_{|\alpha|\leq r-1}\frac{1+|f_{\alpha}(a)|}{\alpha!}\|u-v\|^r+\sum_{|\alpha|\leq r-1}\frac{1}{\alpha!}\|u-v\|^r\\
&\leq\Big(1+3^rK+\sum_{|\alpha|\leq r-1}\frac{2+|f_{\alpha}(a)|}{\alpha!}\Big)\|u-v\|^r,
\end{split}
\end{equation*}
and therefore to get inequality (\ref{r-1diff}) it is enough to define $K':=(1+3^rK+\sum_{|\alpha|\leq r-1}\frac{2+|f_{\alpha}(a)|}{\alpha!})$, as required.
\end{proof}

\subsection{${\mathcal S}^r$-maps and embeddings}
We deal next with maps instead of functions. This allows to work with locally compact semialgebraic sets as closed subsets (see Corollary \ref{ext}).

\begin{defn}
Let $M\subset\R^m$ be a semialgebraic set. A map 
$$
\varphi:=(\varphi_1,\ldots,\varphi_n):M\to\R^n
$$ 
is a \em ${\mathcal S}^r$-map \em if each component $\varphi_i$ is an ${\mathcal S}^r$-function. In addition, $\varphi$ is a \em ${\mathcal S}^r$-embedding \em if it is injective and $\varphi^{-1}:\varphi(M)\to\R^m$ is an ${\mathcal S}^r$-map. In such case we say that $\varphi:M\to\varphi(M)$ is a \em ${\mathcal S}^r$ diffeomorphism, \em and that \em $M$ and $\varphi(M)$ are ${\mathcal S}^r$ diffeomorphic. \em 
\end{defn}

The following composition result is crucial.

\begin{thm}\label{comp} 
Let $M\subset\R^m$ be a semialgebraic set, $\varphi:=(\varphi_1,\ldots,\varphi_n):M\to\R^n$ an ${\mathcal S}^r$-map and let $N:=\varphi(M)$. Let $f$ be an ${\mathcal S}^r$-function on $N$. Then $g:=f\circ\varphi$ is an ${\mathcal S}^r$-function on $M$. 
\end{thm}
\begin{proof}
We deal with the case $r\geq 2$. The case $r=1$ follows with the obvious modifications. As $\varphi$ is an ${\mathcal S}^r$-map, for $i=1,\ldots,n$ there exists a semialgebraic jet $\Phi_i:=(\varphi_{i,\alpha})_{|\alpha|\leq r}$ on $M$ of order $r$ such that $\varphi_{i,0}=\varphi_i$ and for each $a\in M$ it holds $|R_x^r\Phi_i(y)|=o(\|x-y\|^r)$ for $x,y\in M$ when $x,y\to a$. Denote $\varphi_\alpha:=(\varphi_{1,\alpha},\ldots,\varphi_{n,\alpha})$ for $|\alpha|\leq r$ and $\Phi:=(\varphi_\alpha)_{|\alpha|\leq r}=(\Phi_1,\ldots,\Phi_n)$. We also write $T_x^r\Phi:=(T_x^r\Phi_1,\ldots,T_x^r\Phi_n)$ and $R_x^r\Phi:=(R_x^r\Phi_1,\ldots,R_x^r\Phi_n)$.

As $f$ is an ${\mathcal S}^r$-function on $N$, there exists a semialgebraic jet $F:=(f_\alpha)_{|\alpha|\leq r}$ on $N$ of order $r$ such that $f_0=f$ and for each $b\in N$ there exists an increasing, continuous, concave function $\alpha_b:[0,+\infty)\to[0,+\infty)$ such that $\alpha_b(0)=0$ and for each $w\in\R^n$ and $u,v\in N$ with $u,v\to b$, 
$$
|T_u^rF(w)-T_v^rF(w)|\leq\alpha_b(\|u-v\|)\cdot(\|w-u\|^r+\|w-v\|^r),
$$ 
see Remark \ref{Rest}(i). In particular, for each $a\in M$ the increasing, continuous and concave function $\alpha_{\varphi(a)}:[0,+\infty)\to[0,+\infty)$ satisfies $\alpha_{\varphi(a)}(0)=0$ and for each $z\in\R^m$ and $x,y\in M$ with $x,y\to a$,
\begin{multline}\label{ineq}
|T_{\varphi(x)}^rF(T_x^r\Phi(z))-T_{\varphi(y)}^rF(T_x^r\Phi(z))|\\
\leq\alpha_{\varphi(a)}(\|\varphi(x)-\varphi(y)\|)\cdot(\|T_x^r\Phi(z)-\varphi(x)\|^r+\|T_x^r\Phi(z)-\varphi(y)\|^r).
\end{multline} 

Let us construct a suitable semialgebraic jet $G:=(g_\alpha)_{|\alpha|\leq r}$ on $M$ of order $r$ such that $g_0=g$. Write $\y:=(\y_1,\ldots,\y_m)$ and $\z:=(\z_1,\ldots,\z_m)$. Let $Q(x,\y)\in{\mathcal S}(M)[\y]$ be a polynomial of degree $\leq r^2$ such that for each $x\in M$ we have
$$
T_{\varphi(x)}^rF(T_x^r\Phi(\z))=Q(x,\z-x).
$$
Let $g_\alpha\in{\mathcal S}(M)$ be $\alpha!$ times the coefficient of $Q$ corresponding to $\y^\alpha$. Note that $g_\alpha$ is a finite sum of finite products of some of the semialgebraic functions $f_\alpha\circ\varphi$ and $\varphi_{i,\alpha}$. Thus, $G:=(g_\alpha)_{|\alpha|\leq r}$ is a semialgebraic jet on $M$ of order $r$. Observe that $g_0=f_0\circ\varphi=f\circ\varphi=g$. By definition
$$
T_x^rG(\z)=\sum_{|\alpha|\leq r}\frac{g_\alpha(x)}{\alpha!}(\z-x)^\alpha\in\R[\z-x].
$$ 
Let $S(x,\y)\in{\mathcal S}(M)[\y]$ be such that $T_x^rG(\z)=S(x,\z-x)$ and define $P:=Q-S$. For each $x\in M$ the non-zero monomials of the polynomial $P_x(\z):=P(x,\z-x)\in\R[\z-x]$ have degree between $r+1$ and $r^2$. As the coefficients of $P$ are continuous functions, for each $a\in M$ there exists a constant $K_a>0$ such that 
\begin{equation}\label{pasito1}
|P_x(z)|<K_a\sum_{k=r+1}^{r^2}\|z-x\|^k
\end{equation} 
for $x$ close to $a$ and $z\in\R^m$. Consequently, there exists a constant $K_a'>0$ such that
\begin{equation}\label{pasito2}
|P_x(z)|<K_a'\cdot\big(\|z-x\|^{r+1}+\|z-x\|^{r^2}\big).
\end{equation}
For $z\in\R^m$ and $x,y\in M$ with $x,y\to a$ we have by \eqref{ineq} and \eqref{pasito2}
\begin{equation}\label{ineq2}
\begin{split}
|T_x^rG(z)-T_y^rG(z)|&\leq|T_{\varphi(x)}^rF(T_x^r\Phi(z))-T_{\varphi(y)}^rF(T_y^r\Phi(z))|+|P_x(z)|+|P_y(z)|\\
&\leq\alpha_{\varphi(a)}(\|\varphi(x)-\varphi(y)\|)\cdot(\|T_x^r\Phi(z)-\varphi(x)\|^r+\|T_x^r\Phi(z)-\varphi(y)\|^r)\\
&+K'_a\cdot(\|z-x\|^{r+1}+\|z-x\|^{r^2}+\|z-y\|^{r+1}+\|z-y\|^{r^2}). 
\end{split}
\end{equation}
As $\varphi=\varphi_0$, the equality
$$
T_x^r\Phi(\z)-\varphi(x)=\sum_{1\leq|\beta|\leq r}\frac{\varphi_\alpha(x)}{\alpha!}(\z-x)^\beta
$$ 
holds for each $a\in M$. Thus, there exists a constant $L_a>0$ such that for each $z\in\R^m$ and $x$ close to $a$,
\begin{equation}\label{paso2}
\|T_x^r\Phi(z)-\varphi(x)\|\leq\sum_{1\leq|\beta|\leq r}\frac{\|\varphi_\alpha(x)\|}{\alpha!}\cdot|(z-x)^\beta|\leq L_a\sum_{1\leq j\leq r}\|z-x\|^j.
\end{equation}
Hence, there exists a constant $L_a'>0$ depending on $a$ such that
\begin{equation}\label{paso3}
\|T_x^r\Phi(z)-\varphi(x)\|^r\leq L_a'\cdot\big(\|z-x\|^r+\|z-x\|^{r^2}\big).
\end{equation}
In addition, as $\varphi$ is an ${\mathcal S}^r$-map, for each $a\in M$ there exist by Proposition \ref{lips}(i) a constant $B_a>0$ such that $|\varphi(x)-\varphi(y)|\leq B_a\cdot\|x-y\|$ for $x,y\in M$ close to $a$. By inequalities \eqref{ineq2}, \eqref{paso2} and \eqref{paso3}
\begin{multline}\label{ineq3}
|T_x^rG(z)-T_y^rG(z)|\leq\alpha_a(B_a\|x-y\|)\cdot L_a'\cdot\big(\|z-x\|^r+\|z-x\|^{r^2}+\|z-y\|^r+\|z-y\|^{r^2}\big)\\
+K'_a\cdot\big(\|z-x\|^{r+1}+\|z-x\|^{r^2}+\|z-y\|^{r+1}+\|z-y\|^{r^2}\big).
\end{multline}

Let us prove: \em For each point $a\in M$ and each $\beta\in\N^m$ with $|\beta|\leq r$ it holds $|R_x^{r-|\beta|}G_\beta(y)|=o(\|x-y\|^{r-|\beta|})$ for $x,y\in M$ when $x,y\to a$\em.

We follow the strategy employed in the proof of implication (2.2.3) $\Longrightarrow$ (2.2.2) of \cite[I.Thm. 2.2]{m} with the suitable modifications. Using \cite[I.(1.5)\&(1.6)]{m} we have
$$
T_x^rG(z)-T_y^rG(z)=\sum_{|\beta|\leq r}\frac{(z-x)^\beta}{\beta!}R_y^rG_\beta(x).
$$
By \eqref{ineq3}
\begin{equation}\label{ineq4}
\begin{split}
|T_x^rG(z)-T_y^rG(z)|&=\Big|\sum_{|\beta|\leq r}\frac{(z-x)^\beta}{\beta!}R_y^rG_\beta(x)\Big|\\
&\leq\alpha_a(B_a\cdot\|x-y\|)\cdot L_a'\cdot\big(\|z-x\|^r+\|z-x\|^{r^2}+\|z-y\|^r+\|z-y\|^{r^2}\big)\\
&+K'_a\cdot\big(\|z-x\|^{r+1}+\|z-x\|^{r^2}+\|z-y\|^{r+1}+\|z-y\|^{r^2}\big).
\end{split}
\end{equation}
Assume $x\neq y$ and write $\lambda:=\|x-y\|$. Define 
$$
z':=x+\frac{z-x}{\lambda},
$$
which satisfies $z-x=\lambda(z'-x)$. Consequently, 
\begin{equation}\label{lambda}
\|z-y\|\leq\|z-x\|+\lambda=\lambda(\|z'-x\|+1).
\end{equation} 
It is straightforward to check by induction that for each positive integer $\ell\geq 1$ and each real number $\varepsilon>0$ we have $(1+\varepsilon)^\ell \leq 2^\ell(1+\varepsilon^\ell)$. Thus, if $\lambda<1$, there exists a constant $C_1>0$ such that
\begin{equation}\label{constant2}
\begin{split}
\|z-x\|^r&+\|z-x\|^{r^2}+\|z-y\|^r+\|z-y\|^{r^2}\\
&=\lambda^r\|z'-x\|^r+\lambda^{r^2}\|z'-x\|^{r^2}+\lambda^r(1+\|z'-x\|)^r+\lambda^{r^2}(1+\|z'-x\|)^{r^2}\\
&\leq\lambda^r(\|z'-x\|^r+\|z'-x\|^{r^2}+(1+\|z'-x\|)^r+(1+\|z'-x\|)^{r^2})\\&\leq4\lambda^r(1+\|z'-x\|)^{r^2}\leq C_1\lambda^r(1+\|z'-x\|^{r^2}).
\end{split}
\end{equation}
Analogously, there exists a constant $C_2>0$ such that
\begin{equation}\label{constant3}
\|z-x\|^{r+1}+\|z-x\|^{r^2}+\|z-y\|^{r+1}+\|z-y\|^{r^2}\leq C_2\lambda^{r+1}(1+\|z'-x\|^{r^2}). 
\end{equation}
As $x$ is close to $y$, we may assume $\lambda<1$ and from \eqref{ineq4} to \eqref{constant3} there exists $C_a>0$ such that
\begin{equation}\label{ineq5}
\begin{split}
\Big|\sum_{|\beta|\leq r}\frac{\lambda^{|\beta|}}{\beta!}(z'&-x)^\beta R_y^rG_\beta(x)\Big|\\
&\leq C_a\cdot(\alpha_a(B_a\lambda)\lambda^r\big(1+\|z'-x\|^{r^2}\big)+\lambda^{r+1}\big(1+\|z'-x\|^{r^2}\big))\\
&=C_a\cdot(\alpha_a(B_a\lambda)+\lambda)\lambda^r\big(1+\|z'-x\|^{r^2}\big).
\end{split}
\end{equation}

We claim: \em There exists a constant $C_a'>0$ such that
$\Big|\frac{\lambda^{|\beta|}}{\beta!}R_y^rG_\beta(x)\Big|\leq C_a'\cdot(\alpha_a(B_a\lambda)+\lambda)\lambda^r$\em. 

Assume the claim proved for a while. We have
$$
|R_y^rG_\beta(x)|\leq\beta!\cdot C_a'\cdot(\alpha_a(B_a\|x-y\|)+\|x-y\|)\|x-y\|^{r-|\beta|}
$$
and consequently $|R_y^rG_\beta(x)|=o(\|x-y\|^{r-|\beta|})$ if $|\beta|\leq r$, $x,y\in M$ and $x,y\to a$. Thus, $g$ is an ${\mathcal S}^r$-function, which proves the statement. Let us finish the proof by showing the remaining claim. 

Fix $x$ and $y$ and treat the sum on the left of \eqref{ineq5} as a polynomial $H$ in $z'-x$. Denote 
$$
H(z'-x)=\sum_{|\beta|\leq r} a_\beta(z'-x)^\beta
$$
where $a_\beta=\frac{\lambda^{|\beta|}}{\beta!} R_y^rG_\beta(x)$. Fix enough points $p_1,\ldots,p_\ell\in\R^m$ such that the linear system
$$
\begin{cases}
H(p_1)=b_1,\\
\hspace*{5mm}\vdots\\
H(p_\ell)=b_\ell,
\end{cases}
$$
allows us to compute the coefficients of $P$ as a solution of a square compatible system (same number of equations and variables) with unique solution. The points $p_1,\ldots,p_\ell$ are fixed and do not depend on $x$, whereas the values $b_1,\ldots,b_\ell$ depend on $x$. Let $M\in{\mathfrak M}_\ell(\R)$ be the inverse of the matrix of coefficients of that linear system. We have
$$
(a_\beta)_{|\beta|\leq r}^t=M\begin{pmatrix}
b_1\\
\vdots\\
b_\ell
\end{pmatrix}.
$$
Thus, each $a_\beta=\sum_{k=1}^\ell m_{ik}b_k$ for some $i=1,\ldots,\ell$. Consequently, by \eqref{ineq5}
$$
|a_\beta|\leq\sum_{k=1}^\ell|m_{ik}||b_k|=\sum_{k=1}^\ell|m_{ik}||H(p_k)|\leq (\alpha_a(B_a\lambda)+\lambda)\lambda^r \sum_{k=1}^\ell|m_{ik}|(1+\|p_k\|^{r^2}).
$$
To finish it is enough to take $C'_a:=\max\{\sum_{k=1}^\ell|m_{ik}|(1+\|p_k\|^{r^2}):\ i=1,\ldots,\ell\}$.
\end{proof}

Straightforward consequences of the theorem above are the following.

\begin{remarks}\label{potabs}
(i) Let $f\in{\mathcal S}^r(M)$ with empty zero set. Then $|f|\in{\mathcal S}^r(M)$ because $f(M)\subset\R\setminus\{0\}$ and $g:\R\setminus\{0\}\to\R,\,x\mapsto |x|$ is an ${\mathcal S}^r$-function. 

(ii) Let $f\in{\mathcal S}^r(M)$ and ${\ell}\geq r+1$. Then $|f|^{\ell}\in{\mathcal S}^r(M)$ because $\R\to\R,\,t\mapsto |t|^{\ell}$ is an ${\mathcal S}^r$-function.

(iii) Let $M\subset\R^m$ and $N\subset\R^n$ be semialgebraic sets and let $\varphi:M\to\R^n$ and $\psi:N\to\R^p$ be ${\mathcal S}^r$-maps such that $\varphi(M)\subset N$. Then $\psi\circ\varphi:M\to\R^p$ is an ${\mathcal S}^r$-map.

(iv) We may assume always that $M\subset\R^m$ is bounded when dealing with rings of ${\mathcal S}^{r\diam}$-functions. 

Indeed, let $N:=\varphi^{-1}(M)$, where $B$ is the open ball of radius $1$ in $\R^m$ centered at the origin and $\varphi$ is the Nash diffeomorphism
\begin{equation}\label{nshdiff}
\varphi:B\to\R^m,\,x\mapsto\frac{x}{\sqrt{1-\|x\|^2}}.
\end{equation}
It induces an $\R$-algebra isomorphism ${\mathcal S}^{r\diam}(M)\to{\mathcal S}^{r\diam}(N),\,f\mapsto f\circ\varphi$.\qed
\end{remarks}

Next, let us show that we can represent closed semialgebraic subsets of a semialgebraic set $M\subset\R^m$ as zero-sets of ${\mathcal S}^r$-functions on $M$.

\begin{lem}\label{zero} 
If $Z$ is a closed semialgebraic subset of $M$ there exists $g\in{\mathcal S}^{r*}(M)$ such that $Z=Z(g)$.
\end{lem}
\begin{proof}
As $N:=\cl(Z)$ is a closed semialgebraic subset of $\R^m$ and $Z=M\cap N$, there exists by \cite[Prop.I.4.5]{sh} or \cite[Thm C.11]{vdD} an ${\mathcal S}^r$-function $f\in{\mathcal S}^r(\R^m)$ such that $N=Z(f)$. Thus, $Z=Z(f|_M)$ and $f|_M\in{\mathcal S}^r(M)$. Consequently, $g:=\frac{f|_M}{1+(f|_M)^2}\in{\mathcal S}^{r*}(M)$ and $Z=Z(g)$.
\end{proof}

The following is a direct consequence.

\begin{cor}[Urysohn's separation]\label{separation}
Let $M_1,M_2\subset M$ be closed and disjoint semialgebraic subsets of $M$. Then there exists $f\in{\mathcal S}^{r*}(M)$ such that $f|_{M_1}\equiv0$ and $f|_{M_2}\equiv1$.
\end{cor}

\begin{proof}
By Lemma \ref{zero} there exist $g,h\in{\mathcal S}^{r*}(M)$ with $M_1=Z(g)$ and $M_2=Z(h)$. As $M_1\cap M_2=\varnothing$, the sum $g^2+h^2$ never vanishes on $M$ and $f:=g^2/(g^2+h^2)\in{\mathcal S}^{r*}(M)$ satisfies the statement.
\end{proof}

As mentioned above, it follows from Theorem \ref{comp} that the ${\mathcal S}^r$-functions on a locally compact semialgebraic set $M$ are the restrictions to $M$ of ${\mathcal S}^r$-functions on open semialgebraic neighborhoods of $M$. 

\begin{lem}\label{ext} 
Let $M\subset\R^m$ be a locally compact semialgebraic set. Then $M$ is ${\mathcal S}^r$-diffeomorphic to a closed semialgebraic subset of $\R^{m+1}$. In addition, if $f$ is an ${\mathcal S}^{r\diam}$-function on $M$ and $(f_\alpha)_{|\alpha|\leq r}$ is a semialgebraic jet associated to $f$, then there exists an ${\mathcal S}^{r\diam}$-function $F$ on the open semialgebraic neighborhood $U:=\R^m\setminus(\cl(M)\setminus M)$ of $M$ such that $F|_M=f$ and $\frac{\partial^{|\alpha|}F}{\partial\x^\alpha}|_M=f_\alpha$ for $|\alpha|\leq r$.
\end{lem}
\begin{proof}
By \cite[Prop.2.7.5]{bcr} there exists $h\in{\mathcal S}(\R^m)$ such that $h|_U$ is a strictly positive Nash function and $Z(h)=\R^m\setminus U=\cl(M)\setminus M$. The image of the Nash map $\varphi:U\to\R^{m+1},\ x\mapsto\big(x,\frac{1}{h(x)}\big)$ is the closed semialgebraic set $C:=\varphi(U)=\{(x,y)\in\R^{m+1}:yh(x)=1\}$ and the projection 
$$
\pi:\R^{m+1}\to\R^m,\ (x,x_{m+1}):=(x_1,\ldots,x_m,x_{m+1})\to x
$$ 
induces by restriction a Nash map $\rho:=\pi|_{C}:C\to U$. Observe that $\varphi:U\to C$ and $\rho:C\to U$ are mutually inverse homeomorphisms. Thus, as $M$ is closed in $U$, it holds that $N:=\varphi(M)$ is closed in $C$. 

Let $f\in{\mathcal S}^{r\diam}(M)$ and note that by Theorem \ref{comp} the function $g:=f\circ\rho|_{N}$ is ${\mathcal S}^{r\diam}$ on $N$. For each $\beta:=(\beta',\beta_{m+1})\in\N^{m+1}$ with $|\beta|\leq r$ define 
\begin{equation}\label{deriv0}
g_\beta:=\begin{cases}
f_{\beta'}\circ\rho|_{N}&\text{if $\beta_{m+1}=0$,}\\
0&\text{if $\beta_{m+1}>0$.}
\end{cases}
\end{equation} 
As $\rho(N)=M$ and $(f_\alpha)_{|\alpha|\leq r}$ is a semialgebraic jet that satisfies condition \ref{diff} for $f$, one deduces $(g_\beta)_{|\beta|\leq r}$ is a semialgebraic jet that satisfies condition \ref{diff} for $g$. By Fact \ref{thfact} (combined with Corollary \ref{separation} in the bounded case) there exists an ${\mathcal S}^{r\diam}$-function $G$ on $\R^{m+1}$ such that $\frac{\partial^{|\beta|}G}{\partial\y^\beta}|_{N}=g_\beta$ if $|\beta|\leq r$ and $\y:=(\y_1,\ldots,\y_{m+1})$. The function $F:=G\circ\varphi\in{\mathcal S}^{r\diam}(U)$ satisfies $F|_M=f$. We claim: \em $\frac{\partial^{|\alpha|}F}{\partial\x^\alpha}|_{M}=f_\alpha$ if $|\alpha|\leq r$ and $\x:=(\x_1,\ldots,\x_m)$\em.

Fix $\alpha\in\N^m$ with $|\alpha|\leq r$ and denote $\xi:=\frac{1}{h}$. The reader can check inductively (using the chain rule) that
\begin{equation}\label{deriv}
\frac{\partial^{|\alpha|}F}{\partial\x^\alpha}=\frac{\partial^{|\alpha|}G}{\partial\y^{(\alpha,0)}}\circ\varphi+L,
\end{equation}
where $L$ is a (finite) linear combination of partial derivatives of $G$ of the type $\frac{\partial^{|\gamma|+k}G}{\partial\y^\gamma\partial\y_{m+1}^k}\circ\varphi$ (for some $\gamma\in\N^m$ and $k\geq1$ such that $|\gamma|+k\leq|\alpha|$) multiplied by (finite) products of partial derivatives $\frac{\partial^{|\eta|}\xi}{\partial\x^\eta}$ of $\xi$ (for some $\eta\in\N^m$ with $|\eta|\leq|\alpha|$). By \eqref{deriv0} we have $L|_{M}=0$ (because $\frac{\partial^{|\gamma|+k}G}{\partial\y^\gamma\partial\y_{m+1}^k}\circ\varphi|_M=0$ if $\gamma\in\N^m$ and $k\geq1$) and 
$$
\frac{\partial^{|\alpha|}F}{\partial\x^\alpha}|_M=\frac{\partial^{|\alpha|}G}{\partial\y^{(\alpha,0)}}\circ\varphi|_M=f_{\alpha}\circ\rho|_{N}\circ\varphi|_M=f_\alpha,
$$
as required.
\end{proof}
\begin{remark}
If $M\subset\R^m$ is a Nash manifold, then ${\mathcal S}^r(M)$ coincides with the usual ring of ${\mathcal S}^r$-functions on $M$, which are (via Nash tubular neighborhoods) the restrictions to $M$ of ${\mathcal S}^r$-functions on neighborhoods of $M$ in $\R^m$.
\end{remark}

\subsection{Rings of ${\mathcal S}^{r\diam}$-functions as direct limits}\setcounter{paragraph}{0}\label{LK}
Let $M\subset\R^m$ be a semialgebraic set an let $r\geq0$ an integer. The collection ${\mathcal S}^r(M)$ of all the ${\mathcal S}^r$-functions on $M$ is a subring of ${\mathcal S}(M)$ \em whose units are those $f\in{\mathcal S}^r(M)$ with empty zero-set\em.
 
Indeed, given $f,g\in{\mathcal S}^r(M)$ we must prove that $fg,f+g\in{\mathcal S}^r(M)$ and that $\frac{1}{f}\in{\mathcal S}^r(M)$ if $Z(f)=\varnothing$. Consider the ${\mathcal S}^r$-functions
\begin{align*}
h_1&:\R^2\to\R,\ (x,y)\mapsto xy,\\
h_2&:\R^2\to\R,\,(x,y)\mapsto x+y\\
h_3&:\R\setminus\{0\}\to\R\setminus\{0\},\ x\mapsto\tfrac{1}{x}, 
\end{align*}
and the ${\mathcal S}^r$-map $\varphi:M\to\R^2,\ x\mapsto(f(x),g(x))$. By Theorem \ref{comp}, $fg=h_1\circ\varphi$, $f+g=h_2\circ\varphi$ and $\frac{1}{f}=h_3\circ f$ are ${\mathcal S}^r$-functions (the latter if $Z(f)=\varnothing$).

The set ${\mathcal S}^{r*}(M):={\mathcal S}^r(M)\cap{\mathcal S}^{*}(M)$ of bounded ${\mathcal S}^r$-functions on $M$ is a subalgebra of ${\mathcal S}^r(M)$, that coincides with ${\mathcal S}^r(M)$ if $M$ is compact. The multiplicatively closed subset 
$$
{\mathcal W}^r(M):=\{f\in{\mathcal S}^r{^*}(M):Z(f)=\varnothing\}
$$ 
of ${\mathcal S}^r{^*}(M)$ contains $1$ but not $0$, so ${\mathcal S}^r{^*}(M)_{{\mathcal W}^r(M)}=\{f/g:\,f\in{\mathcal S}^r{^*}(M)\ \&\ g\in{\mathcal W}^r(M)\}$ is an $\R$-algebra and it coincides with ${\mathcal S}^r(M)$. The inclusion ${\mathcal S}^r{^*}(M)_{{\mathcal W}^r(M)}\subset{\mathcal S}^r(M)$ follows from the properties above, where we showed that each $g\in{\mathcal S}^r(M)$ with empty zero-set is a unit in ${\mathcal S}^r(M)$. Conversely, each $h\in{\mathcal S}^r(M)$ can be written as $h=f/g$, where $f:=h/(1+h^2)\in{\mathcal S}^r{^*}(M)$ and $g:=1/(1+h^2)\in{\mathcal W}^r(M)$.

Let us present the rings ${\mathcal S}^{r*}(M)$ and ${\mathcal S}^r(M)$ as direct limits of rings of ${\mathcal S}^r$-functions on compact and locally compact semialgebraic sets, respectively. In the sequel we prove some nice properties of the ring of ${\mathcal S}^r$-functions on a locally compact semialgebraic set and we show how some of them transfer through the direct limit (with some limitations, as we already pointed out in the Introduction).

An \emph{${\mathcal S}^r$-completion of $M$} is a pair $(E,{\tt j})$ where ${\tt j}:M\to\R^n$ is an ${\mathcal S}^r$-embedding such that $E=\cl({\tt j}(M))$. Note that $E$ is a closed semialgebraic subset of $\R^n$. If in addition each coordinate function of the map ${\tt j}$ is bounded, then $(E,{\tt j})$ is a \emph{${\mathcal S}^{r*}$-completion of $M$} (note that in this case $E$ is a \em compact \em semialgebraic set). We point out that $\mathcal{S}^{0*}$-completions were called semialgebraic compactifications in \cite{fg3}. We denote ${\mathfrak C}^r(M)$ and ${\mathfrak C}^{r*}(M)$ the collections of all ${\mathcal S}^r$ and ${\mathcal S}^{r*}$-completions of $M$ respectively. To ease notations we write ${\mathfrak C}^r$ and ${\mathfrak C}^{r*}$ hopefully without ambiguity. We will also write ${\mathfrak C}^{r\diam}$ to refer indistinctly to the previous two collections of pairs, when a result is valid for both of them.

For the elements of ${\mathfrak C}^{r\diam}$ we define $(E_1,{\tt j}_1)\preceq(E_2,{\tt j}_2)$ if there exists an ${\mathcal S}^r$-map $\rho_{21}:E_2\to E_1$ with ${\tt j}_1=\rho_{21}\circ{\tt j}_2$. As ${\tt j}_i(M)$ is dense in $E_i$, the map $\rho_{21}$ is determined by $\rho_{21}|_{{\tt j}_2(M)}={\tt j}_1\circ{\tt j}_2^{-1}:{\tt j}_2(M)\to{\tt j}_1(M)$ and $\rho_{21}(E_2)$ is dense in $E_1$. The $\R$-homomorphisms $\rho_{21}^*:{\mathcal S}^r(E_1)\to{\mathcal S}^r(E_2),\,f\mapsto f\circ\rho_{21}$ and ${\tt j}^*:{\mathcal S}^r(E)\to{\mathcal S}^r(M),\,f\mapsto f\circ{\tt j}$ are always injective. In addition, if $(E_1,{\tt j}_1)\preceq(E_2,{\tt j}_2)\preceq(E_3,{\tt j}_3)$, then $\rho_{31}=\rho_{21}\circ\rho_{32}$. We are ready to prove equation \eqref{dlim}.

\begin{lem}\label{LD}
The family $({\mathfrak C}^{r\diam},\preceq)$ is a directed set and ${\mathcal S}^{r\diam}(M)$ coincides with the direct limit $\displaystyle\lim_{\longrightarrow}{\mathcal S}^{r\diam}(E)$ where $(E,{\tt j})\in{{\mathfrak C}^{r\diam}}$.
\end{lem}
\begin{proof}
Let $(E_1,{\tt j}_1),(E_2,{\tt j}_2)\in{\mathfrak C}^{r\diam}$ with $E_i\subset\R^{n_i}$, $n:=n_1+n_2$, the ${\mathcal S}^{r\diam}$-embedding $({\tt j}_1,{\tt j}_2):M\to\R^n$ and the projections $\pi_i:E_1\times E_2\to E_i$. Define $E_3:=\cl(({\tt j}_1,{\tt j}_2)(M))\cap(E_1\times E_2)$ and let ${\tt j}_3:M\to E_3,\ x\mapsto({\tt j}_1(x),{\tt j}_2(x))$, which is an ${\mathcal S}^{r\diam}$ embedding. Then the restriction $\rho_{3i}:=\pi_i|_{E_3}$ is an ${\mathcal S}^r$-map and $\rho_{3i}\circ{\tt j}_3={\tt j}_i$. Thus $(E_3,{\tt j}_3)\in{\mathfrak C}^{r\diam}$ and $(E_i,{\tt j}_i)\preceq(E_3,{\tt j}_3)$ for $i=1,2$. Consequently, $({\mathfrak C}^{r\diam},\preceq)$ is a directed set.

To prove that ${\mathcal S}^{r\diam}(M)=\displaystyle\lim_{\longrightarrow}{\mathcal S}^{r\diam}(E)$ it is enough to show: \em For each $f\in{\mathcal S}^{r\diam}(M)$, there exist $(E,{\tt j})\in{\mathfrak C}^{r\diam}$ and $F\in{\mathcal S}^{r\diam}(E)$ such that $f=F\circ{\tt j}$\em. By Remark \ref{potabs}(iv) we may assume $M$ is bounded. Let $f\in{\mathcal S}^{r\diam}(M)$ and consider the closure $E$ in $\R^{m+1}$ of the graph of $f$. Notice that ${\tt j}:M\to\R^{m+1},\ x\mapsto(x,f(x))$ is an ${\mathcal S}^r$-embedding with $E=\cl({\tt j}(M))$. The function $F:=\pi|_E$, where $\pi:\R^{m+1}\to\R,\ (x,x_{m+1})\mapsto x_{m+1}$, belongs to ${\mathcal S}^{r\diam}(E)$ and $f=F\circ{\tt j}$, as required.
\end{proof}

\subsection{Comparison between ${\mathcal S}^0(M)$ and ${\mathcal S}(M)$}\label{s2:1}

In \cite[Cor.6]{fe2} it is proved that if $M$ is a $2$-dimensional semialgebraic set such that the germ $M_x$ is connected for each $x\in\cl(M)$, then ${\mathcal S}(M)={\mathcal S}^0(M)$. The following result is the $2$-dimensional version of Theorem \ref{main2} and generalizes \cite[Cor.6]{fe2}.

\begin{prop}\label{main2:2dim}
Let $M\subset\R^m$ be a $2$-dimensional semialgebraic set. If $M$ is non-problematic then ${\mathcal S}^\diam(M)={\mathcal S}^{0\diam}(M)$. Conversely, if the map $\varphi:\Spec^\diam(M)\to\Spec^{0\diam}(M),\ \gtp\mapsto\gtp\cap{\mathcal S}^{0\diam}(M)$ is injective then $M$ is non-problematic. 
\end{prop}
\begin{proof}
We give the proof for the ${\mathcal S}$ case, the bounded one is analogous. Suppose that $M$ is not problematic. Then for each $x\in M$ there exists an open ball $B$ of $\R^m$ such that for each $y\in\cl(M)\cap B$ the germ $M_y$ is connected. Let $f\in{\mathcal S}(M)$. By \cite[Cor.6]{fe2} there exists $\veps>0$ such that for $N^x:=\cl(M)\cap B(x,\veps)$ there exists an ${\mathcal S}$-function $F^x:N^x\to\R$ that extends $f|_{M\cap B(x,\veps)}$. As $B(x,\veps)$ is an open semialgebraic set, $M\cap B(x,\veps)$ is dense in $N^x$. Therefore, the extension $F^x$ is unique. Hence, for each $x\in M$ there exists $\varepsilon>0$ such that for each $z\in N^x:=\cl(M)\cap B(x,\veps)$ there exists a unique $t_z\in\R$ such that the tuple $(z,t_z)$ belongs to the closure of the graph of $f|_{M\cap B(x,\veps)}$ and the function $F^x:N^x\rightarrow \R,\ z\mapsto t_z$ that defines this assignment is continuous. As the latter is a first order statement, we can choose $\veps$ semialgebraically uniform on $x$. Thus, the set $N:=\bigcup_{x\in M} N^x$ is an open semialgebraic neighborhood of $M$ in $\cl(M)$ and the function $F:N\to\R$ given by $F(y)=F^x(y)$ if $y\in N^x$ is an ${\mathcal S}$-extension of $f$. 

Conversely, assume $\varphi$ is injective and suppose that $M$ is problematic at a point $p\in M$. In particular, $M$ is non-locally compact at $p$. Without loss we may assume $M$ is bounded. Denote $X:=\cl(M)$, which is a compact set. Let $(K,\Phi)$ be a semialgebraic triangulation of $X$ compatible with $M$ and $\{p\}$, that is, $K$ is a finite simplicial complex and $\Phi:|K|\to X$ is a semialgebraic homeomorphism such that both $M$ and $\{p\}$ are union of images of open simplices. To ease notation we identify $X$ with $|K|$ and the involved objects $M$ and $\{p\}$ with their inverse images under $\Phi$. 

As $M$ is problematic at $p$, there exists a sequence $\{x_k\}_{k\geq1}\subset X\setminus M$ converging to $p$ such that the germ $M_{x_k}$ is disconnected. As $M$ is not locally compact at $p$, we deduce $\dim(X\setminus M)=1$. Thus, there exist two $2$-simplices $\sigma_1,\sigma_2\in K$ such that $\sigma_j^0\subset M$ for $j=1,2$, they have a common face $\tau$, which is a $1$-simplex and satisfies $\tau^0\subset X\setminus M$ and $p$ is a vertex of $\tau$.

Consider the closed semialgebraic set $T_j:=M\cap\sigma_j\subset M$ for $j=1,2$. Define $\gtp_j$ as the collection of all $f\in{\mathcal S}(M)$ such that there exists a semialgebraic neighborhood $U$ of $p$ in $\tau$ and an ${\mathcal S}$-extension $F:T_j\cup U\to\R$ of $f|_{T_j}$ satisfying $F|_U=0$.

We claim: \em$\gtp_j\in\Spec (M)$ for $j=1,2$\em.

Fix $j=1,2$. Let $R$ be a real closed field extension of $\R$ such that there exists a positive element $\epsilon\in R$ with $\epsilon<x$ for each positive $x\in\R$. Let us construct a homomorphism $\phi_j:{\mathcal S}(M)\to R$ whose kernel is $\gtp_j$. Denote $\tau_R$ the $1$-simplex in $R^m$ defined by the vertices of $\tau$. Pick a point $p_0\in\tau_R$ such that $\|p-p_0\|<\epsilon$. By \cite[Cor.6]{fe2} ${\mathcal S}(T_j)={\mathcal S}^0(T_j)$. Thus, for each ${\mathcal S}$-function $f:T_j\to\R$ there exists an open semialgebraic neighborhood $N$ of $T_j$ in $\cl(T_j)$ and an ${\mathcal S}$-extension $F:N\to\R$ of $f$. Consider the realization of $N$ and $F$ in $R$, which we denote by $N_R$ and $F_R:N_R\to R$ respectively. As $p_0\in N_R$, we can consider the evaluation homomorphism
$$
\psi_j:{\mathcal S}(T_j)\to R,\ f\mapsto F_R(p_0).
$$
The homomorphism $\psi_j$ is well-defined. If we pick another ${\mathcal S}$-extension of $f$, then both ${\mathcal S}$-extensions coincide in a neighborhood of $p$ because $M$ is dense in $\cl(M)$. As $T_j$ is closed in $M$, the restriction map ${\mathcal S}(M)\to{\mathcal S}(T_j)$ is surjective \cite{dk1}, so the kernel of the homomorphism 
$$
\phi_j:{\mathcal S}(M)\to R,\ f\mapsto\psi_j(f|_{T_j})
$$ 
is exactly $\gtp_j$. This is so because if an ${\mathcal S}$-extension $F$ of $f|_{T_j}$ does not vanish on a semialgebraic neighborhood of $p$ in $\tau$, then (by semialgebraicity) there exists a real $\varepsilon>0$ such that for each $z\in\tau$ with $\|p-z\|<\varepsilon$ we have that $F(z)\neq 0$. As this is a first order statement we get that $F_R(p_0)\neq 0$.

We claim: $\gtp_1\neq\gtp_2$. 

Denote $v_0:=p,v_1,v_2\in\sigma_1$ the vertices of $\sigma_1$. Each point $x\in T_1$ can be written as $x=t_0v_0+t_1v_1+t_2v_2$ for some $t_0,t_1,t_2\in\R$ such that $t_0+t_1+t_2=1$. Consider the ${\mathcal S}$-function
$$
f_1:T_1\to\R, x\mapsto 1-t_0.
$$
By Urysohn's separation Lemma \ref{separation} there exists an ${\mathcal S}$-function $f_2:T_2\to\R$ such that $f_2=0$ on an open semialgebraic neighborhood of $p$ and $f_2=1$ on an open semialgebraic neighborhood of the other vertex $v$ of $\tau$. As $T_1\cap T_2\subseteq\{p,v\}$, the function $f_0: T_1\cup T_2\to\R$ given by
$$
f_0(x):=\begin{cases}
f_1(x)&\text{if $x\in T_1$},\\
f_2(x)&\text{if $x\in T_2$}
\end{cases}
$$
is an ${\mathcal S}$-function. As $T_1\cup T_2$ is closed in $M$, there exists an ${\mathcal S}$-function $f:M\to\R$ such that $f|_{T_1\cup T_2}=f_0$. We have $f\in\gtp_2\setminus\gtp_1$.

Finally, note that $g\in\gtp_j\cap{\mathcal S}^0(M)$ if and only if there exists an open semialgebraic neighborhood $N$ of $M$ in $\cl(M)$ and an extension $G\in{\mathcal S}^0(N)$ of $g$ such that $g|_{\tau\cap N}=0$. Thus, $\gtp_1\cap{\mathcal S}^0(M)=\gtp_2\cap{\mathcal S}^0(M)$, so $\varphi$ is not injective, which is a contradiction. Consequently, $M$ is non-problematic, as required.
\end{proof}

The idea to prove Theorem \ref{main2} is to reduce the general statement to the $2$-dimensional case. To that aim, we recall from \cite[Prop. 3.2]{fe1} the following decomposition of $M$ as an irredundant finite union of closed pure dimensional semialgebraic subsets of $M$ as well as some of its main properties. There exists a unique finite family $\{M_1,\ldots,M_r\}$ of semialgebraic subsets of $M$ satisfying the following properties: 
\begin{itemize}
\item[(i)] Each $M_i$ is the closure in $M$ of the set of points of $M$ whose local dimension is equal to some fixed value. In particular, $M_i$ is pure dimensional and closed in $M$. 
\item[(ii)] $M=\bigcup_{i=1}^rM_i$.
\item[(iii)] $M_i\setminus\bigcup_{j\neq i}M_j$ is dense in $M_i$.
\item[(iv)] $\dim(M_i)>\dim(M_{i+1})$ for $i=1,\ldots,r-1$. In particular, $\dim(M_1)=\dim(M)$.
\end{itemize}

\noindent We call the sets $M_i$ the \em bricks \em of $M$ and $\{M_1,\ldots,M_r\}$ constitute the \em family of bricks \em of $M$.

\begin{proof}[Proof of Theorem \em\ref{main2}] 
We provide the proof for the ${\mathcal S}$ case. The ${\mathcal S}^*$ case is analogous and we omit the proof.

To show that (iii) implies (i), pick $f\in{\mathcal S}(M)$ and let us show: $f\in{\mathcal S}^0(M)$. 

If $M$ is locally compact, there is nothing to prove as $M$ is open in its closure. Assume $M$ is non-locally compact. Denote $M_i$ the $2$-dimensional brick of $M$ and $M':=\bigcup_{j\neq i}M_j$. As $M_i$ and $M'$ are closed in $M$, we have $\cl(M_i)\cap M'=M_i\cap M'$ and $M_i\cap\cl(M')=M_i\cap M'$. By hypothesis $M_i$ is non-locally compact and it is not problematic. By Proposition \ref{main2:2dim} there exists an open semialgebraic neighborhood $N_i$ of $M_i$ in $\cl(M_i)$ and an extension $F_i\in{\mathcal S}^0(N_i)$ of $f|_{M_i}$. We may assume: \em $N_i$ satisfies $N_i\cap\cl(M')=M_i\cap M'$\em. 

Indeed, it is enough to replace $N_i$ by $\widetilde{N}_i:=N_i\setminus (\cl(M')\setminus M')$. As $M'$ is locally compact, the set $\cl(M')\setminus M'$ is closed in $\R^m$, so $\widetilde{N}_i$ is open in $\cl(M_i)$. 

In addition, $\widetilde{N}_i$ contains $M_i$. If $x\in M_i$ belongs to $\cl(M')\setminus M'$, then $x\in\cl(M')\cap M_i=M'\cap M_i\subset M'$,
which is a contradiction.

On the other hand,
$$
M_i\cap M'\subset\widetilde{N}_i\cap\cl(M')=\widetilde{N}_i\cap M'\subset\cl(M_i)\cap M'=M_i\cap M',
$$
so $\widetilde{N}_i\cap\cl(M')=M_i\cap M'$.

Next, consider the semialgebraic set $N:=M'\cup N_i$. We have
$$
\cl(N_i)\cap N=(\cl(N_i)\cap M')\cup N_i\subset(\cl(M_i)\cap M')\cup N_i=(M_i\cap M')\cup N_i=N_i,
$$
so $N_i$ is closed in $N$. In addition, 
$$
\cl(M')\cap N=M'\cup (\cl(M')\cap N_i)=M'\cup (M'\cap M_i)=M',
$$
so $M'$ is also closed in $N$. Hence,
$$
F:N\to\R,\ x\mapsto\begin{cases}
f(x)&\text{if $x\in M'$,}\\
F_i(x)&\text{if $x\in N_i$}
\end{cases}
$$
is an ${\mathcal S}$-extension of $f$ to the locally compact semialgebraic set $N$, so $f\in{\mathcal S}^0(M)$.

To finish we prove (ii) implies (iii), because (i) implies (ii) is clear. Suppose first that there exists a brick $M_\ell$ of dimension $k\geq 3$ that is non-locally compact. The canonical diagram 
$$
\xymatrix{\Spec(M_\ell)\hspace{3mm}\ar@{^{}->}[r]\ar@{^{}->}^{\varphi_\ell}[d]&\Spec(M)\ar@{^{}->}^{\varphi}[d]\\
\Spec^0(M_\ell)\ar@{^{}->}[r]&\Spec^0(M)}
$$
commutes. The map $\Spec(M_\ell)\to\Spec(M)$ is injective, because $M_\ell$ is closed in $M$ and the restriction homomorphism ${\mathcal S}(M)\to{\mathcal S}(M_\ell)$ is surjective \cite{dk1}. By hypothesis $\varphi$ is injective, so $\varphi_\ell$ is also injective. Thus, we may assume $M=M_\ell$ is pure dimensional and $\dim(M)=k\geq 3$.

For simplicity, we can suppose $M$ is bounded. Pick a point $p\in M$ such that $M$ is non-locally compact at $p$. Let $(K,\Phi)$ be a semialgebraic triangulation of $X:=\cl(M)$ compatible with $M$ and $\{p\}$. We identify $X$ with $|K|$ and the involved objects $M$ and $\{p\}$ with their inverse images under $\Phi$. Let $\tau$ be a simplex of $K$ such that $p\in\tau$ and $\tau^0\subset X\setminus M$. The existence of $\tau$ is guaranteed because $M$ is non-locally compact at $p$ (otherwise, the star of $p$ in $X$ would be contained in $M$). As $M$ is pure dimensional of dimension $k$, there exists a $k$-simplex $\sigma\subset M$ such that $\tau$ is a face of $\sigma$. Denote $v_0:=p$ and take a point $v_1\in\tau^0$. Next, pick a point $v_2\in\sigma^0$ and consider the $2$-simplex $\sigma_1$ spanned by the affinely independent points $v_0,v_1,v_2$. As $\dim(\sigma)\geq 3$, there exists a point $v_3\in\sigma^0\subset M$ such that $v_3$ is not contained in the plane spanned by $v_0,v_1,v_2$. Consider the $2$-simplex $\sigma_2$ spanned by $v_0,v_1,v_3$ and note that $\sigma_1\cap\sigma_2$ is the $1$-simplex spanned by $v_0$ and $v_1$. Let us define the closed semialgebraic subset $T:=(\sigma_1\cup\sigma_2)\cap M$ of $M$, which is problematic at $p$. As $T$ is closed in $M$, the canonical map $\Spec(T)\to\Spec^0(T)$ is injective (use the same argument already used for $M_\ell$). By Proposition \ref{main2:2dim} the $2$-dimensional set $T$ is not problematic, which is a contradiction.

Hence, we may assume that the only non-locally compact brick of $M$ is the $2$-dimensional $M_{i_0}$. As before, since $M_{i_0}$ is closed in $M$, we deduce that $\varphi_{i_0}:\Spec(M_{i_0})\to\Spec^0(M_{i_0})$ is injective. By Proposition \ref{main2:2dim} $M_{i_0}$ is non-problematic, as required.
\end{proof}

\section{Zariski spectra of rings of ${\mathcal S}^{r\diam}$-functions}\label{s3}

In this chapter we compare the Zariski and maximal spectra of the ring of the ${\mathcal S}^r$-functions with that of the ${\mathcal S}^0$-functions of a semialgebraic set $M$. In Subsection \ref{LNull} we first prove the \L ojasiewicz's Nullstellensatz in the locally compact case and a weak version of the latter in the non-locally compact case. Both results will be crucial to prove Theorem \ref{main1} in Subsection \ref{s3:1} for the ${\mathcal S}^r$ case and in Subsection \ref{s3:2} for the ${\mathcal S}^{r*}$ case. 

\subsection{\L ojasiewicz's Nullstellensatz.}\label{LNull} 
First, we analyze \L ojasiewicz's Nullstellensatz for ${\mathcal S}^r$-functions. We need some preliminary results. Along this section we fix integers $r,m$ with $r\geq0$ and $m\geq 1$ and a semialgebraic set $M\subset\R^m$. For our convenience we introduce more notation. For each $f\in{\mathcal S}^r(M)$ we denote $D(f):=M\setminus Z(f)$. The following result is a generalization of \cite[Prop.2.6.4]{bcr}.

\begin{prop}\label{suboclase} 
Let $G\in{\mathcal S}^r(\R^m)$ and let $F\in{\mathcal S}^r(D(G))$. Then there exists an integer $\ell\geq 1$ such that the function
$$
H_{\ell}:\R^m\to\R,\,x\mapsto
\begin{cases}
G^{\ell}(x)F(x)&\text{if $G(x)\neq0$,}\\
0&\text{if $G(x)=0$}\\
\end{cases}
$$
belongs to ${\mathcal S}^r(\R^m)$.
\end{prop}
\begin{proof} 
Using a recursive argument it is enough to study the case $r=1$. For simplicity we will prove the existence and continuity of the partial derivative $\partialder_{{\tt e}_1}H_{\ell}$ of $H_{\ell}$ with respect to the first variable at every point $a\in Z(G)$. The case $r=0$ was already proved in \cite[Prop.2.6.4]{bcr}. We apply this to 
$$
F|_{D(G)}:D(G)\to\R\quad\text{and}\quad\partialder_{{\tt e}_1}F|_{D(G)}:D(G)\to\R.
$$
and find a positive integer $k$ such that the function $H_k$ in the statement and the function
$$
{\widehat H}_k:\R^m\to\R,\,x\mapsto
\begin{cases}
G^k(x)\partialder_{{\tt e}_1}F|_{D(G)}(x)&\text{if $G(x)\neq0$,}\\
0&\text{if $G(x)=0$}
\end{cases}
$$ 
are continuous. We claim: \em the semialgebraic function $H_{k+1}\in{\mathcal S}^1(\R^m)$\em. 

Let us show first that $\partialder_{{\tt e}_1}H_{k+1}(a)=0$ for each $a\in Z(G)$. We have 
\begin{equation*}
\begin{split}
\partialder_{{\tt e}_1}H_{k+1}(a)&=\lim_{t\to0}\frac{H_{k+1}(a+t{\tt e}_1)-H_{k+1}(a)}{t}=\lim_{t\to0}\frac{G^{k+1}(a+t{\tt e}_1)F(a+t{\tt e}_1)}{t}\\
&=\Big(\lim_{t\to0}\frac{G(a+t{\tt e}_1)-G(a)}{t}\Big)\cdot\big(\lim_{t\to0}G^k(a+t{\tt e}_1)F(a+t{\tt e}_1)\big)\\
&=\partialder_{{\tt e}_1}G(a)\cdot\big(\lim_{t\to0}H_{k}(a+t{\tt e}_1)\big)=\partialder_{{\tt e}_1}G(a)\cdot H_{k}(a)=0.
\end{split}
\end{equation*}
If $a\in D(G)$, we have
\begin{equation*}
\begin{split}
\partialder_{{\tt e}_1}H_{k+1}(a)&=(k+1)G^k(a)\partialder_{{\tt e}_1}G(a)F(a)+G^{k+1}(a)\partialder_{{\tt e}_1}F(a)\\
&=(k+1)H_{k}(a)\partialder_{{\tt e}_1}G(a)+G(a){\widehat H}_k(a),
\end{split}
\end{equation*}
so $H_{k+1}\in{\mathcal S}^1(D(G))$ because $H_k,{\widehat H}_k,d_{{\tt e}_1}G$ and $G$ are continuous on $D(G)$. As both $H_k$ and ${\widehat H}_k$ vanish on $Z(G)$, we conclude $H_{k+1}\in{\mathcal S}^1(\R^n)$, as required.
\end{proof}

We are ready to prove \L ojasiewicz's Nullstellensatz in the locally compact case for ${\mathcal S}^r$-functions.

\begin{thm}\label{null20}
Let $M\subset\R^m$ be a locally compact semialgebraic set and let $f_1,f_2\in{\mathcal S}^r(M)$ be such that $Z(f_1)\subset Z(f_2)$. Then there exist an integer $\ell\geq 1	$ and $g\in{\mathcal S}^r(M)$ such that $f_2^{\ell}=gf_1$ and $Z(f_2)=Z(g)$.
\end{thm}
\begin{proof}
By Corollary \ref{ext} we may assume $M$ is closed. By Fact \ref{thfact} there exist $F_i\in{\mathcal S}^r(\R^m)$ with $F_i|_M=f_i$ for $i=1,2$. By Lemma \ref{zero} there exists $L\in{\mathcal S}^r(\R^m)$ such that $M=Z(L)$. Then 
$$
Z(L^2+F_1^2)=Z(L^2)\cap Z(F_1^2)=M\cap Z(F_1)=Z(f_1)\subset Z(f_2)=Z(L^2+F_2^2).
$$
Therefore 
$$
\Phi:D(L^2+F_2^2)\to\R,\,x\mapsto\frac{1}{L^2(x)+F_1^2(x)}
$$
is an ${\mathcal S}^r$-function and by Proposition \ref{suboclase} there exists a positive integer $k$ such that
$$
\Psi:\R^m\to\R,\,x\mapsto
\begin{cases}
\big(L^2(x)+F_2^2(x)\big)^k\Phi(x)&\text{if $L^2(x)+F_2^2(x)\neq0$,}\\
0&\text{if $L^2(x)+F_2^2(x)=0$}\\
\end{cases}
$$
belongs to ${\mathcal S}^r(\R^m)$. As $(L^2+F_2^2)^k=\Psi\cdot(L^2+F_1^2)$ and $L|_M\equiv0$, 
$$
f_2^{2k}=(\Psi|_M)f_1^2=((\Psi|_M)f_1)f_1.
$$
Therefore, ${\ell}:=2k$ and $g:=(\Psi|_M)f_1\in{\mathcal S}^r(M)$ do the job because 
$$
Z(g)=\big(Z(L^2+F_2^2)\cap M\big)\cup Z(f_1)=Z(f_2)\cup Z(f_1)=Z(f_2),
$$
as required.
\end{proof}

\begin{remark}\label{counterloja}
The local compactness of $M$ is essential in \L ojasiewicz's Nullstellensatz (see \cite[Rmk.1.2]{fg6}).
\end{remark}

\L ojasiewicz's Nullstellensatz has a useful consequence when studying the structure of radical ideals of rings of ${{\mathcal S}^r}$ functions. A classical concept to analyze this structure is the following:	

\begin{defn}[Ideal of zeros or $z$-ideal]\label{primez}
Let $M\subset\R^m$ be a semialgebraic set. An ideal $\gta$ in ${\mathcal S}^r(M)$ is a \em $z$-ideal if the following condition holds: \em for each $g\in{\mathcal S}^r(M)$ and each $f\in\gta$ such that $Z(f)\subset Z(g)$ we have $g\in\gta$\em. 
\end{defn}

Each $z$-ideal of ${\mathcal S}^r(M)$ is a real radical ideal. In the locally compact case the converse also holds true:

\begin{lem}\label{zloc} 
Let $M\subset\R^m$ be a locally compact semialgebraic set. Each radical ideal $\gta$ in ${\mathcal S}^r(M)$ is a $z$-ideal. In particular, each prime ideal of ${\mathcal S}^r(M)$ is a $z$-ideal. 
\end{lem}
\begin{proof}
Let $g\in{\mathcal S}^r(M)$ and let $f\in\gta$ be such that $Z(f)\subset Z(g)$. By Theorem \ref{null20} there exists a positive integer $\ell$ and $h\in{\mathcal S}^r(M)$ such that $g^{\ell}=hf\in\gta$, so $g\in\gta$, as required.
\end{proof}

\begin{remark}\label{nzifd}
Again local compactness is crucial in Lemma \ref{zloc}, see \cite[(3.4.1)]{fg5}.
\end{remark}

We present next a weak version of \L ojasiewicz's Nullstellensatz in the non-locally compact case that will be crucial in the following sections.

\begin{prop}\label{crucial}
Let $M\subset\R^m$ be a semialgebraic set and $f_1,f_2\in{\mathcal S}^r(M)$. Consider the semialgebraic set $S:=\{x\in\R^m:\ (x,0)\in\cl(\Gamma(f_1))\setminus\cl(\Gamma(f_2))\}$. The following assertions are equivalent:
\begin{itemize}
\item[(i)] $M\cap\cl(S)=\varnothing$.
\item[(ii)] There exists an open semialgebraic neighborhood $N$ of $M$ in $\cl(M)$ and extensions $F_1,F_2\in{\mathcal S}^0(N)$ of $f_1,f_2$ such that $Z(F_1)\subset Z(F_2)$.
\item[(iii)] There exist an integer $\ell\geq 1$ and $G\in{\mathcal S}^0(N)$ such that $g:=G|_M\in{\mathcal S}^r(M)$, $Z(f_2)=Z(g)$ and $f_2^\ell=f_1g$.
\end{itemize}
\end{prop}
\begin{proof}
(i) $\Longrightarrow$ (ii) By Lemma \ref{lips} there exists an open semialgebraic neighborhood $N\subset\cl(M)$ of $M$ in $\cl(M)$ such that $f_i$ extends to semialgebraic functions $F_i\in{\mathcal S}^{r-1}(N)$ for $i=1,2$. As $F_i$ is a continuous extension of $f_i$ and $M\subset N\subset\cl(M)$, we have $\Gamma(F_i)\subset\cl(\Gamma(f_i))$, so $Z(F_i)\subset\{x\in\R^m:\ (x,0)\in\cl(\Gamma(f_i))\}$. We claim: $Z(F_1)\setminus Z(F_2)\subset S$.

Pick a point $x\in Z(F_1)\setminus Z(F_2)$. Assume $c:=F_2(x)>0$ and let $U\subset\R^m$ be an open semialgebraic neighborhood of $x$ such that $F_2(y)>\frac{c}{2}$ for each $y\in U\cap N$. Thus, 
$$
\Gamma(F_2)\cap\Big(U\times\Big(-\frac{c}{3},\frac{c}{3}\Big)\Big)=\varnothing,
$$
so $(x,0)=(x,F_1(x))\in\cl(\Gamma(f_1))\setminus\cl(\Gamma(f_2))$, as claimed.

By hypothesis $M\cap\cl(Z(F_1)\setminus Z(F_2))=\varnothing$, so substituting $N$ by $N\setminus\cl(Z(F_1)\setminus Z(F_2))$, $F_1$ by $F_1|_N$ and $F_2$ by $F_2|_N$, we have $Z(F_1)\subset Z(F_2)$. 

(ii) $\Longrightarrow$ (iii) Consider the ${\mathcal S}^0$-map ${\tt j}:N\to\R^{m+2},\ x\mapsto(x,F_1(x),F_2(x))$. Observe that ${\tt j}(N)$ is locally compact because ${\tt j}:N\to{\tt j}(N)$ is a semialgebraic homeomorphism. Consider the ${\mathcal S}^r$-functions
\begin{align*}
A_1:{\tt j}(N)\to\R,\ (x,x_{n+1},x_{n+2})\to x_{n+1},\\
A_2:{\tt j}(N)\to\R,\ (x,x_{n+1},x_{n+2})\to x_{n+2}.
\end{align*}
We have 
$$
Z(A_1)={\tt j}(Z(F_1))\subset{\tt j}(Z(F_2))=Z(A_2).
$$
By Theorem \ref{null20} there exists $\ell\geq1$ and $G_0\in{\mathcal S}^r({\tt j}(N))$ such that $A_2^\ell=A_1G_0$ and $Z(A_2)=Z(G_0)$. As ${\tt j}|_M:M\to\R^{m+2}$ is an ${\mathcal S}^r$-map, $g:=G_0\circ{\tt j}|_M\in{\mathcal S}^r(M)$ and
$$
f_2^\ell=(A_2\circ{\tt j}|_M)^\ell=(A_1\circ{\tt j}|_M)(G_0\circ{\tt j}|_M)=f_1g.
$$
Observe that $Z(f_2)=Z(g)$ and $G:=G_0\circ{\tt j}\in{\mathcal S}^0(N)$.

(iii) $\Longrightarrow$ (i) By Lemma \ref{lips} we may assume that there exists an open semialgebraic neighborhood $N$ of $M$ in $\cl(M)$ such that $f_1,f_2,g$ have extensions $F_1,F_2,G\in{\mathcal S}^{r-1}(N)$. As $M$ is dense in $N$, we have by continuity $F_2^\ell=F_1G$, so $Z(F_1)\subset Z(F_2)$. Consider the closed semialgebraic set $C:=\cl(M)\setminus N$ and let us show: $S\subset C$. This finishes the proof because $M\cap\cl(S)\subset M\cap C=\varnothing$. 

Pick $x\in S\subset\cl(M)$ and suppose that $x\in N$. As $(x,0)\in\cl(\Gamma(f_1))\setminus\cl(\Gamma(f_2))$ and $x\in N$, we have $F_1(x)=0$ and $F_2(x)\neq0$, which is a contradiction because $Z(F_1)\subset Z(F_2)$.
\end{proof}

\subsection{Zariski and maximal spectra of rings of ${\mathcal S}^r$-functions}\label{s3:1}
 
In this section we prove Theorem \ref{main1} in the ${\mathcal S}^r$ case and provide some relevant consequences. We would like to stress that for the locally compact case, the result is a straightforward consequence of \cite{cc}. Indeed, by the \L ojasiewicz's Nullstellensatz and Lemma \ref{zero}, the distributive lattice of quasi-compact open subsets of $\Spec^r(M)$ is isomorphic to the distributive lattice of open semialgebraic subsets of $M$. As the distributive lattice of quasi-compact open subsets of $\Spec^0(M)$ is also isomorphic to the latter, we get that $\Sper^r(M)$ and $\Sper^0(M)$ are homeomorphic. 

The previous argument cannot be carried out in the general case (neither in the locally compact case when dealing with bounded functions).

\begin{proof}[Proof of Theorem \em\ref{main1} \em in the ${\mathcal S}^r$ case] 
Let us prove first that the map
$$
\varphi:\Spec^0(M)\to\Spec^r(M),\ \gtp\mapsto\gtp\cap{\mathcal S}^r(M)
$$ 
is bijective. Let $\gtp_1,\gtp_2\in\Spec^0(M)$ be distinct prime ideals. We may assume that there exists $f\in\gtp_2\setminus\gtp_1$. Let $N\subset\cl(M)$ be an open semialgebraic neighborhood of $M$ such that $f$ extends to a semialgebraic function $F$ on $N$. Note that $N$ is locally compact. By Lemma \ref{zero} there exists $G\in{\mathcal S}^r(N)$ such that $Z(F)=Z(G)$ and denote $g:=G|_M$. By Proposition \ref{crucial} there exist $\ell_1,\ell_2\geq 1$ and $h_1,h_2\in{\mathcal S}(M)$ such that $g^{\ell_1}=f h_1$ and $f^{\ell_2}=gh_2$. Consequently, $g\in(\gtp_2\cap{\mathcal S}^r(M))\setminus(\gtp_1\cap{\mathcal S}^r(M))$ and $\varphi(\gtp_1)\neq\varphi(\gtp_2)$.

Let $\gtq$ be a prime ideal of ${\mathcal S}^r(M)$. Let $\gtp$ be the set of all $f\in{\mathcal S}^0(M)$ satisfying the following property: \em there exist $g\in\gtq$, an open semialgebraic neighborhood $N$ of $M$ in $\cl(M)$ and extensions $F,G\in{\mathcal S}^0(N)$ of $f,g$ such that $Z(G)\subset Z(F)$\em. We claim: \em $\gtp$ is a prime ideal of ${\mathcal S}^0(M)$ equal to $\sqrt{\gtq{\mathcal S}^0(M)}$ and $\gtp\cap{\mathcal S}^r(M)=\gtq$\em.

It is straightforward to check that $\gtp$ is an ideal, so let us show that it is prime. Let $f_1,f_2\in{\mathcal S}^0(M)$ be such that $f_1f_2\in\gtp$ and let us check that either $f_1\in\gtp$ or $f_2\in\gtp$. Let $N\subset\cl(M)$ be an open semialgebraic neighborhood of $M$ such that there exists extensions $F_1,F_2\in{\mathcal S}^{r-1}(N)$ of $f_1,f_2$ (see Lemma \ref{lips}). By the definition of $\gtp$ we may assume that there exists $H\in{\mathcal S}^0(N)$ such that $Z(H)\subset Z(F_1F_2)$ and $h:=H|_M\in\gtq\subset{\mathcal S}^r(M)$. Let $G_1,G_2\in{\mathcal S}^r(N)$ be such that $Z(F_i)=Z(G_i)$ for $i=1,2$. Thus,
$$
Z(H)\subset Z(F_1F_2)=Z(F_1)\cup Z(F_2)=Z(G_1)\cup Z(G_2)=Z(G_1G_2).
$$
By Proposition \ref{crucial} there exists $\ell\geq1$ and $a\in{\mathcal S}^r(M)$ such that $((G_1G_2)|_M)^\ell=ha\in\gtq$. As $\gtq$ is prime, we may assume $g_1:=G_1|_M\in\gtq$. As $Z(F_1)=Z(G_1)$, we conclude $f_1\in\gtp$. 

Now, let us prove: $\gtp=\sqrt{\gtq{\mathcal S}^0(M)}$. 

Pick $f\in\sqrt{\gtq{\mathcal S}^0(M)}$. Then there exist $\ell\geq 1$, $g_1,\ldots,g_s\in\gtq$ and $h_1,\ldots,h_s\in{\mathcal S}^0(M)$ such that $f^\ell=g_1h_1+\cdots+g_sh_s$. Let $N$ be an open semialgebraic neighborhood of $M$ in $\cl(M)$ such that there exist extensions $F,G_i,H_i\in{\mathcal S}^0(N)$ of $f,g_i,h_i$ respectively for each $i=1,\ldots,s$. As $M$ is dense in $N$, it holds $F^\ell=G_1H_1+\cdots+G_sH_s$ and define $G:=G^2_1+\cdots+G_s^2$. We have $G|_M\in\gtq$ and $Z(G)\subset Z(F^\ell)$, so $f^\ell\in\gtp$ and $f\in\gtp$. To prove the converse inclusion pick $f\in\gtp$. There exist $g\in\gtq$ and an open semialgebraic neighborhood $N\subset\cl(M)$ of $M$ and extensions $F,G\in{\mathcal S}^0(N)$ of $f,g$ such that $Z(G)\subset Z(F)$. By Theorem \ref{null20} there exist $\ell\geq1$ and $H\in{\mathcal S}^0(N)$ such that $F^\ell=GH$. Denote $h:=H|_M\in{\mathcal S}^0(M)$ and observe that $f^\ell=gh\in\gtq{\mathcal S}^0(M)$, so $f\in\sqrt{\gtq{\mathcal S}^0(M)}$.

To finish let us check: $\gtp\cap{\mathcal S}^r(M)=\gtq$. 

The inclusion right to left is clear, so let us prove the converse one. Pick $f\in\gtp\cap{\mathcal S}^r(M)$. Then there exists $g\in\gtq$, an open semialgebraic neighborhood $N\subset\cl(M)$ of $M$ and extensions $F,G\in {{\mathcal S}}^0(N)$ of $f,g$ such that $Z(G)\subset Z(F)$. By Proposition \ref{crucial} there exist $\ell\geq1$ and $h\in{\mathcal S}^r(M)$ such that $f^\ell=gh\in\gtq{\mathcal S}^r(M)=\gtq$, so $f\in\gtq$.

The map $\varphi$ is continuous because it is induced by the ring inclusion ${\tt j}:{\mathcal S}^r(M)\hookrightarrow{\mathcal S}^0(M)$. To prove that $\varphi$ is a homeomorphism it only remains to show: \em$\varphi$ is an open map\em. 

Let $\Dd:=\{\gtp\in\Spec^0(M):\ f\notin\gtp\}$ for some $f\in{\mathcal S}^0(M)$. Let $N\subset\cl(M)$ be an open semialgebraic neighborhood of $M$ such that there exists an extension $F\in{\mathcal S}^0(N)$ of $f$. By Lemma \ref{zero} there exists $G\in{\mathcal S}^r(N)$ such that $Z(F)=Z(G)$. Define $g:=G|_M\in{\mathcal S}^r(M)$ and let us prove that 
$$
\varphi(\Dd)=\{\gtq\in\Spec^r(M):g\notin\gtq\}.
$$
Note that there exist $\ell_1,\ell_2\geq1$ and $h_1,h_2\in{\mathcal S}^0(M)$ such that $f^{\ell_1}=gh_1$ and $g^{\ell_2}=fh_2$. Thus, if $\gtp\in \Dd$ then $g\notin\gtp\cap{\mathcal S}^r(M)=\varphi(\gtp)$. Conversely, let $\gtq\in\Spec^r(M)$ be such that $g\notin\gtq$ and let us show that $f\notin\sqrt{\gtq{\mathcal S}^0(M)}$. Note that $g\notin\sqrt{\gtq{\mathcal S}^0(M)}$ because $\sqrt{\gtq{\mathcal S}^0(M)}\cap{\mathcal S}^r(M)=\gtq$. Thus, $fh_2=g^{\ell_2}\not\in\sqrt{\gtq{\mathcal S}^0(M)}$, so $f\notin\sqrt{\gtq{\mathcal S}^0(M)}$, as required.
\end{proof}

\begin{remark}\label{radONE}
In the previous proof we have shown that for each semialgebraic set $M$ and $\gtq\in\Spec^r(M)$, the prime ideal $\sqrt{\gtq{\mathcal S}^0(M)}$ equals the set of functions $f\in{\mathcal S}^0(M)$ for which there exists an integer $\ell\geq 1$, $g\in\gtq$ and $h\in{\mathcal S}^0(M)$ such that $f^\ell=gh$. In addition, $\sqrt{\gtq{\mathcal S}^0(M)}$ is the set of all $f\in{\mathcal S}^0(M)$ for which there exist $g\in\gtq$ and extensions $F,G\in{\mathcal S}^0(N)$ of $f,g$ to some open semialgebraic neighborhood $N$ of $M$ in $\cl(M)$ such that $Z(G)\subset Z(F)$. 

In particular, if $\gtq$ is a $z$-ideal, then $\sqrt{\gtq{\mathcal S}^0(M)}$ is a $z$-ideal. 

Indeed, let $f_1,f_2\in{\mathcal S}^0(M)$ be such that $f_1\in\sqrt{\gtq{\mathcal S}^0(M)}$ and $Z(f_1)\subset Z(f_2)$. Then, there exist $g_1\in\gtq$ and extensions $F_1,G_1\in{\mathcal S}^0(N)$ of $f_1,g_1$ to some open semialgebraic neighborhood $N$ of $M$ in $\cl(M)$ such that $Z(G_1)\subset Z(F_1)$. We may assume that there exists an extension $F_2\in{\mathcal S}^0(N)$ of $f_2$. By Lemma \ref{zero} we can pick a function $G_2\in{\mathcal S}^r(N)$ such that $Z(G_2)=Z(F_2)$. Let us denote $g_2:=G_2|_M$. We have 
$$
Z(g_1)\subset Z(f_1)\subset Z(f_2)=Z(g_2),
$$ 
so $g_2$ belongs to the $z$-ideal $\gtq$ and $f_2\in\sqrt{\gtq{\mathcal S}^0(M)}$, as required.
\end{remark}

\subsection{Zariski and maximal spectra of rings of ${\mathcal S}^{r*}$ functions}\label{s3:2} 
\L ojasiewicz's Nullstellensatz \ref{null20} has played a crucial role in the proof of Theorem \ref{main1} in the ${\mathcal S}^r$ case. As Theorem \ref{null20} does not have a bounded counterpart, we need to develop another tool in order to give a proof of Theorem \ref{main1} in the ${\mathcal S}^{r*}$ case.

\begin{lem}\label{gorro}
Let $M\subset\R^m$ be a semialgebraic set and let $\gta$ be a (proper) ideal of ${\mathcal S}^{r*}(M)$. Then the set
$$
\widehat{\gta}:=\{f\in{\mathcal S}^{0*}(M):\ \forall\veps\in{\mathcal S}^{0*}(M),\ \veps>0\ \exists g\in\gta\ \text{such that}\ |f-g|<\veps\}
$$
is a (proper) ideal of ${\mathcal S}^{0*}(M)$ that contains $\gta$. In addition, the following properties hold:
\begin{itemize}
\item[(i)] Let $\veps\in{\mathcal S}^{0*}(M)$ be strictly positive and let $f\in{\mathcal S}^{0*}(M)\setminus\widehat{\gta}$. Then there exists $g\in{\mathcal S}^{r*}(M)\setminus\gta$ such that $|f-g|<\veps$.
\item[(ii)] If $\gta$ is a radical ideal of ${\mathcal S}^{r*}(M)$ that contains $\veps\in{\mathcal S}^{r*}(M)$ with empty zero-set, $\gta=\widehat{\gta}\cap{\mathcal S}^{r*}(M)$.
\item[(iii)] If $\gtb$ is a radical ideal of ${\mathcal S}^{0*}(M)$ that contains $\delta\in{\mathcal S}^{0*}(M)$ with empty zero-set and $\gta:=\gtb\cap{\mathcal S}^{r*}(M)$, then $\widehat{\gta}=\gtb$ and $\gta$ contains $\veps\in{\mathcal S}^{r*}(M)$ with empty zero-set.
\item[(iv)] If $\gtq$ is a prime ideal of ${\mathcal S}^{r*}(M)$ that contains $\veps\in{\mathcal S}^{r*}(M)$ with empty zero-set, $\widehat{\gtq}$ is a prime ideal of ${\mathcal S}^{0*}(M)$ and $\widehat{\gtq}=\sqrt{\gtq{\mathcal S}^{0*}(M)}$.
\end{itemize} 
\end{lem} 
\begin{proof}
It is clear that $\gta\subset\widehat{\gta}$. We prove next that $\widehat{\gta}$ is an ideal of ${\mathcal S}^{0*}(M)$. 

Let $f_1,f_2\in\widehat{\gta}$ and let $\veps\in{\mathcal S}^{0*}(M)$ be strictly positive. There exist $g_1,g_2\in\gta$ with $|f_i-g_i|<\veps/2$ for $i=1,2$, so $g:=g_1+g_2\in\gta$ and $|(f_1+f_2)-g|\leq|f_1-g_1|+|f_2-g_2|<\veps$. Thus, $f_1+f_2\in\widehat{\gta}$.

Let $f\in\widehat{\gta}$ and $a\in{\mathcal S}^{0*}(M)$. Fix $\veps\in{\mathcal S}^{0*}(M)$ strictly positive. There exist an open semialgebraic neighborhood $N\subset\cl(M)$ of $M$ and extensions $F,A,E\in{\mathcal S}^{0*}(N)$ of $f,a,\veps$. As $M\subset\{E>0\}$, we may assume after shrinking $N$ if necessary that $E$ is strictly positive on $N$. By Lemma \ref{ext} there exist an open subset $U$ of $\R^n$ such that $N$ is closed in $U$ and extensions $F',A',E'\in{\mathcal S}^{0*}(U)$ of $F,A,E$. We may assume in addition that $E'$ is strictly positive on $U$. Let $L>0$ be such that $|F'|,|A'|,|E'|<L$. By \cite[Thm.8.8.4]{bcr} there exists $H\in{\mathcal N}(U)$ such that $|H-A'|<\frac{E'}{2L}$. As $H$ is bounded by $L+\frac{1}{2}$, we deduce $h:=H|_M\in{\mathcal S}^{r*}(M)$. Since $f\in\widehat{\gta}$, there exists $g\in\gta$ such that $|f-g|<\frac{\veps}{2L+1}$. Consequently, $gh\in\gta$ and
$$
|fa-gh|\leq|f||a-h|+|h||f-g|<L\frac{\veps}{2L}+\Big(L+\frac{1}{2}\Big)\frac{\veps}{2L+1}=\veps.
$$
We conclude $fa\in\widehat{\gta}$.

If $1\in\widehat{\gta}$, there exists $g\in\gta$ such that $|g-1|<1/2$, so $1/2<g<3/2$. Thus, $g$ is a unit in ${\mathcal S}^{r*}(M)$, which is a contradiction because $\gta$ is a proper ideal and contains no unit.

We prove next the properties in the statement.

(i) As $f\in{\mathcal S}^{0*}(M)\setminus\widehat{\gta}$, there exists $\veps_0\in{\mathcal S}^{0*}(M)$ strictly positive such that for each $h\in\gta$ there exists $x_0\in M$ satisfying $|f(x_0)-h(x_0)|\geq\veps_0(x_0)$. Let $N\subset\cl(M)$ be an open semialgebraic neighborhood of $M$ and extensions $F,E_0,E\in{\mathcal S}^{0*}(N)$ of $f,\veps_0,\veps$ such that $E_0,E$ are strictly positive on $N$. By Lemma \ref{ext} there exist an open semialgebraic set $U\subset\R^m$ such that $N$ is closed in $U$ and extensions $F',E_0',E'\in{\mathcal S}^*(U)$ of $F,E_0,E$ such that $E_0'$ and $E'$ are strictly positive on $U$. Consider the strictly positive function $E_1':=\min\{E',E_0'\}/2\in{\mathcal S}^*(U)$. By \cite[Thm.8.8.4]{bcr} there exists $G\in{\mathcal N}(U)$ such that $|F'-G|<E_1'$, so in particular $G$ is bounded. Note that $g:=G|_M\in{\mathcal S}^{r*}(M)\setminus\gta$ because $0<\veps_1:=E_1|_M<\veps_0$ and $|f-g|<\veps_1$. In addition, $|f-g|<\veps_1<\veps$.

(ii) The inclusion $\gta\subset\widehat{\gta}\cap{\mathcal S}^{r*}(M)$ is clear. Taking $\veps^2$ instead of $\veps$ if necessary, we may assume $\veps>0$. Pick $h\in\widehat{\gta}\cap{\mathcal S}^{r*}(M)$ and let $g\in\gta$ be such that $|h-g|<\veps$. Thus, $(h-g)^2<\veps^2\in\gta$, so $a:=(h-g)^2/\veps^2\in{\mathcal S}^{r*}(M)$ and $h^2=g(2h-g)+\veps^2 a\in\gta$. Hence, $h\in\gta$.

(iii) We may assume $\delta>0$. We prove first $\widehat{\gta}\subset\gtb$. Given $f\in\widehat{\gta}$ there exists $g\in\gta\subset\gtb$ with $|f-g|<\delta$, so $h:=|f-g|/\delta\in{\mathcal S}^{0*}(M)$. As 
$$
f^2=(f-g)^2+g(2f-g)=h^2\delta^2+g(2f-g)\in\gtb 
$$
and $\gtb$ is a radical ideal, $f\in\gtb$. For the converse inclusion suppose that there exists $f\in\gtb\setminus\widehat{\gta}$. By part (i) there exists $g\in{\mathcal S}^{r*}(M)\setminus\gta$ with $|f-g|<\delta$. Thus, $h:=|f-g|/\delta\in{\mathcal S}^{0*}(M)$, so $|f-g|=h\delta\in\gtb$. Consequently, $g^2=|f-g|^2+2fg-f^2\in\gtb$. As $\gtb$ is a radical ideal, $g\in\gtb\cap{\mathcal S}^{r*}(M)=\gta$, which is a contradiction.

As $\delta\in\gtb=\widehat{\gta}$, there exists $g\in\gta$ with $|\delta-g|<\delta$, so $g>0$ and $Z(g)=\varnothing$.

(iv) We may assume $\veps>0$. Suppose there exist $f_1,f_2\in{\mathcal S}^{0*}(M)\setminus\widehat{\gtq}$ such that $f_1f_2\in\widehat{\gtq}$. There exist $g_1,g_2\in{\mathcal S}^{r*}(M)\setminus\gtq$ such that $|f_i-g_i|<\veps$ for $i=1,2$. Let $L>0$ be such that $|\veps|,|f_i|<L$. We have 
\begin{equation*}
|f_1f_2-g_1g_2|\leq|f_1||f_2-g_2|+|g_2||f_1-g_1|\leq L\veps+(L+\veps)\veps=(2L+\veps)\veps.
\end{equation*}
As $f_1f_2\in\widehat{\gtq}$, there exists $g\in\gtq$ with $|f_1f_2-g|<\veps$. Thus,
\begin{equation*}
|g-g_1g_2|\leq|g-f_1f_2|+|f_1f_2-g_1g_2|<(1+2L+\veps)\veps.
\end{equation*}
Define $h:=g_1g_2\in{\mathcal S}^{r*}(M)\setminus\gtq$. We have $|g-h|<(1+3L)\veps$, so 
$$
a:=\frac{h-g}{\veps}\in{\mathcal S}^{r*}(M)\quad\text{and}\quad h-g=\veps a\in\gtq.
$$
Thus, $h=g+\veps a\in\gtq$, which is a contradiction.

As $\widehat{\gtq}$ is a prime ideal that contains $\gtq$, it follows that $\sqrt{\gtq{\mathcal S}^{0*}(M)}\subset\widehat{\gtq}$. By (ii) we have 
$$
\gtq\subset\sqrt{\gtq{\mathcal S}^{0*}(M)}\cap{\mathcal S}^{r*}(M)\subset\widehat{\gtq}\cap{\mathcal S}^{r*}(M)=\gtq.
$$
By (iii) we deduce $\widehat{\gtq}=\sqrt{\gtq{\mathcal S}^{0*}(M)}$, as required.
\end{proof}

\begin{cor}\label{cor:Xbounded}
Let $M\subset\R^m$ be a semialgebraic subset. Denote ${\mathcal W}^r(M):=\{f\in{\mathcal S}^{r*}(M):\ Z(f)=\varnothing\}$ and consider the space $\mathfrak{X}^r(M):=\{\gtp\in\Spec^{r*}(M):\ \gtp\cap{\mathcal W}^r(M)\neq\varnothing\}$ for each integer $r\geq0$. Then the maps
\begin{align*}
\Psi_M&:\mathfrak{X}^0(M)\to\mathfrak{X}^r(M),\,\gtp\mapsto\gtp\cap{\mathcal S}^{r*}(M)\\
\Phi_M&:\mathfrak{X}^r(M)\to\mathfrak{X}^0(M),\,\gtq\mapsto\widehat{\gtq}=\sqrt{\gtq{\mathcal S}^{0*}(M)}
\end{align*}
are mutually inverse.
\end{cor}
\begin{proof}
By Lemma \ref{gorro} (iii) and (iv) we deduce $\Phi_M\circ\Psi_M(\gtp)=\gtp$ for each $\gtp\in\mathfrak{X}^0(M)$, whereas $\Psi_M\circ\Phi_M(\gtq)=\gtq$ for each $\gtq\in\mathfrak{X}^r(M)$ by Lemma \ref{gorro} (ii).
\end{proof}

We are ready to complete the proof of Theorem \ref{main1}.

\begin{proof}[Proof of Theorem \em\ref{main1} \em in the ${\mathcal S}^{r*}$ case]
By Corollary \ref{cor:Xbounded} to prove that the maps $\varphi$ and $\psi$ are mutually inverse it is enough to show that the maps
\begin{align*}
\Phi:&\Spec^{0*}(M)\setminus\mathfrak{X}^0(M)\to\Spec^{r*}(M)\setminus\mathfrak{X}^r(M),\ \gtp\mapsto\gtp\cap{\mathcal S}^{r*}(M)\\
\Psi:&\Spec^{r*}(M)\setminus\mathfrak{X}^r(M)\to\Spec^{0*}(M)\setminus\mathfrak{X}^0(M),\ \gtq\mapsto\sqrt{\gtq{\mathcal S}^{0*}(M)}
\end{align*}
are mutually inverse. As ${\mathcal S}^0(M)={\mathcal S}^{0*}(M)_{{\mathcal W}^0(M)}$ and ${\mathcal S}^r(M)={\mathcal S}^{r*}(M)_{{\mathcal W}^r(M)}$, by \cite[Prop.3.11 \& Ch.3, Ex. 21]{am} the map
$$
\gamma^r:\Spec^r(M)\to\Spec^{r*}(M)\setminus\mathfrak{X}^r(M),\,\gtp\mapsto\gtp\cap{\mathcal S}^{r*}(M)
$$
is a homeomorphism whose inverse is given by 
$$
(\gamma^r)^{-1}:\Spec^{r*}(M)\setminus\mathfrak{X}^r(M)\to\Spec^r(M),\,\gtq\mapsto\gtq{{\mathcal S}^r}(M). 
$$
By the version of Theorem \ref{main1} for ${\mathcal S}^r$-functions the maps
\begin{align*}
\rho&:\Spec^0(M)\to\Spec^r(M),\,\gtp\mapsto\gtp\cap{\mathcal S}^r(M)\\ 
{\tt i}&:\Spec^r(M)\to\Spec^0(M),\,\gtq\mapsto\sqrt{\gtq{\mathcal S}^0(M)}
\end{align*}
are mutually inverse homeomorphisms. Thus, the map
$$
\Phi=\gamma^r\circ{\rho}\circ(\gamma^0)^{-1}:\Spec^{0*}(M)\setminus\mathfrak{X}^0(M)\to\Spec^{r*}(M)\setminus\mathfrak{X}^r(M),\,\gtp\mapsto\gtp\cap{\mathcal S}^{r*}(M)
$$
is a homeomorphism and the diagram
$$
\xymatrix{
\Spec^r(M)\ar@{->}[rr]^(0.4){\gamma^r}&&\Spec^{r*}(M)\setminus\mathfrak{X}^r(M)\\
\Spec^0(M)\ar@{->}[rr]^(0.4){\gamma^0}\ar@{->}[u]^{\rho}&&\Spec^{0*}(M)\setminus\mathfrak{X}^0(M)\ar[u]_{\Phi}
}
$$
commutes. For each $\gtq\in\Spec^{r*}(M)\setminus\mathfrak{X}^r(M)$ we have $\Phi^{-1}(\gtq)=\sqrt{\gtq{\mathcal S}^0(M)}\cap{\mathcal S}^{0*}(M)$. Thus, to prove that $\varphi$ and $\psi$ are mutually inverse it only remains to show the equality
$$
\sqrt{\gtq{\mathcal S}^{0*}(M)}=\sqrt{\gtq{\mathcal S}^0(M)}\cap{\mathcal S}^{0*}(M).
$$
By Theorem \ref{main1} for ${\mathcal S}^r$-functions and since $\gtq{\mathcal S}^r(M)\in\Spec^r(M)$,
$$
\sqrt{\gtq{\mathcal S}^0(M)}\cap{\mathcal S}^r(M)=\gtq{\mathcal S}^r(M)\quad\leadsto\quad\sqrt{\gtq{\mathcal S}^0(M)}\cap{\mathcal S}^{r*}(M)=\gtq.
$$
Pick $f\in\sqrt{\gtq{\mathcal S}^0(M)}\cap{\mathcal S}^{0*}(M)$. Let $N\subset\cl(M)$ be an open semialgebraic neighborhood of $M$ such that there exists an extension $F\in{\mathcal S}^0(N)$ of $f$. We show: \em there exist $Q\in{\mathcal S}^{r*}(N)$ and $k\geq 1$ such that $f^{2k+2}\in\gtq{\mathcal S}^0(M)$, $F^{2k+2}\leq Q$ and $Z(Q)=Z(F)$\em. 

Indeed, by Lemma \ref{zero} and Theorem \ref{null20} there exist $k\geq1$, $G\in{{\mathcal S}^{r*}(N)}$ and $H\in{\mathcal S}^0(N)$ such that $F^k=GH$ and $Z(G)=Z(F)$. Without loss of generality we assume $f^k\in\gtq{\mathcal S}^0(M)$. Consider the open semialgebraic set $W:=\{|FH|<1\}$ of $N$ that contains $Z(F)$. As $F,G$ are bounded, there exists $L>0$ such that $F^{2k+2}<L$ and $G^2<L$. Let $\{\sigma,1-\sigma\}$ be an ${\mathcal S}^r$ partition of unity subordinated to $\{W,N\setminus Z(F)\}$ and define 
$$
Q:=G^2\sigma+L(1-\sigma)\in{\mathcal S}^{r*}(M).
$$
If $x\in W$, then $F^{2k+2}(x)=G^2(x)F^2(x)H^2(x)\leq G^2(x)$ (because $F^2(x)H^2(x)<1$), so 
$$
F^{2k+2}(x)\leq G^2(x)+(L-G^2(x))(1-\sigma(x))=Q(x).
$$
If $x\in N\setminus W$, then $\sigma(x)=0$, so $Q(x)=L\geq F(x)^{2k+2}$. In addition, $Z(Q)=Z(1-\sigma)\cap Z(G)=Z(G)=Z(F)$.

Again by Theorem \ref{null20} there exist $\ell\geq 1$ and $A\in{\mathcal S}^0(N)$ such that $Q^\ell=F^{2k+2}A$, so $q:=Q|_M$ satisfies $q^\ell\in\gtq{\mathcal S}^0(M)$. Thus, $q\in\sqrt{\gtq{\mathcal S}^0(M)}\cap{\mathcal S}^{r*}(M)=\gtq$. As $F^{2k+2}\leq Q$ and $Z(F)=Z(Q)$, we have $B:=\frac{F^{2k+3}}{Q}\in{\mathcal S}^{0*}(N)$. Consequently, $f^{2k+3}=qB|_M\in\gtq{\mathcal S}^{0*}(M)$ and $f\in\sqrt{\gtq{\mathcal S}^{0*}(M)}$. We have proved $\sqrt{\gtq{\mathcal S}^0(M)}\cap{\mathcal S}^{0*}(M)\subset\sqrt{\gtq{\mathcal S}^{0*}(M)}$ and the converse inclusion is straightforward. Thus, $\sqrt{\gtq{\mathcal S}^0(M)}\cap{\mathcal S}^{0*}(M)=\sqrt{\gtq{\mathcal S}^{0*}(M)}$, as stated.

As in the ${\mathcal S}^r$ case, the map $\varphi$ is continuous because it is induced by the ring inclusion ${\tt j}:{\mathcal S}^{r*}(M)\hookrightarrow{\mathcal S}^{0*}(M)$. To prove that $\varphi$ is a homeomorphism it only remains to show: \em$\varphi$ is an open map\em.

Let $f\in{\mathcal S}^{0*}(M)$ and $\Dd^{0*}:=\{\gtp\in\Spec^{0*}(M):\ f\not\in\gtp\}$. Let us prove that $\varphi(\Dd^{0*})$ is an open subset of $\Spec^{r*}(M)$. Pick $\gtq_0\in\varphi(\Dd^{0*})$ and let $\widehat{\gtq}_0=\varphi^{-1}(\gtq_0)\in\Dd^{0*}$. We claim: \em there exists $g\in{\mathcal S}^{r*}(M)\setminus\gtq_0$ such that $Z(g)=Z(f)$ and $|g|\leq f^2+f^{2k}$ for some $k\geq1$\em.

Assume the previous claim proved for a while and let us check that 
$$
\gtq_0\in\Dd^{r*}:=\{\gtq\in\Spec^{r*}(M):\ g\not\in\gtq\}\subset\varphi(\Dd^{0*}).
$$
As $g\in{\mathcal S}^{r*}(M)\setminus\gtq_0$, we have $\gtq_0\in\Dd^{r*}$. Let $\gtq\in\Dd^{r*}$ and assume $\widehat{\gtq}\not\in\Dd^{0*}$, that is, $f\in\widehat{\gtq}$. As $|g|\leq f^2+f^{2k}$ and $Z(g)=Z(f)$, the function $h:=\frac{g}{f^2+f^{2k}}g\in{\mathcal S}^{0*}(M)$, so $g^2=h(f^2+f^{2k})\in\widehat{\gtq}\cap{\mathcal S}^{r*}(M)=\gtq$ and $g\in\gtq$, which is a contradiction. Thus, $\Dd^{r*}\subset\varphi(\Dd^{0*})$ and $\varphi(\Dd^{0*})$ is open in $\Spec^{0*}(M)$.

So let us prove our claim. Let $N\subset\cl(M)$ be an open semialgebraic neighborhood of $M$ and let $F\in{\mathcal S}^{0*}(N)$ be an extension of $f$. By Lemma \ref{zero} there exists $G_0\in{\mathcal S}^{r*}(N)$ such that $Z(F)=Z(G_0)$. By Theorem \ref{null20} there exist an even integer $\ell\geq 1$ and $H_0\in{\mathcal S}^0(N)$ such that $G_0^{2\ell}=F^2H_0\leq F^2(1+H_0^2)$. Note that $H_0$ needs not to be bounded. Take $H_1\in{\mathcal S}^r(N)$ such that $H_0^2\leq H_1^2$ and define $G_1:=\frac{G_0^\ell}{\sqrt{1+H_1^2}}\in{\mathcal S}^{r*}(N)$. We have $G_1^2\leq F^2$.

As $N$ is locally compact, it is closed in the open semialgebraic neighborhood $U:=\R^m\setminus(\cl(U)\setminus U)$. By Lemma \ref{ext} there exist extensions $F_2\in{\mathcal S}^{0*}(U)$ of $F$ and $G_2'\in{\mathcal S}^{r*}(U)$ of $\frac{G_0}{\sqrt[2\ell]{1+H_1^2}}\in{\mathcal S}^{r*}(N)$. As $N$ is a closed semialgebraic subset of $U$, there exists by Lemma \ref{zero} $G_2''\in{\mathcal S}^{r*}(U)$ whose zero set is $N$. Define $G_2:=(G_2')^\ell+(G_2'')^2\in{\mathcal S}^{r*}(U)$ and observe that it is an extension of $G_1$ such that $Z(G_2)=Z(G_1)\subset Z(F)\subset Z(F_2)$. 

By Theorem \ref{null20} there exist $k\geq1$ and $A\in{\mathcal S}^0(U)$ such that $F_2^{2k}=G_2^2A$. As $f\notin \widehat{\gtq}_0$, there is $\veps\in{\mathcal S}^{0*}(M)$ such that $0<\veps<1$ and for each $g\in\gtq_0$ there exists $x_0\in M$ satisfying $|f^{2k}(x_0)-g(x_0)|\geq\veps(x_0)$. We may assume shrinking $U$ that $\veps$ extends to a strictly positive $E\in{\mathcal S}^{0*}(U)$.

Let $A^\bullet\in{\mathcal N}(U)$ be a Nash function such that $|A-A^\bullet|<\frac{E}{L}$ where $L>1$ satisfies $G_2^2<L$. Define $G:=G_2^2A^\bullet\in{\mathcal S}^r(U)$. We have
$$
|G|\leq|G_2^2(A-A^\bullet)|+|G_2^2A|<\frac{E}{L}G_2^2+F_2^{2k}<E+F_2^{2k},
$$
so $G\in{\mathcal S}^{r*}(U)$. Denote $g_2:=G_2|_M,g:=G|_M\in{\mathcal S}^{r*}(M)$ and $a:=A|_M,a^\bullet:=A^\bullet\in{\mathcal S}^r(M)$. We have
$$
|f^{2k}-g|=|g_2^2(a-a^\bullet)|<L\frac{\veps}{L}=\veps,
$$
so $g\not\in\gtq_0$. In addition, since $g_2=G_1|_M$ and $G^2_1\leq F^2$,
$$
|f^{2k}-g|\leq g_2^2\frac{\veps}{L}\leq f^2.
$$
Consequently, 
$$
-f^2-f^{2k}\leq-f^2+f^{2k}\leq g\leq f^2+f^{2k}, 
$$
so $|g|\leq f^2+f^{2k}$ and $Z(f)=Z(G_0|_M)\subset Z(g)\subset Z(f)$, that is, $Z(g)=Z(f)$, as required.
\end{proof}

\begin{cor}\label{dim}
Let $M\subset\R^m$ be a semialgebraic set. Then
\begin{itemize}
\item[(i)] $\dim({\mathcal S}^{r\diam} (M))=\dim(M)$.
\item[(ii)] ${\mathcal S}^{r\diam}(M)$ is a Gelfand ring.
\end{itemize}
\end{cor}
\begin{proof} 
By Theorem \ref{main1} it is enough to prove the result for $r=0$. 

(i) Let $\gtp_0\subset\cdots\subset\gtp_s$ be a chain of prime ideals in ${\mathcal S}^{0\diam}(M)$. Let $f_i\in\gtp_i\setminus\gtp_{i-1}$ for $i=1,\ldots,s$. Let $N\subset\cl(M)$ be an open semialgebraic neighborhood of $M$ and let $F_i\in{\mathcal S}^{0\diam}(N)$ be an extension of $f_i$. Let $\gtq_i:=\{F\in \mathcal{S}^{0\diam}(N): F|_M\in \gtp_i\}\in \Spec^{0\diam}(N)$ and observe that $F_i\in\gtq_i\setminus\gtq_{i-1}$ for $i=1,\ldots,s$. The equality $\dim({\mathcal S}^{0\diam}(N))=\dim(N)=\dim(M)$ was proved in \cite{cc,s4,s6,fg6}. Therefore $\dim({\mathcal S}^{0\diam}(M))\leq\dim(M)$. For the converse inequality, by cylindrical decomposition of semialgebraic sets there exists a closed ball $B$ such that
$$
M\cap B=\cl(M)\cap B
$$
is compact and $\dim(M)=\dim(M\cap B)$. As $M\cap B$ is closed in $\R^m$, the restriction map ${{\mathcal S}^{0\diam}}(M)\to {{\mathcal S}^{0\diam}}(M\cap B)$ is surjective. Thus, 
$$
\dim({\mathcal S}^{0\diam}(M))\geq\dim({\mathcal S}^{0\diam}(M\cap B))=\dim({\mathcal S}^\diam(M\cap B))=\dim(M\cap B)=\dim(M).
$$

(ii) Let $\gtp$ be a prime ideal in ${\mathcal S}^{0\diam}(M)$. We show that the set of prime ideals of ${\mathcal S}^{0\diam}(M)$ that contain $\gtp$ is a spear, that is, is totally ordered by the inclusion. This implies that $\gtp$ is contained in a unique maximal ideal, so ${\mathcal S}^{r\diam}(M)$ is a Gelfand ring. 

Assume there exists prime ideals $\gtp_1$ and $\gtp_2$ containing $\gtp$ for which there exist functions $f_1\in\gtp_1\setminus\gtp_2$ and $f_2\in\gtp_2\setminus\gtp_1$. Let $N\subset\cl(M)$ be an open semialgebraic neighborhood of $M$ such that there exist extensions $F_1,F_2\in{\mathcal S}^{0\diam}(N)$ of $f_1,f_2$. Consider the prime ideals $\gtq:=\{F\in \mathcal{S}^{0\diam}(N):F|_M\in\gtp\}\in \Spec^{0\diam}(N)$ and $\gtq_i:=\{F\in \mathcal{S}^{0\diam}(N):F|_M\in\gtp_i\}\in \Spec^{0\diam}(N)$ for $i=1,2$. Observe that $\gtq\subset\gtq_1\cap\gtq_2$, $F_1\in\gtp_1\setminus\gtp_2$ and $F_2\in\gtp_2\setminus\gtp_1$. But this is a contradiction because $N$ is locally compact and therefore the set of prime ideals of ${\mathcal S}^0(N)={\mathcal S}(N)$ that contain $\gtq$ is a spear, as required.
\end{proof}

\begin{cor}\label{spcmaxh} 
Let $\Specmax^{0\diam} M$ and $\Specmax^{r\diam} M$ be the subspaces of $\Spec^{0\diam}(M)$ and $\Spec^{r\diam}(M)$ consisting of the maximal ideals of ${\mathcal S}^{0\diam}(M)$ and ${\mathcal S}^{r\diam}(M)$. Then the map
$$
\Specmax^{0\diam} M\to\Specmax^{r\diam} M,\,\gtm\mapsto\gtm\cap{\mathcal S}^{r\diam}(M)
$$
is a homeomorphism.
\end{cor}
\begin{proof} 
This is a straightforward consequence of Theorem \ref{main1} because $\Specmax^{0\diam} M$ and ${\Specmax^{r\diam} M}$ are respectively the subsets of closed points of $\Spec^{0\diam}(M)$ and $\Spec^{r\diam}(M)$.
\end{proof}

\begin{remark}\label{tilde}
It was proved in \cite[Cor.4.4]{fg4} that $\Specmax M$ is the semialgebraic Stone--\v{C}ech compactification of $M$, so Corollary \ref{spcmaxh} states that the same holds true for the space $\Specmax^r M$ of maximal ideals in ${\mathcal S}^r(M)$. This can be seen as the semialgebraic counterpart of the analogous result by Bkouche \cite{b} for the ring of differentiable functions on a differentiable manifold.
\end{remark}

\subsection{Consequences} 
If $M$ is locally compact, all properties and statements proved in \cite{fe1,fe2,fe3,fg1,fg3,fg4,fg5,fg6,fg7} concerning Zariski and maximal spectra of ${\mathcal S}^{0\diam}(M)$ that involve ${\mathcal S}^{r\diam}$-maps between semialgebraic sets (for instance, inclusions of semialgebraic sets) also hold for those of ${\mathcal S}^{r\diam}(M)$ for each $r\geq1$ as a straightforward application of Lemma \ref{reduc} below. In the general case, one has to adapt the corresponding results for rings of semialgebraic functions to the ring ${\mathcal S}^{0\diam}(M)$, which is the direct limit ${\mathcal S}^{0\diam}(M)=\displaystyle\lim_{\longrightarrow}{\mathcal S}^{\diam}(N)$. As an example on how to proceed, see for instance the proof of Corollary \ref{dim}.

\begin{lem}\label{reduc}
Let $f:=(f_1,\ldots,f_n):M\to N$ be a semialgebraic map between semialgebraic sets $M\subset\R^m$ and $N\subset\R^n$ such that each component $f_i:M\to\R$ is an ${\mathcal S}^{r\diam}$-function. Consider the homomorphisms of rings 
\begin{align*}
f^*&:{\mathcal S}^{r\diam}(N)\to{\mathcal S}^{r\diam}(M),\ g\mapsto g\circ f,\\
f^*&:{\mathcal S}^{0\diam}(N)\to{\mathcal S}^{0\diam}(M),\ h\mapsto h\circ f
\end{align*}
induced by $f$. Then the following diagram between spectral spaces is commutative.
{\small
$$
\xymatrix{
\Spec^{0\diam}(M)\ar[rr]^{\Spec^{0\diam}(f^*)}\ar[d]^{\varphi_M}&&\Spec^{0\diam}(N)\ar[d]^{\varphi_N}&\gtp\ar@{|->}[r]\ar@{|->}[d]&{f^*}^{-1}(\gtp)\ar@{|->}[d]\\
\Spec^{r\diam}(M)\ar[rr]^{\Spec^{r\diam}(f^*)}&&\Spec^{r\diam}(N)&\gtp\cap{\mathcal S}^{r\diam}(M)\ar@{|->}[r]&{f^*}^{-1}(\gtp\cap{\mathcal S}^{r\diam}(M))={f^*}^{-1}(\gtp)\cap{\mathcal S}^{r\diam}(N)
}
$$}
\end{lem}
\begin{proof}
Let $g\in f^{*-1}(\gtp\cap{\mathcal S}^{r\diam}(M))$. Then $g\in{\mathcal S}^{r\diam}(N)$ and $g\circ f\in\gtp\cap{\mathcal S}^{r\diam}(M)$, so $g\in f^{*-1}(\gtp)\cap{\mathcal S}^{r\diam}(N)$. Conversely, let $g\in f^{*-1}(\gtp)\cap{\mathcal S}^{r\diam}(N)$, then $g\in{\mathcal S}^{r\diam}(N)$ and $g\circ f\in\gtp$, so $g\circ f\in\gtp\cap{\mathcal S}^{r\diam}(M)$, that is, $g\in f^{*-1}(\gtp\cap{\mathcal S}^{r\diam}(M))$, as required.
\end{proof}

\section{Residue fields of rings of ${\mathcal S}^{r\diam}$-functions}\label{s4}

Our goal in this section is to prove Theorem \ref{main3}, which says that if $M$ is a semialgebraic set and $\gtp$ is a prime ideal of ${\mathcal S}^{r\diam}(M)$ the field of fractions $\kappa(\gtp)$ of ${\mathcal S}^{r\diam}(M)/\gtp$ is a real closed field. In Subsection \ref{main3:locallcompact} we prove this for the ring ${\mathcal S}^r(M)$ with $M$ locally compact and in Subsection \ref{main3:generalnonbd} we show this for the ring ${\mathcal S}^{r\diam}(M)$ for an arbitrary semialgebraic set $M$.

\subsection{The locally compact case}\label{main3:locallcompact}\setcounter{paragraph}{0} We will use the notion of semialgebraic depth of an ideal of a ring of semialgebraic functions firstly introduced in \cite{fg6}. Let $M\subset\R^m$ be a semialgebraic set. The \em semialgebraic depth \em of $\gtq\in\Spec^r(M)$ is the nonnegative integer
$$
\depth(\gtq):=\min\{\dim(Z(f)):f\in\gtq\}.
$$

\begin{lem}\label{ht00}
Let $M\subset\R^m$ be a semialgebraic set and let $\gtq$ be a prime ideal in ${\mathcal S}^r(M)$. 
\begin{itemize}
\item[(i)] $\depth(\gtq)=\depth(\sqrt{\gtq{\mathcal S}^0(M)})$.
\item[(ii)] If $\gtq_1\in\Spec^r(M)$ is a $z$-ideal such that $\gtq_1\subsetneq\gtq$ then $\depth(\gtq)<\depth(\gtq_1)$.
\item[(iii)] If $M$ is locally compact then $\depth(\gtq)+\hgt(\gtq)\leq\dim(M)$.
\end{itemize}
\end{lem}
\begin{proof}
(i) It is clear that $\depth(\gtq)\geq\depth(\sqrt{\gtq{\mathcal S}^0(M)})$. Let $f\in\sqrt{\gtq{\mathcal S}^0(M)}$ be such that $\dim(Z(f))=\depth(\sqrt{\gtq{\mathcal S}^0(M)})$. By Remark \ref{radONE} there exists $g\in\gtq$ such that $Z(g)\subset Z(f)$, so $\depth(\gtq)\leq\dim(Z(g))\leq\dim(Z(f))=\depth(\sqrt{\gtq{\mathcal S}^0(M)})$, as required.

Once we have proved (i), by Theorem \ref{main1} and Remark \ref{radONE} it is enough to prove (ii) and (iii) for $r=0$. This is the content of \cite[Lem. 2.1 \& Thm. 2.2]{fg6}.
\end{proof}

We point out that Lemma \ref{ht00}(ii) is false if $\gtq_1$ is not a $z$-ideal \cite[Rmk.2.3]{fg6}. We will reduce Theorem \ref{main3} for the general locally compact case to the pure dimensional locally compact case via the following technical result:

\begin{lem}\label{ExTDiFF} 
Let $M\subset\R^m$ be a pure dimensional semialgebraic set and $f\in{\mathcal S}(M)$. Then there exists $g\in{\mathcal S}^{r*}(M)$ such that $\dim(Z(g))<\dim(M)$ and $fg$ extends by zero to an ${\mathcal S}^r$-function on $\R^m$. 
\end{lem}
\begin{proof}
By \cite[(2.4.2)]{fgr} $M$ is the disjoint union of finitely many (affine) Nash manifolds $N_1,\dots, N_s$ such that each restriction $f|_{N_i}$ is a Nash function and each $N_i$ with $\dim(N_i)=\dim(M)$ is open in $M$. We may assume that $\dim(N_i)=\dim(M)$ if $1\leq i\leq k$ and $\dim(N_i)<\dim(M)$ if $k+1\leq i\leq s$. As each $N_i$ for $1\leq i\leq k$ is open in the pure dimensional semialgebraic set $M$, the union $N:=\bigcup_{i=1}^kN_i$ is a Nash manifold and a dense open semialgebraic subset of $M$. In addition, the restriction $f|_N$ is a Nash function.

By \cite[Cor.8.9.5]{bcr} there exists a Nash tubular neighborhood $(V,\rho)$ of $N$ in $\R^m$. Let us define $F:=f\circ\rho\in{\mathcal S}(V)$. As $V$ is an open semialgebraic subset of $\R^m$, there exists $G\in{\mathcal S}^{r*}(\R^m)$ with $Z(G)=\R^m\setminus V$. By Proposition \ref{suboclase} there exists $\ell\geq 1$ such that $G^\ell F$ extends by zero to an ${\mathcal S}^r$-function on $\R^m$. As $Z(G^\ell)\cap M=M\setminus V$ has dimension $\leq\dim(M)-1$ and $F|_{M\setminus V}=f|_{M\setminus V}$, the function $g:=G^\ell|_M\in{\mathcal S}^{r*}(M)$ does the job.
\end{proof}

\begin{proof}[Proof of Theorem \em\ref{main3} \em for ${\mathcal S}^r(M)$ with $M$ locally compact] 
Let us prove Theorem \ref{main3} for a prime ideal $\gtp\in\Spec^0(M)$ where $M$ is locally compact. Denote $\gtq:=\gtp\cap{\mathcal S}^r(M)$. Let $\kappa(\gtp)$ and $\kappa(\gtq)$ be the fields of fractions of ${\mathcal S}^0(M)/\gtp$ and ${\mathcal S}^r(M)/\gtq$. We have to show that the inclusion 
$$
{\tt j}:{\mathcal S}^r(M)/\gtq\to{\mathcal S}^0(M)/\gtp.
$$
induces an isomorphism between $\kappa(\gtq)$ and $\kappa(\gtp)$. We will make several reductions:	

\noindent\em First reduction\em. We may assume that $\gtp$ is a minimal prime ideal of ${\mathcal S}^0(M)$. 

Pick $f\in\gtp$ with $\depth(\gtp)=\dim(N)$, where $N:=Z(f)$ is locally compact because it is closed in the locally compact semialgebraic set $M$. By Fact \ref{thfact} the epimorphism $\varphi_0:{\mathcal S}^0(M)\to{\mathcal S}^0(N),\,g\mapsto g|_N$ induces an isomorphism $\varphi_0':{\mathcal S}^0(M)/\ker(\varphi_0)\to{\mathcal S}^0(N)$. As $\gtp$ is a $z$-ideal, $\ker(\varphi_0)\subset\gtp$, so the image $\gtp_1$ of $\gtp/\ker(\varphi_0)$ under $\varphi_0'$ is a prime ideal of ${\mathcal S}^0(N)$. We have an isomorphism
$$
\ol{\varphi}_0:{\mathcal S}^0(M)/\gtp\to{\mathcal S}^0(N)/\gtp_1.
$$

The restriction map $\varphi_r:{\mathcal S}^r(M)\to{\mathcal S}^r(N)$ induces an isomorphism $\varphi_r':{\mathcal S}^r(M)/\ker(\varphi_r)\to{\mathcal S}^r(N)$. As $\gtq$ is a $z$-ideal, we can consider the image $\gtq_1\in\Spec^r(N)$ of $\gtq/\ker(\varphi_r)$ under $\varphi_r'$. Then we have an isomorphism
$$
\ol{\varphi}_r:{\mathcal S}^r(M)/\gtq\to{\mathcal S}^r(N)/\gtq_1.
$$
As the following diagram
$$
\xymatrix{
{\mathcal S}^r(M)/\gtq\hspace{3mm}\ar@{^{}->}[r]^{{\tt j}}\ar@{^{}->}^{\ol{\varphi}_r}[d]&{\mathcal S}^0(M)/\gtp\ar@{^{}->}^{\ol{\varphi}_0}[d]\\
{\mathcal S}^r(N)/\gtq_1\ar@{^{}->}[r]^{{\tt j}_N}&{\mathcal S}^0(N)/\gtp_1
}
$$
is commutative (where ${\tt j}_N(f+\gtq_1)=f+\gtp_1$), to finish this reduction it is enough to check that $\gtp_1$ is a minimal prime ideal of ${\mathcal S}^0(N)$. Otherwise, there exists a prime ideal $\gtp'_1\subsetneq\gtp_1$ in ${\mathcal S}^0(N)$, so $\gtp':=\varphi_0^{-1}(\gtp'_1)$ is a prime ideal of ${\mathcal S}^0(M)$ such that $\ker(\varphi_0)\subset\gtp'\subsetneq\gtp$. As $M$ is locally compact, $\gtp'$ is a $z$-ideal, so $\depth(\gtp)<\depth(\gtp')$ by Lemma \ref{ht00}. However, as $f\in\ker(\varphi_0)\subset\gtp'$, we have $\depth(\gtp')\leq\dim(Z(f))=\depth(\gtp)$, which is a contradiction.

\noindent\em Second reduction\em. We may assume $M$ is pure dimensional.

Let $M\subset\R^m$ be a semialgebraic case and let $\{M_1,\dots,M_s\}$ be the bricks of $M$. By Lemma \ref{zero} there exist $f_1,\dots,f_s\in{\mathcal S}^{r*}(M)$ such that $M_i=Z(f_i)$ for $1\leq i\leq s$. As $f:=\prod_{i=1}^sf_i$ is the zero function on $M$, there exists an index $i\in\{1,\dots,s\}$ such that $f_i\in\gtp$. The epimorphism $\varphi_0:{\mathcal S}^0(M)\to{\mathcal S}^0(M_i),\,g\mapsto g|_{M_i}$ induces an isomorphism $\varphi_{0i}':{\mathcal S}^0(M)/\ker(\varphi_0)\to{\mathcal S}^0(M_i)$. As $\gtp$ is a $z$-ideal, $\ker(\varphi_0)\subset\gtp$. Denote by $\gtp_i\in\Spec^0(M_i)$ the image of $\gtp/\ker(\varphi_0)$ under $\varphi_{0i}'$. Arguing as in the first reduction, it is enough to prove the statement for the prime $\gtp_i$ of the ring ${\mathcal S}^0(M_i)$, where $M_i$ is pure dimensional.

\noindent\em Final step\em. We assume from the beginning that $M$ is pure dimensional and $\gtp$ is a minimal prime ideal of ${\mathcal S}^0(M)$. 

The homomorphism ${\tt j}:{\mathcal S}^r(M)/\gtq\to{\mathcal S}^0(M)/\gtp
$ induces an injective homomorphism ${\tt j}:\kappa(\gtq)\to\kappa(\gtp)$ (we keep the same notation ${\tt j}$ for the sake of simplicity) and we claim that it is in fact an isomorphism. We have to show: \em ${\tt j}$ is surjective\em.

Pick $f\in{\mathcal S}^0(M)$. By Lemma \ref{ExTDiFF} there exist $g\in{\mathcal S}^{r*}(M)$ and $H\in{\mathcal S}^r(\R^m)$ such that $\dim(Z(g))<\dim(M)$ and $h:=H|_M=fg$. Let us check: $g\not\in\gtp$. 

Otherwise, as $\gtp$ is a minimal prime ideal, there exists by \cite[Lem.1.1]{hj} $v\in{\mathcal S}^r(M)\setminus\gtp$ such that $gv=0$. Thus, $D(v)\subset Z(g)$, so
$$
\dim(D(v))\leq\dim(Z(g))<\dim(M).
$$
As $M$ is pure dimensional, $v=0$, which is a contradiction because $v\notin\gtp$. Thus, $g\not\in\gtp$.

Hence, $g+\gtp\neq0$, so $(h+\gtq)/(g+\gtq)\in\kappa(\gtq)$ and $(h+\gtp)/(g+\gtp)\in\kappa(\gtp)$. Consequently, ${\tt j}((h+\gtq)/(g+\gtq))=(h+\gtp)/(g+\gtp)={f+\gtp}$ and ${\tt j}$ is surjective.

It was proved in \cite[\S1.Cor.3.26]{s4} (see also Proposition \ref{f:rcl1} below) that $\kappa(\gtp)$ is a real closed field, so $\kappa(\gtq)$ is as well a real closed field, as required.
\end{proof}

\begin{cor}\label{trdeg} 
Let $M\subset\R^m$ be a locally compact semialgebraic set and let $\gtq\in\Spec^r(M)$. Then:
\begin{itemize}
\item[(i)] $\tr\deg_{\R}(\kappa(\gtq))=\depth(\gtq)$.
\item[(ii)] Let $\rho:M\hookrightarrow M'$ be a ${\mathcal S}^r$-embedding such that $M'\subset\R^n$ is a locally compact semialgebraic set and $\rho(M)$ is dense in $M'$. Let $\gtq':=(\rho^*)^{-1}(\gtq)$ where $\rho^*:{\mathcal S}^r(M')\to{\mathcal S}^r(M)$ is the injective homomorphism induced by $\rho$. The inclusion ${\mathcal S}^r(M')/\gtq'\hookrightarrow{\mathcal S}^r(M)/\gtq$ induces an isomorphism between the corresponding fields of fractions $\kappa(\gtq')$ and $\kappa(\gtq)$.
\end{itemize}
\end{cor}
\begin{proof} 
(i) By Theorem \ref{main1} $\gtp:=\sqrt{\gtq{\mathcal S}^0(M)}\in\Spec^0(M)$ and $\gtq=\gtp\cap{\mathcal S}^r(M)$. Thus, by Theorem \ref{main3} the fields $\kappa(\gtq)$ and $\kappa(\gtp)$ are $\R$-isomorphic. As $\depth(\gtp)=\depth(\gtq)$ by Lemma \ref{ht00}(i), it only remains to prove the statement for $\gtp$, which follows from \cite[Thm. 1.1]{fg6}.
 
(ii) It is enough to show that the inclusion $\kappa(\gtq')\to\kappa(\gtq)$ is surjective. As both $\kappa(\gtq')$ and $\kappa(\gtq)$ are real closed fields, it is enough if we prove that $\tr\deg_{\R}(\kappa(\gtq'))\geq\tr\deg_{\R}(\kappa(\gtq))$, or equivalently, that $\depth(\gtq')\geq\depth(\gtq)$. As $\rho$ is injective, for each $f\in\gtq'$ it holds $\dim(Z(f))\geq\dim(Z(f\circ\rho))\geq\depth(\gtq)$, so $\depth(\gtq')\geq\depth(\gtq)$, as required.
\end{proof}

\subsection{The general case}\label{main3:generalnonbd}
Fix a semialgebraic set $M\subset\R^m$ and an integer $r\geq0$. Denote ${\mathfrak C}^{r\diam}$ as in \ref{LK} the collection of pairs $(E,{\tt j})$ where ${\tt j}:M\to\R^n$ is an ${\mathcal S}^{r\diam}$ embedding such that $E= \cl({\tt j}(M))$ is a closed semialgebraic set (even compact in the bounded case). By Theorem \ref{comp} we have a homomorphism ${\tt j}^*:{\mathcal S}^{r\diam}(E)\to{\mathcal S}^{r\diam}(M),\,f\mapsto f\circ{\tt j}$ and for each ideal $\gta$ in ${\mathcal S}^{r\diam}(M)$ we denote $\gta\cap{\mathcal S}^{r\diam}(E):=({\tt j}^*)^{-1}(\gta)$ and say \emph{$\gta$ lies over $\gta\cap{\mathcal S}^{r\diam}(E)$}.

\begin{prop}\label{citalc1} 
Let $M\subset\R^m$ be a semialgebraic set.
\begin{itemize}
\item[(i)] Given a subset $\Ff:=\{f_1,\ldots,f_k\}\subset{\mathcal S}^{r\diam}(M)$ there exist $(E_{\Ff},{\tt j}_{\Ff})\in{\mathfrak C}^{r\diam}$ and $F_1,\ldots,F_k\in{\mathcal S}^{r\diam}(E_{\Ff})$ such that $F_i\circ{\tt j}_{\Ff}=f_i$. 
\item[(ii)] Given a chain of prime ideals $\gtq_0\subsetneq\cdots\subsetneq\gtq_k$ in ${\mathcal S}^{r\diam}(M)$ there exists $(E,{\tt j})\in{\mathfrak C}^{r\diam}$ such that the prime ideals $\gtQ_i:=\gtq_i\cap{\mathcal S}^{r\diam}(E)$ constitute a chain $\gtQ_0\subsetneq\cdots\subsetneq\gtQ_k$ in ${\mathcal S}^{r\diam}(E)$. 
\item[(iii)] There exists $(E,{\tt j})\in{\mathfrak C}^{r\diam}$ such that $\dim({\mathcal S}^{r\diam}(M))\leq\dim({\mathcal S}^{r\diam}(E))=\dim(M)$.
\end{itemize}
\end{prop}
\begin{proof} 
(i) Consider the ${\mathcal S}^{r\diam}$ embedding ${\tt j}_{\Ff}:M\to\R^{m+k},\ x\mapsto(x,f_1(x),\ldots,f_k(x))$ and define $E_{\Ff}:=\cl({\tt j}_{\Ff})$. In case we are dealing with bounded functions, assume $M$ is bounded. Define also $F_i:=\pi_{m+i}|_{E_{\Ff}}$ where 
$$
\pi_{m+i}:\R^{m+k}\to\R,\ x:=(x_1,\ldots,x_{m+k})\mapsto x_{m+i}
$$ 
for $i=1,\ldots,k$. Each $F_i\in{\mathcal S}^{r\diam}(E_{\Ff})$ and $F_i\circ{\tt j}_{\Ff}=f_i$.

(ii) For $1\leq i\leq k$ pick a function $f_i\in\gtq_i\setminus\gtq_{i-1}$ and consider the pair $(E_{\Ff},{\tt j}_{\Ff})\in{\mathfrak C}^{r\diam}$ provided by part (i) for the family $\Ff:=\{f_1,\ldots,f_k\}$. Clearly, $\gtQ_{i-1}\subset\gtQ_i$ and each $F_i\in\gtQ_i\setminus\gtQ_{i-1}$.

(iii) By part (ii) there exists $(E,{\tt j})\in{\mathfrak C}^{r\diam}$ such that $\dim({\mathcal S}^{r\diam}(M))\leq\dim({\mathcal S}^{r\diam}(E))$. Finally, the equalities $\dim({\mathcal S}^{r\diam}(E))=\dim(E)=\dim(M)$ 
follow from Corollary \ref{dim} and \cite[Prop.2.8.2]{bcr}. 
\end{proof}

\begin{cor}\label{RCf} 
Let $\gtq\in\Spec^{r\diam}(M)$. Then $\kappa(\gtq):=\qf({\mathcal S}^{r\diam}(M)/\gtq)$ is a real closed field.
\end{cor}
\begin{proof} 
As ${\tt j}(M)$ is dense in $E$, the homomorphism ${\tt j}^*:{\mathcal S}^{r\diam}(E)\to{\mathcal S}^{r\diam}(M)$ is injective for each $(E,{\tt j})\in {\mathfrak C}^{r\diam}$. Define $\gtQ_{\tt j}:=({\tt j}^*)^{-1}(\gtq)\in\Spec^{r\diam}(E)$ and note that ${\ol{{\tt j}^*}:{\mathcal S}^{r\diam}(E)/\gtQ_{\tt j}\to{\mathcal S}^{r\diam}(M)/\gtq}$ is also an injective homomorphism. Thus, by Lemma \ref{LD} it follows that $\kappa(\gtq)=\displaystyle\lim_{\longrightarrow}\kappa(\gtQ_{\tt j})$. As each field $\kappa(\gtQ_{\tt j}):=\qf({\mathcal S}^{r\diam}(E)/\gtQ_{\tt j})$ is a real closed ring, the field $\kappa(\gtq)$ is by \cite[\S1 Thm. 4.8]{s4} also a real closed ring, so it is a real closed field, as required.
\end{proof}

\begin{remark}
If $\gtn$ is a maximal ideal of ${\mathcal S}^{r*}(M)$, then the homomorphism of fields ${\tt j}:\R\hookrightarrow\kappa(\gtn):={\mathcal S}^{r*}(M)/\gtn,\,r\mapsto r+\gtn$ is an isomorphism. The extension of real closed fields $\kappa(\gtn)|\R$ is archimedean, because each element of ${\mathcal S}^{r*}(M)$ is bounded by a real number. This proves the surjectivity of the embedding ${\tt j}$ because $\R$ does not admit proper archimedean extensions.	 
\end{remark}

\begin{lem}\label{brimming1} 
Let $\gtq\in\Spec^{r\diam}(M)$. Then there exist $(E,{\tt j})\in{\mathfrak C}^{r\diam}$ and $\gtq_E\in\Spec^{r\diam}(E)$ such that $\kappa(\gtq)=\kappa(\gtq_E)$ and $\gtq$ lies over $\gtq_E$. We refer to $(E,{\tt j})$ as a \em brimming ${\mathcal S}^{r\diam}$-completion \em of $M$ for $\gtq$. 
\end{lem}
\begin{proof}
Consider the chain of homomorphisms $\varphi:{\mathcal S}^{r\diam}(M)\to{\mathcal S}^{r\diam}(M)/\gtq\hookrightarrow\kappa(\gtq)$. For each finite subset $\Ff$ of ${\mathcal S}^{r\diam}(M)$ let $(E_{\Ff},{\tt j}_{\Ff})\in{\mathfrak C}^{r\diam}$ be as in Proposition \ref{citalc1}. Denote 
$$
{\tt j}_{\Ff}^*:{\mathcal S}^{r\diam}(E_{\Ff})\to{\mathcal S}^{r\diam}(M),\,F\mapsto F\circ{\tt j}_{\Ff}\quad\text{and}\quad\varphi_{\Ff}:=\varphi\circ{\tt j}_{\Ff}^*:{\mathcal S}^{r\diam}(E_{\Ff})\to\kappa(\gtq).
$$
Define $\gtq_{\Ff}:=\ker(\varphi_{\Ff})=\gtq\cap{\mathcal S}^{r\diam}(E_{\Ff})$ and ${\tt d}:=\max_{\Ff}\{{\tt d}(\gtq_{\Ff})\}$ where ${\Ff}$ runs over all finite subsets of ${\mathcal S}^{r\diam}(M)$. Fix a finite subset ${\Ff}_0$ of ${\mathcal S}^{r\diam}(M)$ such that ${\tt d}(\gtq_{\Ff_0})={\tt d}$. Denote 
$$
\gtq_{{\Ff}_0}:=\ker(\varphi_{\Ff_0})={\mathcal S}^{r\diam}(E_{\Ff_0})\cap\gtq\quad\text{and}\quad\kappa_0:=\qf({\mathcal S}^{r\diam}(E_{\Ff_0})/\gtq_{{\Ff}_0})\subset\kappa(\gtq).
$$ 
Let us prove: $\kappa(\gtq)=\kappa_0$, so $(E_{\Ff_0},{\tt j}_{{\Ff}_0})$ satisfies the required conditions. 

As both are real closed fields, it is enough to show: \em $\kappa(\gtq)$ is an algebraic extension of $\kappa_0$\em. To that end, it is enough to see: \em $f+\gtq$ is algebraic over $\kappa_0$ for each $f\in{\mathcal S}^{r\diam}(M)\setminus\gtq$\em.

Let $\Ff_1:=\Ff_0\cup\{f\}$ and $(E_{\Ff_1},{\tt j}_{\Ff_1})\in{\mathfrak C}^{r\diam}$. The projection onto all the coordinates except for the last one induces an ${\mathcal S}^{r\diam}$-map $\rho:E_{\Ff_1}\to E_{\Ff_0}$ such that $\rho \circ {\tt j}_{\Ff_1}={\tt j}_{\Ff_0}$. Consider the $\R$-homomorphism ${\mathcal S}^{r\diam}(E_{\Ff_0})\to{\mathcal S}^{r\diam}(E_{\Ff_1}),\,h\mapsto h\circ\rho$. Denote $\gtq_{{\Ff}_1}:=\ker(\varphi_{{\Ff}_1})={\mathcal S}^{r\diam}(E_{\Ff_1})\cap\gtq$ and the real closed field $\kappa_1:=\qf({\mathcal S}^{r\diam}(E_{\Ff_1})/\gtq_{{\Ff}_1})$. We have the following commutative diagrams of homomorphisms
$$
\xymatrix{
{\mathcal S}^{r\diam}(E_{\Ff_0}) \hspace{4mm} \ar@{^{(}->}[r]\ar@{^{(}->}[rd]& {\mathcal S}^{r\diam}(E_{\Ff_1})\ar@{^{(}->}[d]\\
&{\mathcal S}^{r\diam}(M)}
\qquad
\xymatrix{
{\mathcal S}^{r\diam}(E_{\Ff_0})/\gtq_{{\Ff}_0} \ar@{^{(}->}[r]\ar@{^{(}->}[rd]^{\vspace{1cm}}&{\mathcal S}^{r\diam}(E_{\Ff_1})/\gtq_{{\Ff}_1}\ar@{^{(}->}[d]\\
&{\mathcal S}^{r\diam}(M)/\gtq}
$$
so $\kappa_0\subset\kappa_1\subset\kappa(\gtq)$. As $f\in{\Ff_1}$, there exists $F\in{\mathcal S}^{r\diam}(E_{\Ff_1})$ such that $F\circ{\tt j}_{\Ff_1}=f$. To see that $f+\gtq$ is algebraic over $\kappa_0$ it is enough to prove: \em $F+\gtq_{{\Ff}_1}$ is algebraic over $\kappa_0$\em. For that, it is enough to check: \em the transcendence degrees over $\R$ of $\kappa_0$ and $\kappa_1$ are finite and coincide\em. 

This follows from Corollary \ref{trdeg} because
$$
\tr\deg_{\R}(\kappa_0)\leq\tr\deg_{\R}(\kappa_1)={\tt d}(\gtq_{{\Ff}_1})\leq{\tt d}={\tt d}(\gtq_{{\Ff}_0})=\tr\deg_\R(\kappa_0).
$$
Hence, $\tr\deg_{\R}(\kappa_0)=\tr\deg_{\R}(\kappa_1)$, as required.
\end{proof}
\begin{example}
Not all pairs $(E,{\tt j})\in{\mathfrak C}^{r\diam}$ are brimming ${\mathcal S}^{r\diam}$-completions of $M$ for a given $\gtq\in\Spec^r(M)$, see \cite[Rmk.3.1]{fg6}.
\end{example}

\begin{cor}\label{cor:brimming} 
Let $\gtp\in\Spec^{0\diam}(M)$ and $\gtq:=\gtp\cap{\mathcal S}^{r\diam}(M)\in\Spec^{r\diam}(M)$. Then there exist brimming completions $(E_1,{\tt j_1})\in{\mathfrak C}^{0\diam}$ and $(E_2,{\tt j_2})\in{\mathfrak C}^{r\diam}$ of $\gtp$ and $\gtq$ respectively and a ${\mathcal S}^{r\diam}$-map $\rho:E_1\to E_2$ such that ${\tt j_2}=\rho\circ {\tt j_1}$.
\end{cor}
\begin{proof}
We have shown in the proof of Lemma \ref{brimming1} that there exists a finite subset $\Ff_1$ of ${\mathcal S}^{0\diam}(M)$ such that $(E_{\Ff_1},{\tt j}_{\Ff_1})$ is a brimming ${\mathcal S}^{0\diam}$-completion of $\gtp$. Analogously, there exists a finite subset $\Ff_2$ of ${\mathcal S}^{r\diam}(M)$ such that $(E_{\Ff_2},{\tt j}_{\Ff_2})$ is a brimming ${\mathcal S}^{r\diam}$-completion of $\gtq$. Define $\Ff'_1:=\Ff_1\cup\Ff_2$ and note that as we have seen in the proof of Lemma \ref{brimming1} the pair $(E_{\Ff'_1},{\tt j}_{\Ff'_1})$ is also a brimming ${\mathcal S}^{0\diam}$-completion of $\gtp$. In addition, the projection onto all the coordinates except the last $|\Ff_2|$ coordinates induces a ${\mathcal S}^{r\diam}$-map $\rho:E_{\Ff'_1}\to E_{\Ff_2}$ such that ${\tt j}_{\Ff_2}=\rho\circ {\tt j}_{\Ff'_1}$, as required.
\end{proof}

We have all the ingredients to prove Theorem \ref{main3} in the general case. 

\begin{proof}[Proof of Theorem \em \ref{main3} \em in the general case] 
Let $M\subset\R^m$ be an arbitrary semialgebraic set. Pick $\gtp\in\Spec^{0\diam}(M)$ and denote $\gtq:=\gtp\cap{\mathcal S}^{r\diam}(M)\in\Spec^{r\diam}(M)$. We claim: \em the homomorphism ${\mathcal S}^{r\diam}(M)/\gtq\to{\mathcal S}^{0\diam}(M)/\gtp$ induces an isomorphism between the fields of fractions $\kappa(\gtq)$ and $\kappa(\gtp)$\em.

Pick brimming ${\mathcal S}^{r\diam}$-completions $(E_1,{\tt j_1})\in {\mathfrak C}^0$ and $(E_2,{\tt j_2})\in {\mathfrak C}^r$ of $M$ for $\gtp$ and $\gtq$ as in Corollary \ref{cor:brimming} and let $\rho:E_1\to E_2$ be an ${\mathcal S}^{r\diam}$-map such that ${\tt j_2}=\rho\circ {\tt j_1}$. Denote $\gtp_{E_1}:=\gtp\cap{\mathcal S}^{0\diam}(E_1)$ and $\gtq_{E_2}:=\gtq\cap{\mathcal S}^{r\diam}(E_2)$. We point out that if ${\mathcal S}^{r\diam}={\mathcal S}^{r*}$, then $E_2$ is a compact semialgebraic set and it holds ${\mathcal S}^{r*}(E_2)={\mathcal S}^r(E_2)$ and ${\mathcal S}^{0*}(E_2)={\mathcal S}^0(E_2)$. 

The homomorphisms ${\mathcal S}^0(E_1)/\gtp_{E_1}\to{\mathcal S}^{0\diam}(M)/\gtp$ and ${\mathcal S}^r(E_2)/\gtq_{E_2}\to{\mathcal S}^{r\diam}(M)/\gtq$ induce isomorphisms between the corresponding fields of fractions. Define $\gtp_{E_2}:=\gtp\cap{\mathcal S}^0(E_2)$ and note that $\gtq_{E_2}=\gtp_{E_2}\cap{\mathcal S}^r(E_2)$. By the locally compact version of Theorem \ref{main3} proved in Subsection \ref{main3:locallcompact} the homomorphism ${\mathcal S}^r(E_2)/\gtq_{E_2}\to{\mathcal S}^0(E_2)/\gtp_{E_2}$ induces an isomorphism between the corresponding fields of fractions. 

It holds $\gtp_{E_2}=(\rho^*)^{-1}(\gtp_{E_1})$ where $\rho^*:{\mathcal S}^0(E_2)={\mathcal S}^{0\diam}(E_2)\to{\mathcal S}^{0\diam}(E_1),\ f\mapsto f\circ\rho$ is the homomorphism induced by $\rho$. By Corollary \ref{trdeg}(ii) the homomorphism ${\mathcal S}^0(E_2)/\gtp_{E_2}\to{\mathcal S}^{0\diam}(E_1)/\gtp_{E_1}$ induces an isomorphism between the corresponding fields of fractions. Consequently, we have the following commutative diagrams
$$
\xymatrix{
&{\mathcal S}^r(E_2)/\gtq_{E_2}\ar[r]^{\doteq}\ar[dl]_{\doteq}\ar[d]&{\mathcal S}^{r\diam}(M)/\gtq\ar[d]\\
{\mathcal S}^0(E_2)/\gtp_{E_2}\ar[r]^{\doteq}&{\mathcal S}^{0\diam}(E_1)/\gtp_{E_1}\ar[r]^{\doteq}&{\mathcal S}^{0\diam}(M)/\gtp}\quad
\raisebox{-0.75cm}{$\leadsto$}
\xymatrix{
&\kappa(\gtq_{E_2})\ar[r]^{\cong}\ar[dl]_{\cong}\ar[d]&\kappa(\gtq)\ar[d]\\
\kappa(\gtp_{E_2})\ar[r]^{\cong}&\kappa(\gtp_{E_1})\ar[r]^{\cong}&\kappa(\gtp)}
$$
where the symbol $\doteq$ means that the corresponding homomorphism induces an isomorphism between the corresponding fields of fractions. We conclude that the homomorphism ${\mathcal S}^{r\diam}(M)/\gtq\to{\mathcal S}^{0\diam}(M)/\gtp$ induces an isomorphism between $\kappa(\gtq)$ and $\kappa(\gtp)$, as required.
\end{proof}

There exist examples of non-locally compact semialgebraic sets $M$ for which there exists $\gtq\in\Spec^r(M)$ with ${\tt d}(\gtq)<\tr\deg_{\R}(\kappa(\gtq))$, see a counterexample in \cite[(3.4.1)]{fg5} for the case $r=0$. In general, we have the following result:

\begin{cor}\label{trdgz} 
Let $\gtq$ be a prime ideal of ${\mathcal S}^r(M)$. Then ${\tt d}(\gtq)\leq\tr\deg_{\R}(\kappa(\gtq))$. If in addition $\gtq$ is a $z$-ideal, then the equality holds.
\end{cor}
\begin{proof} 
Pick a brimming ${\mathcal S}^r$-completion $(E,{\tt j})$ of $M$ for $\gtq$ and denote $\gtq_E:=\gtq\cap{\mathcal S}^r(E)$. By Corollary \ref{trdeg} and Lemma \ref{brimming1} ${\tt d}(\gtq_E)=\tr\deg_{\R}(\kappa(\gtq_E))=\tr\deg_{\R}(\kappa(\gtq))$, so for the first part of the statement it is enough to check: ${\tt d}(\gtq)\leq{\tt d}(\gtq_E)$. 

Let $F\in\gtq_E$ with ${\tt d}(\gtq_E)=\dim(Z(F))$. Then $f:=F\circ{\tt j}\in\gtq$ and 
$$
{\tt d}(\gtq)\leq\dim(Z(f))\leq\dim(Z(F))={\tt d}(\gtq_E).
$$
Assume next, $\gtq$ is a $z$-ideal. For the converse inequality ${\tt d}(\gtq_E)\leq{\tt d}(\gtq)$, let $h\in\gtq$ be such that ${\tt d}(\gtq)=\dim(Z(h))$. By Lemma \ref{zero} there exists $g\in{\mathcal S}^r(E)$ such that $\cl_E({\tt j}(Z(h)))=Z(g)$, so $Z(h)\subset Z(g\circ{\tt j})$. As $\gtq$ is a $z$-ideal, we get $g\in\gtq_E$. Consequently,
$$
{\tt d}(\gtq_E)\leq\dim(Z(g))=\dim(\cl_E({\tt j}(Z(h))))=\dim(Z(h))={\tt d}(\gtq), 
$$
as required.
\end{proof}

We finish this section providing two important classes of ideals that are $z$-ideals.

\begin{prop}\label{zdl} 
Let $M\subset\R^m$ be a semialgebraic set.
\begin{itemize}
\item[(i)] Let $\gta$ be a $z$-ideal of ${\mathcal S}^r(M)$. Every prime ideal of ${\mathcal S}^r(M)$ that is minimal among the prime ideals containing $\gta$ is a $z$-ideal. In particular, minimal prime ideals of ${\mathcal S}^r(M)$ are $z$-ideals.
\item[(ii)] Let $\gtb$ be a proper ideal in ${\mathcal S}^r(M)$. Then there exists a proper $z$-ideal $\gtb^z$ in ${\mathcal S}^r(M)$ that contains $\gtb$. Every proper prime ideal which is maximal among the proper ideals containing $\gtb$ is a $z$-ideal. In particular, maximal ideals of ${\mathcal S}^r(M)$ are $z$-ideals.
\end{itemize}
\end{prop}
\begin{proof} 
(i) Let $\gtq\in\Spec^r(M)$ be minimal among the prime ideals containing $\gta$. Suppose there exist $f,g\in{\mathcal S}^r(M)$ with $Z(f)\subset Z(g)$ and $f\in\gtq$ but $g\notin\gtq$. Let ${\mathcal T}_0:={\mathcal S}^r(M)\setminus\gtq$ and consider the multiplicatively closed subset ${\mathcal T}:=\{hf^\ell:\ h\in {\mathcal T}_0, \,\ell\in\Z,\ell\geq 0\}$ of ${\mathcal S}^r(M)$. We claim: ${\mathcal T}\cap\gta=\varnothing$. 

Otherwise, there exist $h\in{\mathcal T}_0$ and $\ell\geq 0$ such that $hf^\ell\in\gta$. But $\gta$ is a $z$-ideal and 
$$
Z(hf^\ell)=Z(h)\cup Z(f)\subset Z(h)\cup Z(g)=Z(hg).
$$
Thus, $hg\in\gta\subset\gtq$, which is a contradiction. 

Therefore, ${\mathcal T}^{-1}\gta$ is a proper ideal of the localization ${\mathcal T}^{-1}{\mathcal S}^r(M)$, so there exists an ideal $\gtq_0$ of ${\mathcal S}^r(M)$ such that ${\mathcal T}^{-1}\gtq_0$ is maximal and contains ${\mathcal T}^{-1}\gta$. Consequently, $\gtq_0$ is a prime ideal such that ${\mathcal T}\cap\gtq_0=\varnothing$ and $\gta\subset\gtq_0$. As ${\mathcal T}\cap\gtq_0=\varnothing$ and $f\in{\mathcal T}$, we have $\gtq_0\subsetneq\gtq$, which contradicts the minimality of $\gtq$.

(ii) Let us check: 
$$
\gtb^z:=\{f\in{\mathcal S}^r(M):\,\exists\,g\in\gtb\ \text{such that}\ Z(g)\subset Z(f)\}
$$
\em is a $z$-ideal of ${\mathcal S}^r(M)$ containing $\gtb$\em. 

Given $f_1,f_2\in\gtb^z$ and $h\in{\mathcal S}^r(M)$, there exist $g_1,g_2\in\gtb$ such that $Z(g_i)\subset Z(f_i)$ for $i=1,2$. As $g_1^2+g_2^2\in\gtb$ and $hg_1\in\gtb$, we have 
$$
Z(g_1^2+g_2^2)\subset Z(f_1+f_2)\quad\text{and}\quad Z(hg_1)\subset Z(hf_1).
$$
This proves that $\gtb^z$ is an ideal and it is a $z$-ideal containing $\gtb$. 

If $\gtb^z={\mathcal S}^r(M)$ there exists $g\in\gtb$ such that $Z(g)\subset Z(1)=\varnothing$. Thus $g$ is a unit in ${\mathcal S}^r(M)$, which is a contradiction because $\gtb$ is proper.

Finally, if $\gtm$ is a proper prime ideal which is maximal among the proper ideals containing $\gtb$, then $\gtm^z=\gtm$, as required.
\end{proof}

\section{The ring of ${\mathcal S}^{\infty}$ functions}\label{s5}

The main purpose of this section is to study the ring ${\mathcal S}^{\infty}(M)$ where $M\subset\R^m$ is a semialgebraic set. To that aim, we need to recall the machinery of real closed rings.

\subsection{Real closed rings and real closure} 
Let $A$ be a commutative ring with unity. The set $\Sper(A)$ is the collection of \emph{prime cones} of $A$, that is, subsets $\alpha$ of $A$ such that $\gtp_\alpha:=\alpha\cap-\alpha$ is a prime ideal of $A$ and $\alpha/\gtp_\alpha$ is the positive cone of a total order of $A/\gtp_\alpha$. We denote by $\rho(\alpha)$ the real closure of the field of fractions of $A/\gtp_\alpha$. Let $\rho_\alpha:A\to\rho(\alpha)$ be the natural homomorphism and denote $a(\alpha):=\rho_\alpha(a)$ for each $a\in A$. The subsets $\Uu(a):=\{\alpha\in\Sper(A):\ a(\alpha)>0\}$ for $a\in A$, which are called \em basic open subsets\em, constitute a basis of the \emph{spectral topology} of $\Sper(A)$. A boolean combination of basic open subsets is called a \emph{constructible} set. We refer the reader to \cite[\S7.1]{bcr} for further details concerning the real spectrum of a ring $A$ and its constructible subsets. 

\begin{defn}[{\cite{s2}}]\label{def:rcr}
A commutative ring with unity $A$ is \em real closed \em if it satisfies the following conditions:
\begin{itemize}
\item[(i)] $A$ is a reduced ring,
\item[(ii)] The support map $\supp:\Sper(A)\to\Spec(A),\ \alpha\mapsto\gtp_{\alpha}=\alpha\cap(-\alpha)$ is \em identifying\em, that is, it is a homeomorphism, which induces a bijection between the constructible subsets of $\Sper(A)$ and those of $\Spec(A)$,
\item[(iii)] For each $\gtp\in\Spec(A)$ we have:
\begin{itemize}
\item[(a)] The field of fractions $R:=\qf(A/\gtp)$ is a real closed field and $A/\gtp$ is integrally closed in $R$ and 
\item[(b)] Each $\gtQ\in\Spec(A/\gtp)$ is convex with respect to the unique ordering of $A/\gtp$,
\end{itemize}
\item[(iv)] A finite sum of radical ideals of $A$ is a radical ideal of $A$.
\end{itemize}
\end{defn}
\begin{remark}
We show in Corollary \ref{a} that if $r\geq1$, rings of ${\mathcal S}^{r\diam}$-functions are not real closed rings because they not satisfy condition (iv). On the contrary all these rings satisfy conditions (i), (ii) and (iii).
\end{remark}

We will need a number of well-known facts by the experts in the theory of real closed rings. For the sake of completeness we give the proof of those which do not appear clearly stated and proved in the literature. 

\begin{fact}\label{directlimit}
If $\{(A_i,\varphi_i)\}_{i\in I}$ is a direct system of rings, then the real closure of $A:=\displaystyle\lim_{\substack{\longrightarrow}}A_i$ is $\displaystyle\lim_{\substack{\longrightarrow}}{\rm rcl}(A_i)$ where ${\rm rcl}(A_i)$ denotes the real closure of $A_i$. 
\end{fact}
\begin{proof} 
Let $\varphi_i:A_i\to A$ be the canonical homomorphism for each $i\in I$. Let $\theta:A\to{\rm rcl}(A)$ and $\theta_i:A_i\to{\rm rcl}(A_i)$ be the homomorphism provided by the real closure for each $i\in I$. By the universal property of real closure there exists a unique homomorphism $\psi_i:{\rm rcl}(A_i)\to{\rm rcl}(A)$ such that $\psi_i\circ\theta_i=\theta\circ\varphi_i$ for each $i\in I$. 
$$
\xymatrix{
A\ar[r]^\theta&{\rm rcl}(A)\\
A_i\ar[u]^{\varphi_i}\ar[r]^{\theta_i}&{\rm rcl}(A_i)\ar[u]^{\psi_i}
}
$$
In addition, $\{{\rm rcl}(A_i)\}_{i\in I}$ is a direct system of real closed rings. By the universal property of direct limits there exists a unique homomorphism $\displaystyle\lim_{\substack{\longrightarrow}}{\rm rcl}(A_i)\to{\rm rcl}(A)$. By the universal property of direct limits we also have a unique homomorphism $A\to\displaystyle\lim_{\substack{\longrightarrow }}{\rm rcl}(A_i)$. By \cite[\S I.Thm.4.8]{s4} $\displaystyle\lim_{\substack{\longrightarrow }}{\rm rcl}(A_i)$ is a real closed ring, so by the universal property of real closure we obtain a unique homomorphism ${\rm rcl}(A)\to\displaystyle\lim_{\substack{\longrightarrow }}{\rm rcl}(A_i)$. Using once more universal properties the reader can check that the previous homomorphism is an isomorphism, as required. 
\end{proof}

Recall that given an open semialgebraic set $U\subset\R^m$ a \em Nash function \em $f:U\to\R$ is a $\mathcal{C}^\infty$ semialgebraic function. If $M$ is a semialgebraic subset of $\R^m$, we denote by ${\mathcal N}(M)$ the collection of all functions $f:M\to\R$ that admit a Nash extension to an open semialgebraic neighborhood of $M$ in $\R^m$. Let ${\mathcal N}^*(M)$ be the subring of ${\mathcal N}(M)$ of bounded Nash functions and use ${\mathcal N}^\diam(M)$ to denote both rings indistinctly when a statement holds for both rings. Recall that $M$ is a \em Nash subset \em of a semialgebraic open set $U\subset \R^m$ if there exists a Nash function $f$ on $U$ such that $M=Z(f)$. We say that $M\subset \R^m$ is a \em Nash set \em if it is a Nash subset of some semialgebraic open set. By \cite[(2.12)]{fg2} $M$ is a Nash set if and only if it is a Nash subset of $\R^m\setminus(\cl(M)\setminus M)$, or equivalently, it is a Nash subset of any open semialgebraic neighborhood on which it is closed.
 
\begin{fact}\label{f:rcl2}
Let $X\subset \R^m$ be a Nash set. Then ${\mathcal S}^{0\diam}(X)$ is the real closure of ${\mathcal N}^\diam(X)$.
\end{fact}
\begin{proof}
We can assume that $X$ is a Nash subset of $\R^m$. 

Indeed, by \cite[(2.12)]{fg2} $X$ is a Nash subset of $V:=\R^m\setminus(\cl(X)\setminus X)$. By \cite[Thm.II.5.2]{sh} it holds ${\mathcal N}(X)={\mathcal N}(V)/I(X)$ where $I(X):=\{f\in{\mathcal N}(V):\ f|_X=0\}$. By \cite[Prop.2.7.5]{bcr} there exists $h\in{\mathcal S}^0(\R^m)$ such that $Z(h)=\R^m\setminus V=\cl(X)\setminus X$ and $h|_V$ is Nash. Consider the embedding $\varphi:V\to\R^{m+1},\ x\mapsto(x,\frac{1}{h(x)})$ whose image is the closed semialgebraic set $C:=\{(x,y)\in\R^{m+1}:\ yh(x)=1\}$. Note that $Z:=\varphi(X)$ is a Nash subset of $V\times \R$. As $Z$ is closed in $\R^{m+1}$, again by \cite[(2.12)]{fg2} $Z$ is also a Nash subset of $\R^{m+1}$.

The maps 
\begin{align*}
\varphi_1^*&:{\mathcal N}^{\diam}(Z)\to{\mathcal N}^{\diam}(X),\ f\mapsto f\circ\varphi|_X,\\
\varphi_2^*&:{\mathcal S}^{0\diam}(Z)\to{\mathcal S}^{0\diam}(X),\ f\mapsto f\circ\varphi|_X
\end{align*}
are isomorphisms. Hence ${\mathcal S}^{0\diam}(Z)$ is the real closure of ${\mathcal N}^{\diam}(Z)$ if and only if ${\mathcal S}^{0\diam}(X)$ is the real closure of ${\mathcal N}^{\diam}(X)$.

Next, we can assume that $X=\R^m$. 

Indeed, consider the surjective homomorphism $\varphi:{\mathcal S}^{0\diam}(\R^m)\to{\mathcal S}^{0\diam}(X)$ given by the restriction to $X$. By \cite[Thm.II.5.2]{sh} and Lemma \ref{ext} we have ${\mathcal N}^\diam(X)={\mathcal N}^\diam(\R^m)/I(X)$ and ${\mathcal S}^{0\diam}(X)={\mathcal S}^{0\diam}(\R^m)/\ker(\varphi)$ where $I(X)=\ker(\varphi)\cap{\mathcal N}^\diam(\R^m)$. By \cite[\S I.Lem.4.5]{s4} it follows that ${\mathcal S}^{0\diam}(X)$ is the real closure of ${\mathcal N}^\diam(X)$ once we have proved the statement for $X=\R^m$. 

We have $\R[\x]:=\R[\x_1,\ldots,\x_n]\subset{\mathcal N}(\R^m)\subset{\mathcal S}^0(\R^m)$, so ${\mathcal S}^0(\R^m)$ is the real closure of ${\mathcal N}(\R^m)$, because ${\mathcal S}^0(\R^m)$ is by \cite[\S III.1]{s4} the real closure of $\R[\x]$. 

It only remains to prove the ${\mathcal N}^*$ case. As ${\mathcal S}^{0*}(\R^m)$ is convex in the real closed ring ${\mathcal S}^0(\R^m)$, it is also real closed \cite[Thm. 5.12]{s6}. Thus, the real closure $A$ of the ring ${\mathcal N}^*(\R^m)$ is contained in ${\mathcal S}^*(\R^m)$. Note that $g:=\frac{1}{1+\| x\|^2}\in{\mathcal N}^*(\R^m)\subset A$ and the localization $A_g\subset{\mathcal S}^0(\R^m)$ is a real closed ring which contains ${\mathcal N}(\R^m)={\mathcal N}^*(\R^m)_g$, so $A_g={\mathcal S}^0(\R^m)$. Therefore, to prove the converse inclusion ${\mathcal S}^{0*}(\R^m)\subset A$, it is enough to show by \cite[Prop.29]{s7} the following: \em for each $k\geq1$ and $a\in A$ with $0\leq a\leq g^k$ it holds $a\in g^kA$\em. 

Indeed, as $\frac{1}{2}g^{2k}>0$, there exists $h\in{\mathcal N}^*(\R^m)$ (use Nash approximation of continuous semialgebraic functions \cite[Thm.8.8.4]{bcr}) such that $|h-(a+\frac{1}{2}g^{2k})|<\frac{1}{2}g^{2k}$, so $0<h-a<g^{2k}$. As $h<g^{2k}+a<2g^k$ (recall that by hypothesis $0\leq a\leq g^k$), we get $h/g^k\in{\mathcal N}^*(\R^m)\subset A$, so $h\in g^kA$. As $A$ is real closed and $0<h-a<g^{2k}$, there exists by \cite[Def.]{ps} $b\in A$ such that $(h-a)^2=g^{2k}b$. As $b(\alpha)\geq0$ for each $\alpha\in\Sper(A)$, we deduce by \cite[\S I.Lem.4.9]{s4} that $\sqrt{b}\in A$, so $h-a=g^k\sqrt{b}\in g^kA$. We conclude that $a\in g^kA$, as required.
\end{proof}

As $\mathcal{S}^{0\diam}(\R^n)$ is the real closure of $\mathcal{N}^\diam(\R^n)$, it is well-known that there exists a strong relation between their real spectra and their residue fields. Since this is a relevant effect, we also give here a direct proof of this relation (in the bounded case) with explicit computations.

\begin{fact}\label{Nash*sper}
The inclusion ${\tt j}:{\mathcal N}^*(\R^m)\to{\mathcal S}^*(\R^m)$ induces a homeomorphism
$$
\Sper({\tt j}):\Sper({\mathcal S}^*(\R^m))\to\Sper({\mathcal N}^*(\R^m)),\ \alpha\mapsto\alpha\cap{\mathcal N}^*(\R^m).
$$
In addition, given $\alpha\in\Sper({\mathcal S}^*(\R^m))$ denote $\beta:=\Sper({\tt j})(\alpha)$, $\gtp_\alpha:=\alpha\cap (-\alpha)$ and $\gtq_\beta:=\beta\cap (-\beta)$. Then the field of fractions $\kappa(\alpha):=\qf({\mathcal S}^*(\R^m)/\gtp_\alpha)$ is the real closure of the field of fractions $\kappa(\beta):=\qf({\mathcal N}^*(\R^m)/\gtq_\beta)$ for each $\alpha\in\Sper({\mathcal S}^*(\R^m))$.
\end{fact}
\begin{proof}
Consider the Nash function $g:=\frac{1}{1+\|x\|^2}\in{\mathcal N}^*(\R^m)$. Recall that for each $f\in{\mathcal N}(\R^m)$ there exists a non-negative integer $\ell$ such that $h:=fg^\ell\in{\mathcal N}^*(\R^m)$, see \cite[Prop.2.6.2]{bcr}. Thus, ${\mathcal N}(\R^m)={\mathcal N}^*(\R^m)_g$. Analogously, ${\mathcal S}(\R^m)={\mathcal S}^*(\R^m)_g$. Define the open constructible sets ${\mathcal U}:=\{\alpha\in\Sper({\mathcal S}^*(\R^m)):\ g(\alpha)>0\}$ and ${\mathcal V}:=\{\beta\in\Sper({\mathcal N}^*(\R^m)):\ g(\beta)>0\}$. Similarly as in \cite[Prop.3.11 \& Ch.3, Ex. 21]{am}, we have the following commutative diagram,
\begin{equation}\label{cd}
\xymatrix{
\Sper({\mathcal S}(\R^m))\ar@{<->}[r]^\cong\ar@{<->}[d]_\cong&\Sper({\mathcal S}^*(\R^m)_g)\ar@{<->}[r]^(0.825)\cong\ar@{<->}[d]_\cong&{\mathcal U}\ar@{^{(}->}[r]\ar@{<->}[d]^{\Sper({\tt j})|_{\mathcal U}}&\Sper({\mathcal S}^*(\R^m))\ar@{<->}[d]^{\Sper({\tt j})}\\
\Sper({\mathcal N}(\R^m))\ar@{<->}[r]^\cong&\Sper({\mathcal N}^*(\R^m)_g)\ar@{<->}[r]^(0.825)\cong&{\mathcal V}\ar@{^{(}->}[r]&\Sper({\mathcal N}^*(\R^m)).
}
\end{equation}

Now, pick $\alpha\in{\mathcal U}$ and let us show: \em $\kappa(\alpha)$ is the real closure of $\kappa(\beta)$, where $\beta:=\alpha\cap{\mathcal N}^{*}(\R^m)$\em. 

Indeed, the injective ring homomorphism 
$$
{\mathcal N}^*(\R^m)/\gtq_{\beta}\to{\mathcal S}^*(\R^m)/\gtp_{\alpha}
$$
induces a homomorphism $\kappa(\beta)\hookrightarrow\kappa(\alpha)$. As ${\mathcal S}^*(\R^m)$ is a real closed ring, the field $\kappa(\alpha)$ is real closed. To prove that $\kappa(\alpha)$ is the real closure of $\kappa(\beta)$ it only remains to prove: \em each homomorphism ${\tt i}:\kappa(\beta)\to R$ into a real closed field $R$ extends to $\kappa(\alpha)$\em. 

Consider the homomorphism $\psi:={\tt i}\circ\pi:{\mathcal N}^*(\R^m)\to R$ where $\pi:{\mathcal N}^*(\R^m)\to{\mathcal N}^*(\R^m)/\gtq_\beta$ is the canonical projection. As $g\notin\gtq_\beta$, the previous homomorphism extends to $\Psi:{\mathcal N}(\R^m)={\mathcal N}^*(\R^m)_g\to R$. Finally, the ring ${\mathcal S}(\R^m)$ is the real closure of ${\mathcal N}(\R^m)$, so we obtain a homomorphism $\widehat{\Phi}:{\mathcal S}(\R^m)\to R$ extending $\Phi$, which in turn induces a homomorphism $\kappa(\beta)\to R$ extending ${\tt i}$.

We prove next: \em $\Sper({\tt j})$ is a homeomorphism\em. To show that $\Sper({\tt j})$ is surjective it is enough to find a preimage of $\beta\in\Spec({\mathcal N}^*(\R^m))$ with $g(\beta)\leq 0$. As $g>0$ in $\R^m$ and $\R^m$ is dense in $\Sper({\mathcal S}^*(\R^m))$, we have $g(\beta)=0$, so $g\in\gtq_\beta$. Define
$$
\widehat{\gtq}_\beta:=\{f\in{\mathcal S}^{0*}(\R^m):\ \forall\veps\in{\mathcal S}^{0*}(\R^m),\ \veps>0\ \exists h\in\gtq_\beta\text{ such that}\ |f-h|<\veps\}.
$$
Similarly to what we have done in Lemma \ref{gorro}, the reader proves that $\widehat{\gtq}_\beta$ is a prime ideal of ${\mathcal S}^*(\R^m)$ and $\widehat{\gtq}_\beta\cap{\mathcal N}^*(\R^m)=\gtq_\beta$. We claim: \em the injective homomorphism 
$$
\psi:{\mathcal N}^*(\R^m)/\gtq_\beta\hookrightarrow{\mathcal S}^*(\R^m)/\widehat{\gtq}_\beta,\ h+\gtq_\beta\mapsto h+\widehat{\gtq}_\beta
$$
is an isomorphism\em. To that end, we prove: \em $\psi$ is surjective\em.

Pick $f\in{\mathcal S}^*(\R^m)$ and let $h\in{\mathcal N}^*(\R^m)$ be an approximation of $f$ such that $|f-h|<g$, see \cite[Thm.8.8.4]{bcr}. As $Z(g)=\varnothing$, the function
$$
a:=\frac{f-h}{g}\in{\mathcal N}^*(\R^m),
$$
so $f-h=ga\in\gtq_\beta{\mathcal S}^*(\R^m)\subset\widehat{\gtq}_\beta$ and $\psi(h+\gtq_\beta)=f+\widehat{\gtq}_\beta$, so $\psi$ is surjective and it is an isomorphism. 

Let $\alpha$ be the unique prime cone of $\Sper({\mathcal S}^*(\R^m))$ such that $\gtp_\alpha=\widehat{\gtq}_\beta$ and $\alpha/\gtp_\alpha=\psi(\beta/\gtq_\beta)$. We conclude $\Sper({\tt j})(\alpha)=\alpha\cap{\mathcal N}^{*}(\R^m)=\beta$.

Next, let us show: \em $\Sper({\tt j})$ is injective\em. 

Let $\alpha_1,\alpha_2\in\Sper({\mathcal S}^*(\R^m))$ be such that $g(\alpha_1)=g(\alpha_2)=0$ and $\alpha_1\cap{\mathcal N}^*(\R^m)=\alpha_2\cap{\mathcal N}^*(\R^m)$. Assume there exists $f\in\alpha_1\setminus\alpha_2$. Let $h\in{\mathcal N}^*(\R^m)$ be such that $|f-h|<g$. Then
$$
a:=\frac{h-f}{g}\in{\mathcal S}^*(\R^m),
$$
so $h-f=ag\in\gtp_{\alpha_i}$ for $i=1,2$. Thus, $h=f+ag\in\alpha_1\cap{\mathcal N}^*(\R^m)=\alpha_2\cap{\mathcal N}^*(\R^m)$, so $f=h-ag\in\alpha_2$, which is a contradiction. Hence, $\alpha_1=\alpha_2$.

Consequently, the continuous real spectral map (see \cite[Prop.7.1.7]{bcr})
$$
\Sper({\tt j}):\Sper({\mathcal S}^*(\R^m))\to\Sper({\mathcal N}^*(\R^m)),\ \beta\mapsto\beta\cap{\mathcal N}^*(\R^m).
$$
is bijective and it only remains to prove: \em $\Sper({\tt j})$ is an open map\em. 

We already know that $\Sper({\tt j})|_{\mathcal V}$ is open. As $\mathcal V$ is an open subset of $\Sper({\mathcal S}^*(\R^m))$, it only remains to check: \em if $\alpha_0\in\Sper({\mathcal S}^*(\R^m))\setminus{\mathcal V}=\{\alpha\in\Sper({\mathcal S}^*(\R^m)):\ g(\alpha)=0\}$, then the image under $\Sper({\tt j})$ of each open semialgebraic neighborhood of $\alpha_0$ contains an open semialgebraic neighborhood of $\Sper({\tt j})(\alpha_0)$ in $\Sper({\mathcal N}^*(\R^m))$\em. 

Let $f\in{\mathcal S}^*(\R^m)$ be such that $\alpha_0\in\Uu_0:=\{\alpha\in\Sper({\mathcal S}^*(\R^m)):\ f(\alpha)>0\}$. Let us show: \em $\Sper({\tt j})(\Uu_0)$ is an open neighborhood of $\beta_0:=\Sper({\tt j})(\alpha_0)$ in $\Sper({\mathcal N}^*(\R^m))$\em. 

Note that
$$
\alpha_0\in\Uu_1:=\{\alpha\in\Sper({\mathcal S}^*(\R^m)):\ (f-2g)(\alpha)>0\}\subset\Uu_0.
$$
By \cite[Thm.8.8.4]{bcr} there exists $h\in{\mathcal N}^*(\R^m)$ such that $|h-(f-g)|<g$, that is, $f-2g<h<f$ on $\R^m$. Define $\Uu_2:=\{\alpha\in\Sper({\mathcal S}^*(\R^m)):\ h(\alpha)>0\}$ and observe that $\alpha_0\in\Uu_1\subset\Uu_2\subset\Uu_0$. Consequently, $\beta_0\in\Sper({\tt j})(\Uu_2)=\{\beta\in\Sper({\mathcal N}^*(\R^m)):\ h(\beta)>0\}\subset\Sper({\tt j})(\Uu_0)$, as required.
\end{proof}

We would like to stress that the fact that ${\mathcal S}^{0*}(\R^m)$ is the real closure of ${\mathcal N}^*(\R^m)$ can be also deduced from Fact \ref{f:rcl2} and the following:

\begin{fact}[{\cite[Thm.5.2]{sm}}]\label{rclc}
Let $A$ be a commutative ring with unity and let $B$ be a real closed ring. Let ${\tt i}:A\to B$ be an injective homomorphism of rings such that: 
\begin{itemize}
\item[(i)] The map ${\tt i}^*:\Sper(B)\to\Sper(A),\ \beta\mapsto{\tt i}^{-1}(\beta)$ is a homeomorphism.
\item[(ii)] For each $\beta\in\Sper(B)$ the field $\qf(B/\gtq_{\beta})$ is the real closure of $\qf(A/\gtp_\alpha)$, where $\alpha:={\tt i}^*(\beta)$, $\gtp_\alpha:=\alpha\cap (-\alpha)$ and $\gtq_\beta:=\beta\cap (-\beta)$.
\end{itemize}
Then $B$ is the real closure of $A$.
\end{fact}

We present next a result that involves the central object of our work:

\begin{prop}\label{f:rcl1}
Let $r\geq1$ be an integer and let $M\subset\R^m$ be a semialgebraic set of dimension $\geq 1$. Then the ring ${\mathcal S}^{r\diam}(M)$ is not a real closed ring and the inclusion ${\mathcal S}^{r\diam}(M)\hookrightarrow{\mathcal S}^{0\diam}(M)$ provides its real closure.
\end{prop}
\begin{proof}
Pick a point $p\in M$ such that there exists an open semialgebraic neighborhood $U\subset M$ that is an (affine) Nash manifold of dimension $d:=\dim(M)$ endowed with a Nash diffeomorphism $u:U\to\R^d$ such that $u(p)=0$. Consider the ${\mathcal S}^{0*}$-function 
$$
f:\R^d\to\R,\ x\mapsto\begin{cases}
1-\|x\|&\text{if $\|x\|\leq1$},\\
0&\text{otherwise}
\end{cases}
$$ 
The composition $f\circ u^{-1}:U\to\R$ extends by $0$ to an ${\mathcal S}^{0*}$-function $F$ on $M$ that is not an ${\mathcal S}^{r*}$ function. Thus, $\mathcal{S}^{r\diam}(M)\neq \mathcal{S}^{0\diam}(M)$.

By Theorems \ref{main1} and \ref{main3} and Fact \ref{rclc} to prove the second part of the statement it is enough to show: \em${\mathcal S}^{0\diam}(M)$ is a real closed ring for each semialgebraic set $M\subset\R^m$\em. 

If $M=\mathbb{R}^m$, the statement is well-known, see \cite[\S III. 1]{s4} and \cite[Thm.10.5]{t2}. If $M$ is a closed semialgebraic subset of $\R^m$, the restriction map ${\mathcal S}^{0\diam}(\R^m)\to{\mathcal S}^{0\diam}(M)$ is an epimorphism and therefore ${\mathcal S}^{0\diam}(M)$ is a real closed ring by \cite[Thm.3.8]{s6}. If $M\subset\R^m$ is locally compact, the statement follows from Lemma \ref{ext} and the closed semialgebraic case. 

Finally, if $M\subset\R^m$ is an arbitrary semialgebraic set, it follows from Lemma \ref{LD} that
$$
{\mathcal S}^{0\diam}(M)\cong\displaystyle\lim_{\longrightarrow}({\mathcal S}^{0\diam}(E),{\tt j})
$$
where $(E, {\tt j})\in{{\mathfrak C}^{0\diam}}$. As the direct limit of real closed rings is a real closed ring \cite[\S I.Thm.4.8]{s4}, we conclude ${\mathcal S}^{0\diam}(M)$ is a real closed ring, as required. 
\end{proof}

\begin{remark}
Some particular cases of the lemma above can be deduced from the general theory of real closed rings developed in \cite{s4,t2}. 

Indeed, if $M\subset\R^m$ is a semialgebraic set, the real closure of ${\mathcal S}^r(M)$ is contained in ${\mathcal S}^0(M)$. To show that ${\rm rcl}({\mathcal S}^r(M))={\mathcal S}^0(M)$, pick $f\in{\mathcal S}^0(M)$. Let $F_1\in{\mathcal S}^0(N)$ be an extension of $f$ to an open semialgebraic neighborhood $N$ of $M$ in $\cl(M)$. By Lemma \ref{ext} we may assume $N$ is closed in $\R^m$. Thus, there exists an extension $F\in{\mathcal S}^0(\R^m)$ of $f$. As $x_i|_{M}\in{\mathcal S}^r(M)$ for each $i=1,\ldots,n$, we deduce by \cite[Lem.2.10]{t2} that $F\circ (x_1|_M,\ldots,x_n|_M)=f$ belongs to the real closure of ${\mathcal S}^r(M)$.

If $M\subset\R^m$ is a locally compact semialgebraic subset of $\R^m$, one deduces that \em ${\mathcal S}^{0*}(M)$ is the real closure of ${\mathcal S}^{r*}(M)$ \em using \cite[Prop.29]{s7} as we did in the proof of Fact \ref{f:rcl2}. However, for an arbitrary semialgebraic set $M\subset\R^m$ the only proof we know of the previous fact is the one provided above, which depends strongly on the results of this paper. 
\end{remark}

We identify next `the missing property' of the rings of ${\mathcal S}^{r\diam}$-functions to be real closed.

\begin{cor}\label{a}
Let $M\subset\R^m$ be a semialgebraic set of dimension $\geq1$ and let $r\geq1$. Then ${\mathcal S}^{r\diam}(M)$ satisfies conditions \em (i) \em to \em (iii) \em of Definition \em \ref{def:rcr}\em, but it does not satisfy condition \em (iv).
\end{cor}
\begin{proof}
Condition (i) is trivially satisfied. To prove condition (ii) consider the commutative diagram
$$
\xymatrix{
\Sper({\mathcal S}^{0\diam}(M))\ar[d]_{\supp^0}^\cong\ar[r]^{\Sper({\tt j})}_\cong&\Sper({\mathcal S}^{r\diam}(M))\ar[d]^{\supp^r}\\
\Spec^{0\diam}(M)\ar[r]^\varphi_\cong&\Spec^{r\diam}(M)
}\ 
\xymatrix{
\alpha\ar@{|->}[r]\ar@{|->}[d]&\beta=\alpha\cap{\mathcal S}^{r\diam}(M)\ar@{|->}[d]\\
\gtp_\alpha:=\alpha\cap(-\alpha)\ar@{|->}[r]&\gtq_\beta:=\beta\cap(-\beta)=\gtp_\alpha\cap{\mathcal S}^{r\diam}(M)
}
$$
where ${\tt j}:{\mathcal S}^{r\diam}(M)\hookrightarrow{\mathcal S}^{0\diam}(M)$ is the inclusion. By Theorem \ref{main1} we know that $\varphi$ is a homeomorphism \cite[Prop. 3.11 \& 3.23]{s4}, whereas $\Sper({\tt j})$ is as well a homeomorphism by Proposition \ref{f:rcl1}. As ${\mathcal S}^{0\diam}(M)$ is a real closed ring, the support map $\supp^0$ in the left is also a homeomorphism. Thus, also the support map $\supp^r$ in the right is a homeomorphism. We can also prove the latter statement in another way. As $\varphi$ and $\supp^0$ are homeomorphisms, $\supp^r$ is injective. By Theorem \ref{main3} $\supp^r$ surjective, so both $\Sper({\tt j})$ and $\supp^r$ are homeomorphisms.

We have to show in addition that $\supp^r$ is identifying. Recall that the open constructible subsets of $\Sper({\mathcal S}^{r\diam}(M))$ are exactly the open quasicompact ones and similarly for $\Sper^{r\diam}(M)$. Therefore, the homeomorphism $\varphi$ induces a bijection between the open constructible subsets, and consequently between the constructible ones. However, this can be proved directly.

Take $f\in{\mathcal S}^{r\diam}(M)$ and consider the constructible subset $\Uu(f):=\{\alpha\in\Sper({\mathcal S}^{r\diam}(M)):\ f(\alpha)>0\}$ of $\Sper({\mathcal S}^{r\diam}(M))$. By Remark \ref{potabs} (ii) $|f|^{2r+1}\in{\mathcal S}^{r\diam}(M)$. As $\Uu(f)=\Uu(f^{2r+1})$, we deduce $\supp^r(\Uu(f))=\Dd(f^{2r+1}+|f|^{2r+1})$ is a constructible subset of $\Spec^{r\diam}(M)$.

Condition (iii.a) follows from Theorem \ref{main3}. To prove condition (iii.b) let $\gtp\subset\gtq$ be prime ideals of ${\mathcal S}^r(M)$ and let $f,g\in{\mathcal S}^r(M)$ be elements whose classes modulo $\gtp$ are positive with respect to the unique ordering of ${\mathcal S}^r(M)/\gtp$ such that $f+g\in\gtq$. Then $f+g\in\varphi^{-1}(\gtq)$ and the classes of $f$ and $g$ modulo $\varphi^{-1}(\gtp)$ are positive with respect to the unique ordering of ${\mathcal S}^0(M)/\varphi^{-1}(\gtp)$. As ${\mathcal S}^0(M)$ is a real closed ring, $\varphi^{-1}(\gtq)$ is convex with respect to the unique ordering of ${\mathcal S}^0(M)/\varphi^{-1}(\gtp)$, so $f,g\in\varphi^{-1}(\gtq)\cap{\mathcal S}^r(M)=\gtq$ and $\gtq/\gtp$ is convex with respect to the unique ordering of ${\mathcal S}^r(M)/\gtp$.

Let us prove next: \em ${\mathcal S}^{r\diam}(M)$ does not satisfy condition \em (iv). Let us check first: \em it is enough to analyze the $1$-dimensional case\em. 

Let $C\subset M$ be a $1$-dimensional compact semialgebraic set and consider the surjective restriction map $\rho:\mathcal{S}^r(M)\rightarrow \mathcal{S}^r(C)$. By the $1$-dimensional case there exist $\gtp_1,\gtp_2\in\Spec^r(C)$ and $f\in \mathcal{S}^r(C)$ such that $f\in \sqrt{\gtp_1+\gtp_2}\setminus \gtp_1+\gtp_2$. Define $\gtq_i:=\rho^{-1}(\gtp_i)$ for $i=1,2$ and let us show: $\sqrt{\gtq_1+\gtq_2}\neq\gtq_1+\gtq_2$. 

Let $k\geq2$ and $f_i\in\gtp_i$ for $i=1,2$ be such that $f^k=f_1+f_2$. Pick $F,F_1\in \mathcal{S}^r(M)$ satisfying $F|_C=f$ and $F_1|_C=f_1$ and define $F_2:=F^k-F_1$. It holds $F_i\in\gtq_i$ for $i=1,2$, so $F\in \sqrt{\gtq_1+\gtq_2}$. If $F\in\gtq_1+\gtq_2$, there exist $G_i\in\gtq_i$ such that $F=G_1+G_2$. Consequently, $f=g_1+g_2$ where $g_i:=G_i|_C\in\gtp_i$ for $i=1,2$, which is a contradiction. 

Thus, we are only left to prove: \em ${\mathcal S}^{r\diam}(M)$ does not satisfy condition (iv) if $M$ is a $1$-dimensional semialgebraic set\em.

Let $U\subset M$ be an open semialgebraic subset of $M$ for which there exist a Nash diffeomorphism $\Psi:U\to\R$ and let $p\in U$ be such that $\Psi(p)=0$. Let $\gtp_k$ be the prime ideal of all ${\mathcal S}^r$-functions on $M$ that vanishes identically on the germ $\Psi^{-1}(\{(-1)^kt>0\}_0)$ for $k=1,2$. We claim: \em $\gtp_1+\gtp_2$ is not a radical ideal\em. To prove this pick $f\in{\mathcal S}^r(M)$ such that $f\circ\Psi^{-1}|_{(-1,1)}=t$ and $Z(f)=\{p\}$ (an appropriate finite $\mathcal{S}^r$-partition of unity could be useful to find $f$). Let us check: $f\in\sqrt{\gtp_1+\gtp_2}\setminus(\gtp_1+\gtp_2)$. 

Let $f_1\in{\mathcal S}^r(M)$ be such that 
$$
f_1\circ\Psi^{-1}|_{(-1,1)}(t)=\begin{cases}
t^{2r}&\text{if $t\in[0,1)$},\\
0&\text{if $t\in(-1,0]$}
\end{cases}
$$
and $Z(f_1)\subset\Psi^{-1}([-1,1])$. Let $f_2\in{\mathcal S}^r(M)$ be such that $f_2\circ\Psi^{-1}|_{(-1,1)}(t)=f_1\circ\Psi^{-1}|_{(-1,1)}(-t)$. Then $f_1+f_2\in{\mathcal S}^r(M)$ satisfies $(f_1+f_2)\circ\Psi^{-1}|_{(-1,1)}=t^{2r}$. As $Z(f)=\{p\}$ there exists a unit $u\in{\mathcal S}^r(M)$ such that $u\circ\Psi^{-1}|_{(-1,1)}=1$ and $f^{2r}=(f_1+f_2)u\in\gtp_1+\gtp_2$.

Suppose $f\in\gtp_1+\gtp_2$ and let $g_i\in\gtp_i$ for $i=1,2$ be such that $f=g_1+g_2$. We have
$$
t=f\circ\Psi^{-1}|_{(-1,1)}=g_1\circ\Psi^{-1}|_{(-1,1)}+g_2\circ\Psi^{-1}|_{(-1,1)}.
$$
Let $\veps>0$ be such that $g_1\circ\Psi^{-1}$ is identically zero on $(-\veps,0)$ and $g_2\circ\Psi^{-1}$ is identically zero on $(0,\veps)$. In particular,
$$
g_1\circ\Psi^{-1}|_{(-\veps,\veps)}=\begin{cases}
0&\text{if $t\in (-\veps,0)$},\\
t&\text{if $t\in(0,\veps)$}
\end{cases}
$$
is an $\mathcal{S}^r$-function, which is a contradiction.
\end{proof}

\subsection{Nash functions versus ${\mathcal S}^{\infty}$ functions} 
A first natural question that arises when dealing with the ring ${\mathcal S}^{\infty}(M):=\bigcap_{r\geq1}{\mathcal S}^r(M)$ of ${\mathcal S}^\infty$-functions on a semialgebraic set $M\subset\R^m$ is whether it coincides with the ring ${\mathcal N}(M)$ of Nash functions on $M$. This is false in general even if $M$ is compact or Nash as the following examples show.

\begin{examples}\label{cex:nash}
(i) Let $M\subset\R^2$ be the compact semialgebraic set 
$$
([-2,-1]\times[-1,1])\cup([1,2]\times[-1,1])\cup\{\y=0,-1\leq\x\leq 1\}
$$ 
and define $f:M\to\R,\ (x,y)\mapsto y\sqrt{x^2+y^2}$. As $f$ is the restriction to $M$ of a Nash function on $\R^2\setminus\{(0,0)\}$ and $f|_{M\cap B((0,0),1/2)}\equiv0$, we deduce $f$ is an ${\mathcal S}^{\infty}$-function. By the identity principle $f$ does not admit a Nash extension to an open semialgebraic neighborhood of $M$ in $\R^2$.

(ii) Let $X:=\{z(x^2+y^2)-x^3=0\}\subset\R^3$ be Cartan's umbrella and define
$$
f:X\to\R,\ (x,y,z)\mapsto\begin{cases}
\frac{(z-1)^2}{(z-1)^2+x^2+y^2}&\text{if $x^2+y^2\neq0$,}\\
1&\text{otherwise,}
\end{cases}
$$ 
which is not a Nash function on $M$ \cite[\S 3]{e}. However, $f$ has Nash local extensions to both $\mathbb{R}^{3}\setminus \{(0,0,1)\}$ and $\{z>\frac{1}{2}\}$. Thus, $f$ admits semialgebraic jets on $M$ of order $r$ for each $r\geq 1$.
\end{examples}

However, as we show next, ${\mathcal S}^\infty$-functions are locally Nash functions:

\begin{lem}\label{local}
Let $M\subset\R^m$ be a semialgebraic set and let $f\in{\mathcal S}^{\infty\diam}(M)$. For each $x\in M$ there exists an open semialgebraic neighborhood $V^x\subset\R^m$ of $x$ and a Nash function $F_x\in{\mathcal N}^{\diam}(V^x)$ such that $F_x|_{M\cap V^x}=f|_{M\cap V^x}$.
\end{lem}
\begin{proof}
Let $f\in{\mathcal S}^{\infty}(M)$ and assume for simplicity $x=0\in M$. For each $r\geq0$ let $F^r:=(f^r_\alpha)_{|\alpha|\leq r}$ be a semialgebraic jet of order $r$ associated to $f$. We may assume: \em $f^ r_\alpha(0)=f_\alpha^k(0)$ for each pair of integers $r\leq k$ and each $\alpha\in \N^n$ with $|\alpha|\leq r$\em. Thus, the definition $f_\alpha(0):=f_{\alpha}^r(0)$ if $|\alpha|=r$ is consistent.

To that end, recall first that $f_0^r=f$ for each $r\geq0$. Next, replace recursively the semialgebraic jet $F_{r+1}$ by $\widetilde{F}_{r+1}:=(\widetilde{f}^{r+1}_\alpha)$ where $\widetilde{f}_\alpha^{r+1}(x):=f_\alpha^{r+1}(x)-f^{r+1}_\alpha(0)+\tilde{f}^r_\alpha(0)$ whenever $|\alpha|\leq r$. Using property \ref{diff2} one shows that $\widetilde{F}_r$ is a semialgebraic jet associated to $f$ as an ${\mathcal S}^r$-function and $\tilde{f}_\alpha^k(0)=\tilde{f}_\alpha^r(0)$ if $k\geq r$ and $|\alpha|\leq r$. 

Consider the formal series 
$$
h_0:=\sum_\alpha\frac{1}{\alpha!}f_\alpha(0)\x^\alpha\in\R[[\x]]
$$ 
and let $M_0$ be the germ at the origin of $M$. For simplicity we identify ${\mathcal N}(\R_0^m)$ with $\R[[\x]]_{\rm alg}$. It is enough to prove: \em $f|_{M_0}$ is the germ of a Nash function on $M_0$\em.

Let $\Delta_1,\ldots,\Delta_s$ be a Nash stratification of $M_0$ such that each restriction $f|_{\Delta_i}$ is a Nash germ. Let $Z_i$ be the Nash closure of the germ $\Delta_i$ and let $X_i$ be the Nash closure of $\Gamma(f|_{\Delta_i})$. It holds $\dim(X_i)=\dim(\Gamma_i)=\dim(\Delta_i)=\dim(Z_i)$ and $X_i$ is an irreducible Nash germ (because it is the Nash closure of the graph of a Nash function germ on the germ at the origin of a Nash manifold).
 
To ease notation write for a while $\Delta:=\Delta_i$, $Z:=Z_i$, $\Gamma:=\Gamma_i$ and $X:=X_i$. Let $g(\x,\y)\in\R[[\x,\y]]_{\rm alg}$ be such that $X=Z(g)$. We claim: $g(\x,h_0(\x))\in I(Z)\R[[\x]]$.

Let ${\mathfrak G}$ be the collection of the germs $\gamma_0$ at the origin of (continuous) semialgebraic curves $\gamma:[0,1)\to\Delta\subset\R^m$ such that $\gamma(0)=0$. We identify each germ $\gamma_0$ with a Puiseux tuple $\R[[\t^*]]^m$. For each $\gamma\in{\mathfrak G}$ define the homomorphism
$$
\gamma^*:\R[[\x]]\to\R[[\t^*]],\ \zeta\mapsto\zeta\circ\gamma.
$$
Using a Nash stratification of $M$ one proves $M_0=\bigcup_{\gamma\in {\mathfrak G}}\text{Im}(\gamma)_0$, so 
$$
I(Z)=\bigcap_{\gamma\in{\mathfrak G}}\ker(\gamma^*)\cap\R[[\x]]_{\rm alg}.
$$
The completion of the local noetherian ring $\R[[\x]]_{\rm alg}$ is $\R[[\x]]$. We have
$$
I(Z)\R[[\x]]=\Big(\bigcap_{\gamma\in{\mathfrak G}}\ker(\gamma^*)\cap\R[[\x]]_{\rm alg}\Big)\R[[\x]]=\bigcap_{\gamma\in{\mathfrak G}}(\ker(\gamma^*)\cap\R[[\x]]_{\rm alg})\R[[\x]]=\bigcap_{\gamma\in{\mathfrak G}}\ker(\gamma^*).
$$ 
Assume $\xi:=g(\x,h_0(\x))\not\in I(Z)\R[[\x]]$. Then there exists $\gamma\in{\mathfrak G}\subset\R[[\t^*]]^m_{\rm alg}$ such that $\xi\notin\ker(\gamma^*)$.
After reparameterizing the variable $\t$ we may assume $\gamma\in\R[[\t]]^m_{\rm alg}$. Let $k\geq1$ be the order of the series $\xi(\gamma)$. Write $h_0=a+b$ where $a:=\sum_{|\alpha|\leq k}\frac{1}{\alpha!}f_\alpha(0)\x^\alpha\in\R[\x]$ has degree $\leq k$ and $b$ has degree $>k$. There exists a power series $s\in\R[[\x,\z_1,\z_2]]$ such that 
\begin{equation}\label{s}
g(\x,\z_1+\z_2)=g(\x,\z_1)+\z_2s(\x,\z_1,\z_2). 
\end{equation}
Consequently,
$$
g(\x,h_0)=g(\x,a+b)=g(\x,a)+bs(\x,a,b).
$$
Thus, 
$$
\xi(\gamma)=g(\gamma,a(\gamma))+bs(\gamma,a(\gamma),b(\gamma)).
$$
As $\omega(\xi(\gamma))=k$ and $\omega(bs(\x,a,b))\geq k+1$, we deduce $\omega(g(\gamma,a(\gamma)))=k$. Let us analyze next the Puiseux series $f(\gamma)\in\R[[\t^*]]$. We have
$$
|f(x)-a(x)|=\Big|f(x)-\sum_{|\alpha|\leq k}\frac{1}{\alpha!}f_{\alpha}(0)x^\alpha\Big|=o(\|x\|^k)
$$
for $x\in M$ when $x\to 0$. Thus,
$$
\lim_{t\to 0^+}\Big(\frac{f(\gamma(t))-a(\gamma(t))}{t^k}\Big)=0,
$$
so $f(\gamma)-a(\gamma)\in\R[[\t^*]]$ has order strictly greater than $k$. By \eqref{s} we deduce
$$
\omega(g(\gamma,f(\gamma))-g(\gamma,a(\gamma)))>k,
$$ 
so $\omega(g((\gamma),f(\gamma)))=k$, which is a contradiction because $g(\gamma(t),f(\gamma(t)))=0$ for each $t\in [0,1)$. Consequently, $g(\x,h_0(\x))\in I(Z)\R[[\x]]$, as claimed.

Let $f_{i1},\ldots,f_{ip}\in\R[[\x]]_{\rm alg}$ be a system of generators of $I(Z_i)$ and let $g_i\in\R[[\x,\y]]_{\rm alg}$ be such that $Z(g_i)=X_i$. We know that $g_i(\x,h_0(\x))\in I(Z_i)\R[[x]]$. Thus, there exist $a_{i1},\ldots,a_{ip}\in\R[[\x]]$ such that $g_i(\x,h_0)=a_{i1}f_{i1}+\cdots+a_{ip}f_{im}$. By Artin's approximation theorem \cite[Thm.8.3.1]{bcr} there exist Nash functions $h,\widetilde{a}_{i1},\ldots,\widetilde{a}_{ip}\in\R[[\x]]_{\rm alg}$ such that $g_i(\x,h)=\widetilde{a}_{i1}f_{i1}+\cdots+\widetilde{a}_{iq}f_{iq}$ for $i=1,\ldots,s$. Consequently, $\Gamma(h|_{Z_i})\subset X_i$ for each $i=1,\ldots,r$. As $Z_i$ is an irreducible Nash set and $h$ is a Nash function germ, $\Gamma(h|_{Z_i})$ is an irreducible Nash set of dimension $\dim(Z_i)=\dim(X_i)$. Thus, $X_i=\Gamma(h|_{Z_i})$ for $i=1,\ldots,r$, so $h|_{\Delta_i}=f|_{\Delta_i}$ for $i=1,\ldots,r$. Hence, $h|_{M_0}=f|_{M_0}$, as required. 
\end{proof}

The following is a direct consequence of Lemma \ref{local} and Serre's coherence condition (see \cite[\S 2.B]{bfr}).

\begin{cor}\label{cohNash}
Let $U\subset\R^m$ be an open semialgebraic set and let $X\subset U$ be a coherent Nash subset of $U$. Then $\mathcal{N}^\diam(X)=\mathcal{S}^{\infty\diam}(X)$. 
\end{cor}

We are now ready to prove Theorem \ref{nash} and Proposition \ref{nash2}, but beginning with the latter.

\begin{proof}[Proof of Proposition \em \ref{nash2}]
As $X:=\cl(M)\cap U$ is closed in $U$ and it is a Nash set, then it is by \cite[(2.12)]{fg2} a Nash subset of $U$. Thus, if $V\subset U$ is an open semialgebraic neighborhood of $M$, then $X\cap V$ is a Nash subset of $V$. By Lemma \ref{f:rcl2} we know that the real closure of ${\mathcal N}^{\diam}(X\cap V)$ is ${\mathcal S}^{0\diam}(X\cap V)$.

We claim: \em ${\mathcal S}^{0\diam}(M)$ is the real closure of ${\mathcal N}^{\diam}(M)$\em. 

Let ${\mathcal V}$ be the collection of open semialgebraic neighborhoods of $M$ in $U$. For each $V\in{\mathcal V}$ the restriction map ${\mathcal N}^{\diam}(X\cap V)\to{\mathcal N}^{\diam}(M)$ is injective because $\cl(M)\cap V=X\cap V$. For each $f\in{\mathcal N}(M)$ there exists by definition $V\in{\mathcal V}$ and a Nash function $F\in{\mathcal N}(V)$ such that $F|_V=f$. Thus, ${\mathcal N}^{\diam}(M)=\displaystyle\lim_{\substack{\longrightarrow\\V\in{\mathcal V}}}{\mathcal N}^{\diam}(X\cap V)$. As we have proved above, the real closure of ${\mathcal N}^{\diam}(X\cap V)$ is ${\mathcal S}^{0\diam}(X\cap V)$. By Lemma \ref{directlimit} the real closure of ${\mathcal N}^{\diam}(M)$ is $\displaystyle\lim_{\substack{\longrightarrow\\V\in{\mathcal V}}}{\mathcal S}^{0\diam}(X\cap V)$. Again the map ${\mathcal S}^{0\diam}(X\cap V)\to{\mathcal S}^{0\diam}(M)$ is injective for each $V\in{\mathcal V}$ because $\cl(M)\cap V=X\cap V$. By the definition of ${\mathcal S}^{0\diam}$-function it follows that ${\mathcal S}^0(M)=\displaystyle\lim_{\substack{\longrightarrow\\V\in{\mathcal V}}}{\mathcal S}^0(X\cap V)$. Consequently, ${\mathcal S}^{0\diam}(M)$ is the real closure of ${\mathcal N}^{\diam}(M)$.

Finally, ${\mathcal N}^{\diam}(M)\subset{\mathcal S}^{\infty\diam}(M)\subset{\mathcal S}^{0\diam}(M)$, so ${\mathcal S}^{0\diam}(M)$ is also the real closure of ${\mathcal S}^{\infty\diam}(M)$, as required.
\end{proof}

\begin{proof}[Proof of Theorem \em \ref{nash}]
Assume there exists $x\in M$ such that the germ $\cl(M)_x$ is not a germ of a Nash set. Let $\ol{M}^{\rm an}_x$ be the smallest Nash set germ that contains $M_x$. By the curve selection lemma there exists a Nash arc $\gamma:[0,1]\to\R^m$ such that $\gamma(0)=x$ and $\gamma((0,1])_x\subset\ol{M}^{\rm an}_x\setminus\cl(M)_x$. Let $f\in{\mathcal S}^{\infty\diam}(M)$. By Lemma \ref{local} there exist an open semialgebraic neighborhood $V^x\subset\R^m$ of $x$ and a Nash extension $F_x$ of $f|_{V^x\cap M}$ to $V^x$. Note that $F_x|_{\ol{M}^{\rm an}_x}$ is completely determined by $f$. Thus, we get a well-defined homomorphism 
$$
\varphi:{\mathcal S}^{\infty\diam}(M)\to{\mathcal S}^{\infty\diam}(\im(\gamma)_x),\ f\mapsto F_x|_{\im(\gamma)_x}.
$$
Pick an open semialgebraic neighborhood $U$ of $\cl(M)$ and let us consider the restriction homomorphism $\psi:{\mathcal S}^{\infty\diam}(U)\to{\mathcal S}^{\infty\diam}(M),\ f\mapsto f|_M$. By Proposition \ref{nash2} the real closures of ${\mathcal S}^{\infty\diam}(U)$ and ${\mathcal S}^{\infty\diam}(\im(\gamma)_x)$ are ${\mathcal S}^{0\diam}(U)$ and ${\mathcal S}^{0\diam}(\im(\gamma)_x)$. By the universal property of real closure we have the commutative diagram:
$$
\xymatrix{
{\mathcal S}^{0\diam}(U)\ar[r]^{\ol{\psi}}\ar[rd]&{\mathcal S}^{0\diam}(M)\ar[d]^{\ol{\varphi}}\\
&{\mathcal S}^{0\diam}(\im(\gamma)_x)}
$$ 
where $\ol{\varphi}$ and $\ol{\psi}$ are the restriction homomorphisms. Pick $f\in{\mathcal S}^{0\diam}(U)$ whose zero set is $\cl(M)$. Thus, $f|_{\im(\gamma)_x}=(\ol{\varphi}\circ\ol{\psi})(f)=\ol{\varphi}(f|_M)=\ol{\varphi}(0)=0$, so $\im(\gamma)_x\subset Z(f)_x=\cl(M)_x$, which is a contradiction.
\end{proof}

\bibliographystyle{amsalpha}

\end{document}